\documentclass[11pt]{article}

\usepackage{amsmath,amsthm,amssymb,enumerate,dsfont}
\usepackage{algorithmicx,algorithm,algpseudocode}
\usepackage{float}
\usepackage{titlesec}
\usepackage{scalerel}
\usepackage{mathrsfs}
\usepackage{epstopdf}
\usepackage{graphicx}
\usepackage{xspace}
\usepackage{color}
\usepackage{hyperref}
\usepackage[title]{appendix}
\usepackage{geometry}
\usepackage[numbers,sort&compress]{natbib}
\usepackage{caption}
\usepackage[labelformat=simple]{subcaption}
\usepackage{verbatim,url}

\usepackage{multirow}
\usepackage{makecell}

\newcommand{\MATLAB}{\textsc{Matlab}\xspace}
\newcommand{\sync}{\textsf{sync}\xspace}
\newcommand{\G}{\mathcal{G}}
\usepackage[colorinlistoftodos,prependcaption,textsize=tiny]{todonotes}
\DeclareMathOperator{\vect}{vec}

\newcommand{\RR}{\mathbb{R}}

\newcommand{\nm}[1]{\left\|{#1}\right\|}

\newcommand{\ra}{\rightarrow}

\newcommand{\eg}{{\it e.g.}}
\newcommand{\ie}{{\it i.e.}}

\newcommand{\viz}{{\it viz.}}

\DeclareMathOperator*{\argmin}{argmin}
\DeclareMathOperator*{\argmax}{argmax}

\DeclareMathOperator{\Tr}{Tr}

\DeclareMathOperator{\Diag}{Diag}

\newcommand{\BlkDiag}{\mbox{BlkDiag}}

\newtheorem{lem}{Lemma}
\newtheorem{thm}{Theorem}

\newtheorem{defi}{Definition}
\newtheorem{cond}{Condition}
\newtheorem{prop}{Proposition}

\newtheorem{conj}{Conjecture}

\allowdisplaybreaks

\begin{document}
\title{\textbf{A Unified Approach to Synchronization Problems over Subgroups of the Orthogonal Group}}
\author{Huikang Liu\thanks{Shanghai University of Finance and Economics.  E--mail: {\tt liuhuikang@sufe.edu.cn}} \and  Man-Chung Yue\thanks{The University of Hong Kong.  E--mail: {\tt mcyue@hku.hk}} \and Anthony Man-Cho So\thanks{The Chinese University of Hong Kong.  E--mail: {\tt manchoso@se.cuhk.edu.hk}}}
\date{}
\maketitle
\begin{abstract}
The problem of synchronization over a group $\G$ aims to estimate a collection of group elements $G^*_1, \dots, G^*_n \in \G$ based on noisy observations of a subset of all pairwise ratios of the form $G^*_i {G^*_j}^{-1}$. Such a problem has gained much attention recently and finds many applications across a wide range of scientific and engineering areas. In this paper, we consider the class of synchronization problems in which the group is a closed subgroup of the orthogonal group. This class covers many group synchronization problems that arise in practice. Our contribution is fivefold. First, we propose a unified approach for solving this class of group synchronization problems, which consists of a suitable initialization step and an iterative refinement step based on the generalized power method, and show that it enjoys a strong theoretical guarantee on the estimation error under certain assumptions on the group, measurement graph, noise, and initialization. Second, we formulate two geometric conditions that are required by our approach and show that they hold for various practically relevant subgroups of the orthogonal group. The conditions are closely related to the error-bound geometry of the subgroup --- an important notion in optimization. Third, we verify the assumptions on the measurement graph and noise for standard random graph and random matrix models. Fourth, based on the classic notion of metric entropy, we develop and analyze a novel spectral-type estimator. Finally, we show via extensive numerical experiments that our proposed non-convex approach outperforms existing approaches in terms of computational speed, scalability,  and/or estimation error.
\end{abstract}

\bigskip


\bigskip

\section{Introduction}

\label{sec:intro}
In many real-world estimation problems, the signals of interest, which are commonly referred to as \emph{ground truths}, are constrained to lie in a \emph{group}.\footnote{Recall that a \emph{group} is a pair $(\mathcal{G}, \ast)$, where $\mathcal{G}$ is a set and $\ast$ is a binary operation on $\mathcal{G}$, such that (i) $\ast$ is associative (i.e., $g_1\ast(g_2\ast g_3) = (g_1\ast g_2)\ast g_3$ for all $g_1,g_2,g_3\in\mathcal{G}$); (ii) there exists an identity element ${\rm id} \in \mathcal{G}$ (i.e., $g \ast {\rm id} = {\rm id} \ast g = g$ for all $g\in\mathcal{G}$); (iii) for each $g\in\mathcal{G}$, there exists an inverse element $g^{-1}$ of $g$ (i.e., $g\ast g^{-1} = g^{-1} \ast g = {\rm id}$). For simplicity, we shall write $g_1g_2$ for $g_1\ast g_2$, where $g_1,g_2 \in \G$. Also, we shall abuse terminology and refer to $\G$ as the group.}
One such example is the \emph{group synchronization problem} (or simply \emph{synchronization problem}), in which the ground truth takes the form $G^*=(G^*_1,\dots,G^*_n)$ with $G^*_1,\ldots,G^*_n$ being elements of a group $\G$ and is to be estimated from noisy measurements of a subset of all pairwise ratios of the form $G^*_i {G^*_j}^{-1}$ that are computed using the binary operation on $\G$.
Synchronization problems have found applications across a wide range of areas, such as social science, distributed network, signal processing, computer vision, robotics, structural biology, computational genomics, and machine learning, and hence have gained much attention over the past decade.
Indeed, group synchronization problems often appear as sub-tasks in many important problems from the above areas, 
%
including community detection~\cite{cucuringu2015synchronization,abbe2016exact} where $\G$ is the Boolean group; ranking~\cite{cucuringu2016sync} and distributed clock synchronization~\cite{giridhar2006distributed} where $\G$ is the group of 2D rotations; sensor network localization and cryo-electron microscopy where $\G$ is the group of orthogonal matrices or special orthogonal matrices~\cite{cucuringu2012eigenvector,shkolnisky2012viewing}; the pose graph estimation problem~\cite{rosen2016se} where $\G$ is the group of Euclidean motions; the haplotype phasing problem~\cite{si2014haplotype,chen2018projected} where $\G$ is the cyclic group (integers with modulo arithmetics); the multi-graph matching problem~\cite{pachauri2013solving, huang2013consistent} where $\mathcal{G}$ is the group of permutations.
In most applications, after solving the synchronization problem, the estimated group elements will in turn be used to estimate another underlying signal, which is the ultimate target. In principle, one could estimate the underlying signal directly or jointly with the group elements. However, as pointed out in~\cite{bandeira2020non}, estimating the signal given the group elements is a considerably easier task and can be addressed using well-developed techniques for inverse problems. This motivates us to focus on the synchronization problem.

Numerical approaches to synchronization problems are roughly divided into three categories: Spectral-type estimators, semidefinite programming (SDP) relaxations, and non-convex approaches. In the context of synchronization problems, a spectral-type estimator was first introduced in \cite{singer2011angular} for phase synchronization (\ie, $\G $ is the group $\mathcal{SO}(2)$ of 2D rotations).
It has later been generalized to synchronization problems over other subgroups of the orthogonal group (see~\cite{bandeira2013cheeger, pachauri2013solving, arie2012global, ling2020near}) and even general compact groups~\cite{romanov2019noise}.
The main cost of computing a spectral-type estimator comes in two parts. First, the eigenvectors corresponding to the first few eigenvalues of the \emph{graph connection Laplacian}~\cite{bandeira2013cheeger} or a data matrix defined using the noisy observations are computed (see Sections~\ref{sec:problem} and \ref{sec:spectral_ini} for details). Second, a certain rounding procedure is invoked to ensure that the returned estimator lies in the feasible set $\G^n$. The major advantages of spectral-type estimators are their low computational cost and ease of implementation.

The SDP relaxation approach is mainly based on convex relaxations of either the least squares formulation or the least unsquared deviation formulation of the synchronization problem, both of which are non-convex. These convex relaxations are SDP problems (hence the name of the approach) that are obtained via the by-now classic \emph{lifting} technique (see, \eg, \cite{LMS+10}) and are shown to be tight --- the optimal solutions to the relaxed problem are also optimal for its non-relaxed counterpart --- under certain assumptions on the observation and noise models~\cite{wang2013exact,javanmard2016phase,rosen2016se,hajek2016achieving,bandeira2017tightness,bandeira2018random}. 
Recently, using results from harmonic analysis and representation theory, an SDP relaxation for general compact groups and general objective functions (as opposed to just the squared or unsquared deviation) has been derived and studied in~\cite{bandeira2020non}.
A major drawback of the SDP relaxation approach is that it does not scale well with the number $n$ of group elements to be estimated. Indeed, the approach requires solving an SDP problem with an $M\times M$ matrix decision variable, where $M$ is linear in $n$. As such, the computational cost is on the order of $n^{3.5}$~\cite[Section 6.6.3]{BN01}, which can be prohibitive when $n$ is large. One remedy of this is to devise specialized algorithms for solving the large SDP problem. For example, an alternating direction augmented Lagrangian method was developed in \cite{wang2013exact} to solve an SDP relaxation of the least unsquared deviation formulation of the synchronization problem over the special orthogonal group. However, the computational cost is still not satisfactory when $n$ is beyond a few hundreds. Moreover, the performance of the method is rather sensitive to how various parameters (\eg, step size, penalty parameter) are tuned. Another possible remedy is to solve the large SDP problem using the Burer-Monteiro heuristic~\cite{burer2003nonlinear, burer2005local}, which amounts to replacing the large semidefinite matrix variable $Z \in \mathbb{R}^{M\times M}$ by some factorization $YY^\top$ with $Y\in \mathbb{R}^{M\times f}$ and solving the resulting optimization problem with the smaller matrix variable $Y$. However, the Burer-Monteiro heuristic, being a non-convex approach, is prone to sub-optimality due to the presence of local minima, unless the factorization dimension $f$ is sufficiently large~\cite{bandeira2016low, boumal2016non, boumal2020deterministic, ling2020solving}.

In recent years, there have been attempts to solve the non-convex least squares formulation of the synchronization problem directly without relaxing it. Such a non-convex approach typically has two stages. In the first stage, a carefully designed initialization procedure is used to produce a point that is close enough to the ground truth. Then, in the second stage, an iterative procedure is used to refine the initial point. One procedure that is particularly suitable for the second stage is the \emph{generalized power method} (GPM)~\cite{journee2010generalized,luss2013conditional,liu2017discrete}. The idea of applying GPM to synchronization problems was first introduced in \cite{boumal2016nonconvex} for phase synchronization and later further developed in \cite{chen2018projected} for the joint alignment problem and in~\cite{wang2021nonconvex} for the community detection problem. In addition, the analysis of GPM in~\cite{boumal2016nonconvex} for phase synchronization was sharpened in~\cite{liu2016estimation,zhong2017near}. The advantage of the GPM-based non-convex approach is that it is much faster and more scalable than the SDP relaxation approach, as each iteration of GPM involves only matrix-vector multiplications and projections onto the group $\G$. As we shall see later, these projections can be computed efficiently for many concrete groups $\G$ of practical relevance. Also, the GPM-based approach is easy to implement and requires no parameter tuning. It is worth noting that although spectral-type estimators have been shown to be already qualitatively optimal in certain settings~\cite{boumal2016nonconvex, ling2020near}, we have observed in our experiments that by refining a spectral-type estimator using GPM, the quality of the resulting estimator can be substantially better (see Section~\ref{sec:exp}). In other words, we can greatly improve the performance of spectral-type estimators by paying a small amount of extra computational cost.

Interestingly, the procedures used to compute estimators of the ground truth can be viewed as \emph{approximation algorithms} for the non-convex formulation of the synchronization problem at hand. For instance, it is known that the least squares formulation of synchronization over the Boolean group is equivalent to the MAX-CUT problem~\cite{wang2013exact}, while that of phase synchronization is equivalent to a complex quadratic maximization problem with unit modulus constraint~\cite{bandeira2017tightness}. As such, the \emph{approximation accuracy} of various least squares estimators (\ie, the gap between the objective value attained by the estimator in question and the optimal value of the formulation) can be determined, see, \eg,~\cite{goemans1995improved,so2007approximating,trevisan2012max}. However, since an optimal solution to the non-convex formulation is in general different from the ground truth, a more relevant and faithful measure of the quality of an estimator for a synchronization problem is the \emph{estimation error}, which is defined as the deviation of the estimator from the ground truth. For synchronization over the special orthogonal group, the estimation error of various estimators has been studied in~\cite{wang2013exact,bandeira2017tightness,boumal2016nonconvex,liu2016estimation}.

It should be mentioned that when $\mathcal{G} $ is the group of rotations (\ie, special orthogonal matrices), the synchronization problem is equivalent to the problem of \emph{multiple rotation averaging}~\cite{hartley2013rotation}, and the GPM is also known as the Jacobi-type method~\cite[Section~7.4]{hartley2013rotation}. Moreover, various specialized algorithms are developed in the literature for multiple rotation averaging, such as the Weiszfeld algorithm. However, it is often difficult to extend the algorithm and/or its theory to a general subgroup of the orthogonal group. Since our paper focuses on algorithms for general subgroups, we do not go into the details of these specialized algorithms. We refer the interested reader to~\cite{hartley2013rotation} and the references therein.

Finally, we mention two relatively new yet effective approaches to general group synchronization problems, both of which fall into the category of \emph{message passing} algorithms. 
First, an \emph{approximate message passing} algorithm was derived in \cite{perry2018message} for solving synchronization over general compact groups. Based on ideas from statistical physics, the work \cite{perry2018message} provides a non-rigorous analysis on the asymptotic statistical guarantee of their algorithm in the large $n$ limit.
Second, a powerful framework has recently been developed in \cite{lerman2019robust} for removing unreliable observations in the input data to general synchronization problems by leveraging a notion called \emph{cycle-edge consistency}. 
It would be interesting to investigate both the theoretical and practical performance of our non-convex approach when combined with the framework in~\cite{lerman2019robust}. We leave this as a future work.


\subsection{Contributions}
In this paper, we consider synchronization problems over \emph{closed subgroups of the orthogonal group $\mathcal{O}(d)$}, which include, \eg, the orthogonal group $\mathcal{O}(d)$ itself, the special orthogonal group $\mathcal{SO}(d)$, the permutation group $\mathcal{P}(d)$, and the cyclic group $\mathcal{Z}_m$ (see Section~\ref{subsec:setup} for their definitions). 
We propose and analyze a non-convex approach for tackling this class of synchronization problems.
Our main contribution is fivefold. 
\begin{itemize}
\item First, we derive a \emph{master theorem} for the proposed non-convex approach. Under four assumptions that are respectively related to the subgroup, noise, measurement graph, and initialization, our master theorem establishes an upper bound on the estimation error of the iterates of GPM and hence provides performance guarantees for the proposed non-convex approach (see Theorem~\ref{thm:master}). The master theorem is applicable to general closed subgroups of the orthogonal group, measurement graphs, and noise. It also clearly reveals the roles played by these main components of a synchronization problem.

\item Second, we formulate two key geometric conditions on the subgroup that can be used to verify the assumptions in the master theorem. These conditions are closely related to the \emph{error-bound geometry} of the subgroup, which is a classic notion in optimization and plays an important role in the analysis of various iterative methods. We also prove that the two geometric conditions hold for the orthogonal group, the special orthogonal group, the permutation group, and the cyclic group, which are all practically relevant in the context of group synchronization problems.

\item Third, we study random models of the measurement graph and noise. In particular, we show that if the measurement graph is the \emph{Erd{\H{o}}s-R{\'e}nyi random graph} and the noise matrix is a random matrix with independent \emph{sub-Gaussian} entries, then the assumptions on the measurement graph and noise in the master theorem will hold with high probability when the number $n$ of target group elements is sufficiently large. This result is useful since many well-known random variables are sub-Gaussian, including the Gaussian, uniform, Bernoulli, and any bounded random variables.

\item Fourth, we develop a novel spectral-type estimator, named the \emph{entropic spectral estimator}, for our target class of synchronization problems. The entropic spectral estimator has an intimate connection to the classic geometric concept of \emph{metric entropy}. We prove that under the above-mentioned geometric conditions and random models of the measurement graph and noise, the entropic spectral estimator will satisfy the assumption on the initialization of the non-convex approach with high probability.

\item Finally, through extensive numerical experiments, we study the empirical performance of our proposed non-convex approach on synchronization problems over several subgroups of the orthogonal group. The experiment results show that the proposed approach outperforms existing ones in terms of computational speed, scalability, and/or estimation error.
\end{itemize}

Although the idea of applying GPM to solve synchronization problems over closed subgroups of the orthogonal group $\mathcal{O}(d)$ is natural in view of our earlier discussion, due to the non-commutativity of $\mathcal{O}(d)$, many of the key steps in \cite{liu2016estimation} that rely on the commutativity of $\mathcal{SO}(2)$ break down. Hence, extending the theoretical results in \cite{liu2016estimation} to synchronization problems over general subgroups of $\mathcal{O}(d)$ is highly non-trivial. Furthermore, the measurement and noise settings in this paper are significantly more general than those considered in \cite{liu2016estimation}. It should also be pointed out that for cyclic synchronization $\mathcal{Z}_m$-\sync, another non-convex approach was developed in \cite{chen2018projected}.
However, unlike the approach in \cite{chen2018projected}, the dimension of the iterates and the computational cost of our approach do not increase with $m$. Therefore, our approach is arguably more efficient.

\subsection{Organization}
The rest of the paper is organized as follows. We formally introduce the group synchronization problem and some definitions related to it in Section~\ref{sec:problem}. In Section~\ref{sec:non-convex}, we propose a unified non-convex approach for solving synchronization problems over closed subgroups of the orthogonal group. We then prove a master theorem on the performance guarantee of the proposed approach in Section~\ref{sec:master}. 
In Section~\ref{sec:verify}, we verify the conditions of the master theorem for various closed subgroups of the orthogonal group and standard random measurement graph and noise models.
Lastly, we present results on the numerical performance of the proposed approach in Section~\ref{sec:exp} and conclude the paper in Section~\ref{sec:conclusion}.

\subsection{Notation}
We use the following notation throughout the paper. For any $nd\times nd$ (resp. $nd\times d$) block matrix $Y$, we denote by $[Y]_{ij}$ (resp. $[Y]_i$) its $(ij)$-th (resp. $i$-th) $d\times d$ block. For any two matrices $X$ and $Y$, we denote their Kronecker product by $X \otimes Y$ and, if they have conformable dimensions, their inner product by $\langle X , Y\rangle = \Tr( X^\top Y)$. We use $\nm{X}$ and $\nm{X}_F$ to denote the operator norm and Frobenius norm of $X$, respectively. For any integer $k\ge 1$, we denote the $k\times k$ identity matrix by $I_k$. We use $c,c_0,c_1,\ldots$ to denote numerical constants in mathematical statements and proofs, whose values may change from appearance to appearance. For a graph with node set $[n]:= \{1,\dots, n\}$, we denote by $(i,j)$ the edge between nodes $i$ and $j$.

\section{Group Synchronization}\label{sec:problem}
\subsection{Basic Setup} \label{subsec:setup}
Let $d\ge 1$ be an integer. The basic objects in our study are the $d$-dimensional orthogonal group --- \ie, the set of $d\times d$ orthogonal matrices
\[
\mathcal{O}(d) := \left\lbrace Q\in\mathbb{R}^{d\times d} : Q Q^\top = Q^\top Q = I_d \right\rbrace
\]
with matrix multiplication as the binary operation --- and its closed subgroups --- \ie, closed subsets of $\mathcal{O}(d)$ that form a group under matrix multiplication. These include many of the groups mentioned in Section~\ref{sec:intro}, such as the orthogonal group $\mathcal{O}(d)$ itself (the Boolean group corresponds to $\mathcal{O}(1)=\{-1,+1\}$), the special orthogonal group
\[
\mathcal{SO}(d) := \left\lbrace Q\in\mathcal{O}(d) :  \det(Q) = 1 \right \rbrace,
\]
the permutation group 
\[
\mathcal{P}(d) := \mathcal{O}(d) \cap \{0,1\}^{d\times d},
\]
and the cyclic group of order $m$
\begin{equation*}
\mathcal{Z}_m := \left\lbrace \begin{bmatrix}
\cos \frac{2k\pi}{m} & - \sin \frac{2k\pi}{m}\\[1mm]
\sin \frac{2k\pi}{m} & \cos \frac{2k\pi}{m}
\end{bmatrix} : k = 0, \dots, m-1 \right\rbrace
\end{equation*}
(which is a closed subgroup of $\mathcal{O}(2)$).

Given a closed subgroup $\mathcal{G}$ of the orthogonal group, the problem of \emph{group synchronization over $\mathcal{G}$}, denoted by $\G$-\sync, is to estimate the ground truth $G^* = (G_1^*,\dots,G_n^*) \in \G^n$ based on noisy observations of the pairwise ratios
\begin{equation*}
C_{ij} \approx  G_i^* {G_j^*}^{-1} , \quad (i,j) \in  E ,
\end{equation*}
where $E  \subseteq \{ (i,j)\in [n]^2 : i < j \}$ is a collection of index pairs. We view $\mathcal{G}^n$ as a subset of $\RR^{nd\times d}$, so that $[G^*]_i = G^*_i$ for $i=1,\ldots,n$. Note that if we multiply the target elements $G_1^*,\dots,G_n^*$ by a common group element $Q\in\G$ from the right, then the resulting elements $G_1^* Q, \dots, G_n^* Q$ would yield precisely the same set of measurements since
\begin{equation}
\label{eq:1}
G_i^*Q {(G_j^* Q)}^{-1}  =  G_i^* Q Q^{-1} {G_j^* }^{-1}  = G_i^* {G_j^*}^{-1} .
\end{equation} Therefore, we can at best recover the target elements up to some unknown common transformation $Q\in \G$ from the right. This motivates us to define the \emph{estimation error} $\varepsilon(G)$ of an estimator $G \in \G^n$ as 
\begin{equation}\label{eq:est_err}
\varepsilon( G ) := \min_{Q\in \G} \| G - G^* Q \|_F.
\end{equation}
It is immediate from the definition that $\varepsilon (G) = \varepsilon (G Q')$ for any $Q'\in \G$.

The pair $([n],  E)$ forms a graph, called the \emph{measurement graph} of the synchronization problem $\G$-\sync.
There are various ways to model the noisy observations of the pairwise ratios. One way is to adopt the \emph{additive noise model} 
\[ 
C_{ij} = G_i^* {G_j^*}^{-1} + \Theta_{ij}, \quad (i,j) \in E,
\] 
where $\{\Theta_{ij}:(i,j)\in E\}$ are the \emph{noise matrices}. Such a model appears frequently in the group synchronization literature, see, \eg, \cite{javanmard2016phase,singer2011angular,pachauri2013solving,boumal2016nonconvex,liu2016estimation,   bandeira2017tightness,perry2018optimality}. Another way is to adopt the \emph{multiplicative noise model}
\begin{equation*}
C_{ij} = G_i^* {G_j^*}^{-1} \Theta_{ij}, \quad (i,j) \in E,
\end{equation*}
see, \eg,~\cite{chen2018projected,wang2013exact,romanov2019noise}. A special case of this is the so-called \emph{outlier noise model}, where $\Theta_{ij}$ ($(i,j) \in E$) is either a random element distributed uniformly (with respect to the \emph{Haar measure}) over $\G$ or the identity element of $\G$. We will report numerical results of our proposed non-convex approach for both the additive and multiplicative noise models in Section~\ref{sec:exp}.


Most existing works on computational approaches to synchronization problems assume that the noise matrices $\{\Theta_{ij}:(i,j) \in E\}$ are independent. We shall also make this assumption in our subsequent development. However, it is worth noting that such an assumption may not be ideal for recovery. Indeed, in applications such as cryo-electron microscopy~\cite{singer2020computational}, the noisy observations $\{C_{ij}:(i,j) \in E\}$ are usually estimated as maximizers of the cross-correlation, in which case the noise matrices are \emph{dependent}. For approaches that address the case of dependent noise matrices, we refer the reader to, \eg,~\cite{bendory2018bispectrum} and the references therein.

\subsection{Least Squares Formulation}
For any closed subgroup $\G$ of the orthogonal group $\mathcal{O}(d)$, since $Q^{-1} = Q^\top$ for any $Q\in \G$, we can formulate the least squares estimation problem associated with $\G$-\sync as 
\begin{equation*}
\begin{array}{c@{\quad}l}
\displaystyle\min_{G \in \RR^{nd\times d}} & \displaystyle\sum_{(i,j)\in E }\|[G]_i [G]_j^\top - C_{ij} \|_F^2 \\
\noalign{\smallskip}
\mbox{subject to} & [G]_i\in \G,\ i=1,\dots,n.
\end{array}
\end{equation*}
Any optimal solution to this problem is said to be a least squares estimator for the problem $\G$-\sync.\footnote{By the same argument as in \eqref{eq:1}, it can be seen that the least squares estimator is not unique.}
By the orthogonality of the blocks $[G]_1,\ldots,[G]_n$, we can rewrite the above problem as
\begin{equation}
\label{opt:LS}
\begin{array}{c@{\quad}l}
\displaystyle\max_{G \in \RR^{nd\times d}} & \displaystyle \Tr\left( G^\top C G \right) \\
\noalign{\smallskip}
\mbox{subject to} & G\in \G^n,
\end{array}
\end{equation}
where $C\in \RR^{nd \times nd}$ is the block matrix defined by
\begin{equation} \label{eq:def-C}
[C]_{ij} := \begin{cases}
      \hfill C_{ij},    \hfill & \text{ if } (i,j)\in  E , \\
      \hfill C_{ji}^\top, \hfill & \text{ if } (j,i) \in  E , \\
      \hfill I_d, \hfill & \text{ if } i=j, \\
      \hfill \mathbf{0},  \hfill & \text{ otherwise,}
  \end{cases}
  \quad i, j = 1,\ldots,n.
\end{equation}
Problem~\eqref{opt:LS} is non-convex and in general NP-hard, as it is equivalent to the MAX-CUT problem when $\G$ is the Boolean group $\mathcal{O}(1) = \{+1, -1\}$. A standard idea for tackling problem~\eqref{opt:LS} is to consider its SDP relaxations~\cite{arie2012global,wang2013exact,javanmard2016phase,rosen2016se,bandeira2017tightness}, which can be solved by off-the-shelf solvers in polynomial time. However, this approach does not scale well with $n$ since it requires solving an SDP problem with an $nd\times nd$ matrix variable.

\section{A Non-Convex Approach}\label{sec:non-convex}

Instead of convexification, we propose to tackle the non-convex problem~\eqref{opt:LS} directly by a two-stage approach. An important subroutine in our approach is the projection of a $d\times d$ matrix onto the group $\G$. This projection, denoted by $\Pi_\G$, is defined by
\begin{equation}\label{def:Pi}
\Pi_{\G} \left( X \right)  := \argmin_{Q \in \G} \| X - Q \|_F = \argmax_{Q \in \G} \, \langle X , Q \rangle,\quad X\in\mathbb{R}^{d\times d},
\end{equation}
so that
\begin{equation*}
\text{dist}(X, \G) := \min_{Q \in \G} \| X - Q \|_F = \| X - \Pi_{\G} (X) \|_F.
\end{equation*}
Using the projection $\Pi_{\G}$, we can map any $nd\times d$ matrix to the feasible region $\G^n$ of the problem $\G$-\sync using the \emph{block-wise projection} $\Pi_{\G}^n: \RR^{nd\times d} \ra \G^n$ given by\footnote{Strictly speaking, the maps $\Pi_{\G}$ and $\Pi^n_{\G}$ are set-valued since in general problem~\eqref{def:Pi} can have multiple minimizers. Nevertheless, all the results in this paper hold for any of the minimizers.}
\[
\left[ \Pi^n_{\G} \left( Y \right) \right]_i = \Pi_{\G} \left( [Y]_i \right),  \quad i = 1,\dots, n.
\]
When the group $\G$ is clear from the context, we omit the subscript $\G$ from the symbols $\Pi^n_{\G}$ and $\Pi_{\G}$. We record a useful property of the projection $\Pi$ before moving on.
\begin{lem}\label{lem:1}
For any $X \in \RR^{d\times d}$ and any $r > 0$, we have
\begin{equation*}
\Pi ( r \cdot X) = \Pi ( X ).
\end{equation*}
\end{lem}
\noindent The proof of Lemma~\ref{lem:1} is straightforward and thus omitted.

We now describe the details of the two stages of our approach. The first stage aims to find a feasible point $G^0 \in \G^n$ that has a sufficiently small estimation error $\varepsilon (G^0)$. Theorem~\ref{thm:master} below provides an explicit upper bound on the estimation error $\varepsilon (G^0)$ that our non-convex approach requires. In Section~\ref{sec:spectral_ini}, we propose a novel spectral-type estimator, called the \emph{entropic spectral estimator}, for general group synchronization problems and show that under certain random models of the measurements and noise, the entropic spectral estimator satisfies the said upper bound on the estimation error with overwhelming probability. Moreover, the entropic spectral estimator can be computed efficiently. Hence, it can be used as $G^0$ in the first stage. 

In the second stage, starting with the initial point $G^0 $ obtained from the first stage, we iteratively refine the estimates using GPM, which is described below.
\begin{algorithm}[H]
\caption{Generalized Power Method (GPM) for $\G$-\sync} \label{alg:GPM}
\begin{algorithmic}[1]
\State \textbf{Input:} the matrix $C$ and an initial point $G^0$
\For {$t = 0,1,2,\dots$}
\State $G^{ t + 1 } = \Pi^n (CG^t)$
\EndFor
\end{algorithmic}
\end{algorithm}
The name ``generalized power method'' of Algorithm~\ref{alg:GPM} comes from its close resemblance to the classic power method~\cite{saad2011numerical} for computing the dominant eigenvector of a matrix. In fact, problem~\eqref{opt:LS} can be seen as a constrained eigenvalue problem. More precisely, each column $v$ of $G$ is not only normalized to a fixed length (indeed, $\nm{v}_2 = \sqrt{n}$) and orthogonal to all other columns as in the usual eigenvalue problem, but also subject to the extra constraint that the $d$-dimensional sub-vectors $v_{1:d}, v_{d+1: 2d},\dots, v_{(n-1)d+1:nd}$ are all of unit length (here, we use the \MATLAB notation $v_{k:\ell}$ to denote the sub-vector $(v_k,\cdots, v_\ell)^\top$ of $v$). Despite this resemblance, the analysis of GPM for solving group synchronization problems is highly non-trivial due to the said extra constraint.

Overall, the proposed non-convex approach enjoys several computational advantages and is very efficient. 
For the first stage, the main cost of computing the entropic spectral estimator lies in the computation of the first $d$ eigenvectors of the matrix $C$ and the generation of independent copies of uniformly random orthogonal matrices, which can be done efficiently by a host of modern eigen-solvers~\cite{saad2011numerical,liu2019quadratic} and random orthogonal matrix samplers~\cite{stewart1980efficient, genz2000methods, mezzadri2006generate}, respectively.
For the second stage, we can see from Algorithm~\ref{alg:GPM} that GPM does not require the tuning of any parameter and can be implemented extremely easily. The computational cost in each iteration consists of two parts: (i) $d$ matrix-vector multiplications for forming the product $CG^t$ and (ii) the block-wise projection $\Pi^n$, which can be decomposed into $n$ projections $\Pi$ onto the group $\G$. 
As we will see in Section~\ref{sec:group}, for the orthogonal group $\mathcal{O}(d)$, the special orthogonal group $\mathcal{SO}(d)$, and the permutation group $\mathcal{P}(d)$, the projection $\Pi$ reduces to a singular value decomposition (SVD) or a $d$-dimensional linear programming problem; for the cyclic group $\mathcal{Z}_m$, the projection $\Pi$ can be computed using a simple, explicit formula involving trigonometric functions. Thus, the two parts mentioned above are well-suited for parallelization and can be implemented efficiently for many closed subgroups of the orthogonal group.

%

\section{A Master Theorem}\label{sec:master}
In this section, we present a master theorem on the estimation performance of our proposed non-convex approach. We begin with some definitions. Let $ \bar{ E } =  E \cup \{(1,1),\dots, (n,n)\}$ and consider the \emph{extended measurement graph} $([n], \bar{ E })$. Every node in this graph has precisely one self-loop. For $i,j \in [n]$, let $w_{ij} = 1$ if either $(i,j) \in \bar{ E }$ or $(j,i) \in \bar{ E }$; otherwise let $w_{ij} = 0$. The degree $r_i$ of node $i \in [n]$ of the extended measurement graph is then given by $r_i = \sum_{j=1}^n w_{ij}$. Note that for any $i\in [n]$, we have $r_i \ge 1$ since $w_{ii} = 1$. 
Let $\Delta \in \RR^{nd \times nd}$ be the block matrix defined by
\begin{equation} \label{eq:Delta-def}
[\Delta]_{ij} := w_{ij} \left( \left[ C \right]_{ij} - \left[G^* {G^*}^\top \right]_{ij} \right) = \begin{cases}
      \hfill C_{ij} - G_i^* {G_j^*}^\top,   \hfill & \text{ if } (i,j)\in  E, \\
      \hfill C_{ji}^\top - G_i^* {G_j^*}^\top, \hfill & \text{ if } (j,i) \in  E, \\
      \hfill \mathbf{0}, \hfill & \text{ otherwise,}
  \end{cases}
  \quad i,j = 1,\ldots,n,
\end{equation}
where $C \in \RR^{nd \times nd}$ is the matrix defined in~\eqref{eq:def-C} and $G^* \in \G^n$ is the ground truth. Furthermore, let $\bar{D} := {\rm Diag}(r_1,\ldots,r_n) \in \RR^{n \times n}$, $\bar{W} \in \RR^{n \times n}$ be the matrix defined by $\bar{W}_{ij} = w_{ij}$, $i,j=1,\ldots,n$, and $e \in \RR^n$ be the all-one vector. Set $D = \bar{D} \otimes I_d \in \RR^{nd \times nd}$, $W = \bar{W} \otimes I_d \in \RR^{nd \times nd}$, and $F = ee^\top \otimes I_d \in \RR^{nd \times nd}$. We define the following parameter for the measurement graph:
\[ \kappa := \left\| D^{-1}W - \frac{1}{n}F \right\|. \]
The parameter $\kappa$ is related to the connectedness of the measurement graph $([n],E)$, see the discussion after Theorem~\ref{thm:master}.
For any $G \in \G^n$, we define 
\begin{equation}\label{def:Qt}
Q_G := \argmin_{ Q \in \G } \nm{ G - G^* Q }_F = \argmax_{Q \in \G} \, \left\langle {G^*}^\top G , Q \right\rangle  = \Pi \left( {G^*}^\top G \right),
\end{equation}
where the last equality follows from the definition of $\Pi$ in~\eqref{def:Pi}.
In particular, we have $\varepsilon (G) = \nm{ G - G^* Q_G }_F$. For simplicity, we write $Q^t = Q_{G^t}$ for all $t \ge 0$. 

The following theorem is the first main theoretical result of this paper, which provides a bound on the rate at which the estimation error of the iterates generated by the proposed non-convex approach decays under certain assumptions on the subgroup, measurement graph, noise, and initialization.

\begin{thm}[Master Theorem for GPM]
\label{thm:master}
Suppose that 
\begin{enumerate}[(i)]
\item\label{master_condi_i} 
there exists a constant $\alpha \ge 1$ such that for any $t \ge 0$,
\[ \| {G^*}^\top G^t - n\cdot Q^t\|_F \le \frac{\alpha}{2}\nm{ G^t - G^* Q^t }_F^2; \]
\item\label{master_condi_ii} $\kappa \le \tfrac{1}{32}$;
\item\label{master_condi_iii} $\nm{D^{-1} \Delta} \le \tfrac{1}{32}$ and $\nm{ \Pi^n \left( G^* + 2 D^{-1} \Delta G^* \right) - G^* }_F \le \tfrac{(\sqrt{2} -1)\sqrt{n}}{48\alpha}$; 
\item\label{master_condi_iv} $\varepsilon (G^0) \le \tfrac{\sqrt{n}}{8\alpha}$.
\end{enumerate}
Then, the iterates $\{G^t\}_{t\ge1}$ generated by GPM satisfy
\begin{equation*}
\varepsilon (G^t ) \le \left( \frac{1}{\sqrt{2}}\right)^t \varepsilon(G^0) + 6(\sqrt{2} +1)\nm{ \Pi^n \left( G^* + 2 D^{-1} \Delta G^* \right) - G^* }_F,  \quad \forall t \ge 1. 
\end{equation*}
\end{thm}

\noindent Let us elaborate on conditions~\eqref{master_condi_i}--\eqref{master_condi_iv} in and the implications of Theorem~\ref{thm:master} before giving its proof.
Condition~\eqref{master_condi_i} reflects the geometry of the subgroup $\G$ through the projection map $\Pi_{\G}$. Indeed, as we shall see in Section~\ref{sec:group}, it can be interpreted as an \emph{error bound condition}, which provides a measure of proximity of points in the convex hull of $\G$ to $\G$ itself. We will also show in Section~\ref{sec:group} that condition (i) holds with $\alpha = 1$ when $\G$ is the orthogonal group $\mathcal{O}(d)$, the special orthogonal group $\mathcal{SO}(d)$, or the permutation group $\mathcal{P}(d)$. We conjecture that condition~\eqref{master_condi_i} actually holds for arbitrary closed subgroups of the orthogonal group, see Conjecture~\ref{conj:group}. 
Condition~\eqref{master_condi_ii} is related to the measurement graph $([n],  E )$ and represents the requirement that the measurement graph needs to be sufficiently connected. In particular, we have $\kappa = 0$ when the measurement graph is complete (\ie, $w_{ij} = 1$ for all $1\le i<j\le n$). 
Condition~\eqref{master_condi_iii} depends on normalized deviation matrix $D^{-1}\Delta$ and, roughly speaking, requires that the noise in the measurements cannot be too large. This condition suggests that when we quantify the information contained in the observations $\{ C_{ij}: (i,j) \in E\}$, we should normalize the deviations matrices $\{ \Delta_{ij}: (i,j) \in E\}$ by the degrees $\{ r_i: i \in [n]\}$ of the nodes in the (extended) measurement graph. This is consistent with our intuition that under the same level of noise, more information is available at node $i$ if there are more observations related to $G^*_i$. For the special case of $\mathcal{SO}(2)$-\sync with a complete measurement graph, a similar condition is used in the analysis of the SDP relaxation approach, see \cite[Definition 3.1]{bandeira2017tightness}. 
Condition~\eqref{master_condi_iv} captures the requirement that GPM needs to be initialized with a point $G^0$ of sufficiently small estimation error. 

Under the setting of Theorem~\ref{thm:master}, we see that any accumulation point $G^\infty$ of the sequence of iterates $\{G^t\}_{t\ge0}$ generated by GPM satisfy
\begin{equation*}\label{ineq:19}
\varepsilon (G^\infty) \le 6(\sqrt{2} +1)\nm{ \Pi^n \left( G^* + 2 D^{-1} \Delta G^* \right) - G^* }_F,
\end{equation*}
which implies that, in the absence of noise (\ie, $\Delta = \mathbf{0}$), every accumulation point of the sequence $\{G^t\}_{t\ge0}$ is, up to some common transformation, equal to the ground truth $G^*$. In Section~\ref{sec:est_err}, we will show that under standard models of noise and measurement graph, the estimation error $\nm{ \Pi^n \left( G^* + 2 D^{-1} \Delta G^* \right) - G^* }_F$ achieved by our approach is nearly optimal for both continuous and discrete subgroups of the orthogonal group.

We should also point out that the numerical constants in the statement of Theorem~\ref{thm:master} are not important, as they can be improved by bootstrapping the analysis. The key message conveyed by Theorem~\ref{thm:master} is that the estimation error of the iterates produced by the suitably initialized GPM decreases at least geometrically to some quantity characterized by the subgroup, measurement graph and noise.

Note that Theorem~\ref{thm:master} does not guarantee the convergence of the sequence $\{G^t\}_{t\ge0}$ generated by GPM. In the recent paper~\cite{ling2020improved}, which appeared after this paper was posted on arXiv, Ling considered the problem $\mathcal{O}(d)$-\sync under the setting of complete measurement graph and additive Gaussian noise and showed that if the standard deviation of the noise is on the order of $\frac{\sqrt{n}}{\sqrt{d}(\sqrt{d}+\sqrt{\log n})}$, then GPM, when initialized by a spectral estimator, will generate iterates that converge linearly to a maximum likelihood estimator (which coincides with an optimal solution to the non-convex least squares formulation) with high probability. However, whether such a result can be extended to settings involving other closed subgroups of $\mathcal{O}(d)$ or more general measurement graphs or deterministic noise models remains open.

\subsection{Proof of Theorem~\ref{thm:master}}

We first establish an useful inequality for the projection map.
\begin{lem}\label{lem:projxy}
For any $X, Y\in \mathbb{R}^{d\times d}$, we have
\begin{equation*}\label{ineq:projxy}
\nm{\Pi \left( X + \Pi(X) + Y \right) - \Pi \left( X \right) }_F \le 2 \nm{ Y }_F.
\end{equation*}
\end{lem}

\begin{proof}
Let $ Q = \Pi \left( X + \Pi(X) + Y \right)$. We have
\begin{align*}
    \nm{ X + \Pi(X) + Y - Q }_F^2 \leq \nm{ X + \Pi(X) + Y - \Pi(X) }_F^2.
\end{align*}
Simply algebraic manipulation on the last inequality yields
\begin{align}\label{ineq:proj_Q}
   \left\langle Y, Q - \Pi(X) \right\rangle \ge \left\langle X + \Pi(X), \Pi(X) - Q \right\rangle.
\end{align}
Similarly, we also have 
\[ \| X - \Pi(X) \|_F^2 \le \| X - Q \|_F^2 , \]
which implies that
\begin{equation}\label{ineq:31}
\left\langle X, \Pi(X) - Q \right\rangle \geq 0 .
\end{equation}
Therefore,
\begin{align*}
& \|Y\|_F \cdot \|Q - \Pi(X)\|_F 
\ge  \left\langle Y, Q - \Pi(X) \right\rangle 
\ge \left\langle X + \Pi(X), \Pi(X) - Q \right\rangle \\
\ge & \left\langle \Pi(X), \Pi(X) - Q \right\rangle = \left\langle \Pi(X), \Pi(X) - Q \right\rangle = d - \left\langle \Pi(X), Q \right\rangle = \frac{1}{2}\nm{\Pi(X)- Q}_F^2 ,
\end{align*}
where the first inequality follows from the Cauchy-Schwarz inequality, the second from~\eqref{ineq:proj_Q}, and the third from~\eqref{ineq:31}. If $\|Q - \Pi(X)\|_F = 0$, then $\Pi \left( X + \Pi(X) + Y \right) = \Pi(X)$ and hence the desired inequality holds trivially. If $\|Q - \Pi(X)\|_F \neq 0$, then we have arrived at
\[ \|Y\|_F \ge \frac{1}{2}\nm{\Pi(X)- Q}_F, \]
which completes the proof.
\end{proof}

Several interesting consequences are immediate from Lemma~\ref{lem:projxy}. First, by taking $Y$ to be the zero matrix, we see that for any $X\in \RR^{d\times d}$,
\[ \Pi(X + \Pi(X) ) =  \Pi(X).\]
Second, by using the triangle inequality and Lemma~\ref{lem:projxy}, we can easily get that for any $X,Y\in\RR^{d\times d}$ and $Q\in\G$,
\begin{equation}\label{ineq:sum_proj}
\nm{\Pi \left( X + Y \right) - Q }_F  \le 2 \nm{  Y - Q }_F + 3\nm{   \Pi \left( X \right)  - Q }_F .
\end{equation}
The inequality~\eqref{ineq:sum_proj} will be used in the proof of the master theorem. Moreover, if we set $X = r Q$ for $r>0$ in~\eqref{ineq:sum_proj} and take the limit $r \to 0$, then we have
\begin{equation}\label{ineq:contraction}
\nm{\Pi \left( Y \right) - Q }_F \le 2 \nm{   Y  - Q }_F.
\end{equation}
The inequality~\eqref{ineq:contraction} was established in \cite[Proposition 3.3]{liu2016estimation} for the special case of $\mathcal{SO}(2)$ via a completely different proof. Therefore, Lemma~\ref{lem:projxy} is not only a strengthened but also a more general version of \cite[Proposition 3.3]{liu2016estimation}. Since \cite[Proposition 3.3]{liu2016estimation} has already been applied in a number of works to study phase synchronization problems~\cite{zhong2017near} or even other estimation problems~\cite{fanuel2020denoising,tyagi2020error}, we believe that Lemma~\ref{lem:projxy} will find further applications in synchronization or other estimation/optimization problems over general subgroups of the orthogonal group.

The proof of Theorem~\ref{thm:master} also relies on the following technical lemma:
\begin{lem}\label{lem:2}
Let $\alpha \ge 1$ be given.
If for some $t \ge 0$, 
\[\nm{ {G^*}^\top G^t - n\cdot Q^t}_F \le \frac{\alpha}{2}\nm{ G^t - G^* Q^t }_F^2,\] 
then
\begin{equation*}
\varepsilon (G^{t+1} ) \le 4\sqrt{2}  \left( \frac{\alpha \cdot \varepsilon(G^t) }{2\sqrt{n}} + \kappa + \nm{ D^{-1} \Delta }\right) \varepsilon(G^t) + 3\sqrt{2} \nm{ \Pi^n \left( G^* + 2D^{-1} \Delta G^* \right) - G^* }_F .
\end{equation*}
\end{lem}

\begin{proof}
By the definition of $G^{t+1}$ and Lemma~\ref{lem:1}, we have
\begin{equation*}
\varepsilon ( G^{t+1} ) =  \min_{Q\in \G } \| G^{t+1} - G^* Q \|_F \le   \| G^{t+1} - G^* Q^t \|_F =  \| \Pi^n \left( 2 D^{-1}CG^t \right) - G^* Q^t \|_F.
\end{equation*}
Using~\eqref{ineq:sum_proj} with $Q = G^*_i Q^t$, 
\begin{align*}
X  = [G^* Q^t + 2D^{-1} \Delta G^*Q^t]_i ,    
\end{align*}
and
\begin{align*}
Y  =  [G^* Q^t + 2\left( D^{-1} ( C - \Delta ) G^t - G^*Q^t \right) + 2D^{-1} \Delta ( G^t - G^* Q^t )]_i
\end{align*}
for $i \in [n]$, we get
\begin{equation}
\label{ineq:32}
\begin{split}
& \,\varepsilon ( G^{t+1} ) \\
\le &\, 2\sqrt{2} \nm{ G^* Q^t + 2\left( D^{-1} ( C - \Delta ) G^t - G^*Q^t \right) + 2D^{-1} \Delta ( G^t - G^* Q^t ) - G^* Q^t }_F \\
& \,\quad+ 3\sqrt{2} \nm{ \Pi^n \left( G^* + 2D^{-1} \Delta G^* \right) - G^* }_F \\
= &\, 4\sqrt{2}  \nm{D^{-1} ( C - \Delta ) G^t - G^*Q^t + D^{-1} \Delta ( G^t - G^* Q^t ) }_F \\
&\,\quad + 3\sqrt{2} \nm{ \Pi^n \left( G^* + 2D^{-1} \Delta G^* \right) - G^* }_F \\
\le & \, 4\sqrt{2} \nm{ D^{-1} ( C - \Delta ) G^t - G^*Q^t }_F + 4\sqrt{2} \nm{ D^{-1} \Delta }\cdot \varepsilon(G^t) \\
&\,\quad + 3\sqrt{2} \nm{ \Pi^n \left( G^* + 2D^{-1} \Delta G^* \right) - G^* }_F .
\end{split}
\end{equation}
Using the definitions of $C$ in~\eqref{eq:def-C}and $\Delta$ in~\eqref{eq:Delta-def}, it is easy to verify that $[C - \Delta]_{ij} = w_{ij} G_i^* {G^*_j}^\top$ for $i,j \in [n]$. It follows that
\begin{equation}\label{ineq:3}
\begin{split}
&\, \nm{ D^{-1} ( C - \Delta ) G^t - G^*Q^t }_F \\
=& \, \nm{ D^{-1} \cdot \BlkDiag(G_1^*,\ldots,G_n^*) \cdot W \cdot \BlkDiag\left( {G^*_1}^\top,\ldots,{G^*_n}^\top \right) \cdot G^t - G^*Q^t }_F \\
=& \, \nm{ \BlkDiag(G_1^*,\ldots,G_n^*) \left( (D^{-1}W) \cdot \BlkDiag\left( {G^*_1}^\top,\ldots,{G^*_n}^\top \right) \cdot G^t - e \otimes Q^t \right) }_F \\
=&\, \nm{ D^{-1}W \left( \BlkDiag\left( {G^*_1}^\top,\ldots,{G^*_n}^\top \right) \cdot G^t - e \otimes Q^t \right) }_F \\
\le&\, \frac{1}{n} \nm{ F\left( \BlkDiag\left( {G^*_1}^\top,\ldots,{G^*_n}^\top \right) \cdot G^t - e \otimes Q^t \right) }_F \\
&\quad+ \nm{ \left( D^{-1}W - \frac{1}{n}F \right) \left( \BlkDiag\left( {G^*_1}^\top,\ldots,{G^*_n}^\top \right) \cdot G^t - e \otimes Q^t \right) }_F \\
\le& \, \frac{1}{\sqrt{n}} \nm{ {G^*}^\top G^t - n \cdot Q^t }_F + \kappa \cdot \nm{ G^t - G^*Q^t }_F \\
\le& \, \left( \frac{\alpha \cdot \varepsilon (G^t)}{2\sqrt{n}} + \kappa \right) \varepsilon (G^t),
\end{split}
\end{equation}
where the second equality is due to the fact that $D^{-1}$ and $\BlkDiag(G_1^*,\ldots,G_n^*)$ commute and the identity
\[ G^*Q^t = \BlkDiag(G_1^*,\ldots,G_n^*) \cdot (e \otimes Q^t), \]
the third equality is due to the fact that $D^{-1}W(e \otimes Q^t) = e \otimes Q^t$, the second-to-last inequality is due to the identities
\begin{align*}
F\left( \BlkDiag\left( {G^*_1}^\top,\ldots,{G^*_n}^\top \right) \cdot G^t - e \otimes Q^t \right) &= e \otimes \left( {G^*}^\top G^t - n \cdot Q^t \right), \\
\nm{ \BlkDiag\left( {G^*_1}^\top,\ldots,{G^*_n}^\top \right) \cdot G^t - e \otimes Q^t }_F &= \nm{ G^t - \BlkDiag\left( {G^*_1},\ldots,{G^*_n} \right) \cdot (e \otimes Q^t) }_F \\
&= \nm{ G^t - G^*Q^t }_F,
\end{align*}
and the last inequality is due to the assumption of the lemma and the definition of $Q^t$. The desired result now follows by putting~\eqref{ineq:32} and~\eqref{ineq:3} together.
\end{proof}

Armed with Lemma~\ref{lem:2}, we can now prove the master theorem.
\begin{proof}[Proof of Theorem~\ref{thm:master}]
We first show by induction that for any $t\ge 0$,
\begin{equation}\label{ineq:8}
\varepsilon(G^t) \le \frac{\sqrt{n}}{8\alpha}  .
\end{equation}
For $t = 0$, this follows directly from the supposition that $\varepsilon (G^0) \le \tfrac{\sqrt{n}}{8\alpha} $. Next, we assume that $\varepsilon (G^t) \le \tfrac{\sqrt{n}}{8\alpha} $ for some $t \ge 0$.
By Lemma~\ref{lem:2}, conditions~\eqref{master_condi_ii}--\eqref{master_condi_iii}, and the inductive hypothesis, 
\begin{equation*}
\begin{split}
\varepsilon (G^{t+1} ) &\le 4\sqrt{2} \left( \frac{\alpha \cdot\varepsilon(G^t) }{2\sqrt{n}} + \kappa + \nm{ D^{-1} \Delta }\right) \varepsilon(G^t) + 3\sqrt{2} \nm{ \Pi^n \left( G^* + 2 D^{-1} \Delta G^* \right) - G^* }_F \\
& \le 4\sqrt{2} \left( \frac{1}{16} + \frac{1}{32} + \frac{1}{32} \right) \frac{\sqrt{n}}{8\alpha} + \frac{(2 -\sqrt{2})\sqrt{n}}{16 \alpha} = \frac{\sqrt{n}}{8 \alpha},
\end{split}
\end{equation*}
which yields~\eqref{ineq:8}.
Using Lemma~\ref{lem:2}, conditions~\eqref{master_condi_ii}--\eqref{master_condi_iii} and inequality~\eqref{ineq:8}, we have that
\begin{equation*}
\begin{split}
\varepsilon (G^{t+1} ) &\le 4\sqrt{2} \left( \frac{\alpha\cdot\varepsilon(G^t) }{2\sqrt{n}} + \kappa+ \nm{ D^{-1} \Delta }\right) \varepsilon(G^t) + 3\sqrt{2} \nm{ \Pi^n \left( G^* + 2 D^{-1} \Delta G^* \right) - G^* }_F\\
& \le \frac{1}{\sqrt{2}} \cdot \varepsilon(G^t) + 3\sqrt{2} \nm{ \Pi^n \left( G^* + 2 D^{-1} \Delta G^* \right) - G^* }_F \\
& \le \frac{1}{\sqrt{2}} \cdot \left( \frac{1}{\sqrt{2}} \cdot \varepsilon(G^{t-1}) + 3\sqrt{2} \nm{ \Pi^n \left( G^* + 2 D^{-1} \Delta G^* \right) - G^* }_F \right) \\
& \quad + 3\sqrt{2} \nm{ \Pi^n \left( G^* + 2 D^{-1} \Delta G^* \right) - G^* }_F \\
& = \left( \frac{1}{\sqrt{2}} \right)^2 \varepsilon(G^{t-1}) + \left( 1 + \frac{1}{\sqrt{2}} \right) 3\sqrt{2} \nm{ \Pi^n \left( G^* + 2 D^{-1} \Delta G^* \right) - G^* }_F \\
& \le \left(\frac{1}{\sqrt{2}}\right)^{t+1} \varepsilon(G^0) + \left( 1 + \frac{1}{\sqrt{2}} + \left(\frac{1}{\sqrt{2}}\right)^2 + \cdots  \right) 3\sqrt{2} \nm{ \Pi^n \left( G^* + 2 D^{-1} \Delta G^* \right) - G^* }_F \\
& = \left( \frac{1}{\sqrt{2}}\right)^{t+1} \varepsilon(G^0) + 6(\sqrt{2} +1)\nm{ \Pi^n \left( G^* + 2 D^{-1} \Delta G^* \right) - G^* }_F ,
\end{split}
\end{equation*}
which completes the proof.

\end{proof}

In the next section, we show that condition~\eqref{master_condi_i} in Theorem~\ref{thm:master} is satisfied by various closed subgroups of the orthogonal group, while conditions~\eqref{master_condi_ii} and~\eqref{master_condi_iii} are satisfied by certain random measurement graph and random noise models. Moreover, we propose a novel spectral-type estimator and show that for the said subgroups and under the said random measurement graph and noise models, condition~\eqref{master_condi_iv} in Theorem~\ref{thm:master} is satisfied by the proposed estimator. These results demonstrate the utility and power of Theorem~\ref{thm:master}.


\section{Verifying the Conditions of the Master Theorem} \label{sec:verify}
\subsection{Error-Bound Geometry of the Subgroup $\G$}\label{sec:group}
In order to apply the master theorem in the last section, the subgroup $\G$ has to satisfy certain geometric conditions. In this section, we formulate these conditions and verify them for four specific subgroups, namely the orthogonal group $\mathcal{O}(d)$, the special orthogonal group $\mathcal{SO}(d)$, the permutation group $\mathcal{P}(d)$, and the cyclic group $\mathcal{Z}_m$. Along the way, we show that the projection maps associated with these subgroups can be computed in a tractable manner.

To begin, let $\text{conv}(\G)$ denote the Euclidean convex hull of the subgroup $\G$ in $\mathbb{R}^{d\times d}$ (not the geodesic convex hull on $\mathcal{O}(d)$).
It is easy to check that a point $X\in \text{conv}(\G)$ lies in the subgroup $\G$ if and only if any (and hence all) of the following three quantities vanishes: 
\[ \text{\normalfont{dist}}(X, \G),\quad  \Tr ( I_d - X^\top \Pi_{\G} (X) ),\quad d - \nm{ X }_F^2.\]
Therefore, these quantities can serve as proximity measures to $\G$ for points in $\text{conv}(\G)$. We now investigate the connection between our non-convex approach to the geometry of $\G$, as reflected through the ratios of these proximity measures.
The precise geometric conditions are given as follows.
\begin{cond}\label{cond:alpha}
There exists a constant $\alpha \ge 1$ such that
\begin{equation*}
\text{\normalfont{dist}}(X, \G) \le  \alpha  \Tr ( I_d - X^\top \Pi_{\G} (X) ), \quad \forall X\in \rm{conv}(\G).
\end{equation*}
\end{cond}

\begin{cond}\label{cond:beta}
There exists a constant $\beta \in (0,1]$ such that
\begin{equation*}
\beta \Tr ( I_d - X^\top \Pi_{\G} (X) ) \le   d - \nm{ X }_F^2, \quad \forall X\in \rm{conv}(\G).
\end{equation*}
\end{cond}
The above two conditions play an important role in our development.
Indeed, if we take $X = \tfrac{1}{n}{G^*}^\top G$, then Condition~\ref{cond:alpha} immediately implies that
\[ \nm{ {G^*}^\top G - n \cdot Q_G}_F \le \alpha \Tr\left(n\cdot I_d - Q_G^\top {G^*}^\top G \right)  = \frac{\alpha}{2} \nm{ G - G^* Q_G }_F^2, \]
which is precisely condition~\eqref{master_condi_i} in the master theorem. The usefulness of Condition~\ref{cond:beta} will become clear when we study the initialization for GPM in Section~\ref{sec:spectral_ini}.

It is worth noting that Conditions~\ref{cond:alpha} and~\ref{cond:beta} are reminiscent of \emph{error bound conditions}, which have been extensively studied in the optimization literature and applied to analyze the convergence rates of various iterative methods, see, \eg,~\cite{zhou2017unified, liu2019quadratic,yue2019quadratic, yue2019family,wang2021linear, yang2022variance} and the references therein. Roughly speaking, an error bound condition postulates that the distance function associated with some target set (often difficult to characterize theoretically and not computable in practice) is bounded above by a continuous surrogate function (often easier to characterize theoretically and compute in practice) that vanishes on the target set. In the context of synchronization problems, such a condition was first introduced in~\cite{liu2016estimation} to study the optimization performance of GPM.

The main result of this subsection is summarized in the following theorem, which will be proved 
in a case-by-case manner.
\begin{thm}\label{thm:group}
For $\G=\mathcal{O}(d)$ or $\G=\mathcal{SO}(d)$, the projection $\Pi_{\G}$ can be computed in closed form via the SVD of a $d\times d$ matrix; for $\G = \mathcal{P}(d)$, the projection $\Pi_{\G}$ can be computed by solving a $d$-dimensional linear programming problem; for $\G = \mathcal{Z}_m$, the projection $\Pi_{\G}$ can be computed in closed form via the formula in Proposition~\ref{prop:cyclic_proj}. Moreover, Conditions~\ref{cond:alpha} and \ref{cond:beta} hold for the groups $\mathcal{O}(d)$, $\mathcal{SO}(d)$, $\mathcal{P}(d)$, and $\mathcal{Z}_m$ with 
\begin{equation*}
\alpha = 
\begin{cases}
1, &\text{if } \G = \mathcal{O}(d), \, \mathcal{SO}(d), \, \text{or } \mathcal{P}(d), \\
1, &\text{if } \G = \mathcal{Z}_m \text{ with } m = 1,2,\\
\left( \sqrt{2}\sin \tfrac{\pi}{m} \right)^{-1}, &\text{if } \G = \mathcal{Z}_m \text{ with } m \ge 3
\end{cases}
\end{equation*}
and
\begin{equation*}
\beta = 
\begin{cases}
1, &\text{if } \G = \mathcal{O}(d),\\
\frac{1}{2}, &\text{if } \G = \mathcal{SO}(d),\\
\frac{2}{d}, &\text{if } \G = \mathcal{P}(d),\\
1, &\text{if } \G = \mathcal{Z}_m \text{ with } m = 1,\\
\sin^2 \frac{\pi}{m}, &\text{if } \G = \mathcal{Z}_m \text{ with } m \ge 2.
\end{cases}
\end{equation*}
\end{thm}
Motivated by Theorem~\ref{thm:group}, we formulate the following conjecture:
\begin{conj}\label{conj:group}
Let $\G$ be any closed subgroup of the orthogonal group $\mathcal{O}(d)$. Then, there exist constants $\alpha \ge 1$ and $\beta\in (0,1]$ such that
\begin{equation*}\label{ineq:conj}
\frac{\text{\normalfont{dist}}(X, \G)}{\alpha} \le  \Tr ( I_d - X^\top \Pi_{\G} (X) ) \le \frac{d - \nm{ X }_F^2 }{\beta}, \quad \forall X\in \rm{conv}(\G).
\end{equation*} 
\end{conj}
\noindent The validity of Conjecture~\ref{conj:group} would imply that condition~\eqref{master_condi_i} in the master theorem is superfluous.


\subsubsection{Proof of Theorem~\ref{thm:group}: Orthogonal Group $\mathcal{O}(d)$} \label{subsec:O_d} 
The projection $\Pi_{\mathcal{O}(d)}$ is given by the solution to the well-known \emph{orthogonal Procrustes problem}~\cite{gower2004procrustes}, \ie, for any $X\in\mathbb{R}^{d\times d}$ with SVD $U_X \Sigma_X V_X^\top$, we have
\begin{equation*}
\Pi_{\mathcal{O}(d)} (X) = U_X V_X^\top.
\end{equation*} 

Now, let $X\in \text{conv}(\mathcal{O}(d))$ be arbitrary. By Carath\'eodory's theorem, there exist $Q_1,\dots,Q_L\in \mathcal{O}(d)$ and $\omega_1, \dots, \, \omega_L \ge 0 $ such that 
\begin{equation*}
\sum_{\ell = 1}^L \omega_\ell  = 1 \quad\text{and}\quad X = \sum_{\ell = 1}^L \omega_\ell Q_\ell.
\end{equation*}
Then, for any $j \in [d]$,
\begin{equation}\label{ineq:4}
\sigma_j (X) = \sigma_j \left( \sum_{\ell =1}^L \omega_\ell Q_\ell \right) \le \sum_{\ell = 1}^L \omega_\ell \sigma_j \left( Q_\ell \right) = \sum_{\ell = 1}^L \omega_\ell = 1,
\end{equation}
where $\sigma_j(\cdot) $ denotes the $j$-th largest singular value.
This implies that $I_d - \Sigma_X $ is positive semidefinite and 
\begin{equation*}
\begin{split}
\text{dist}(X, \mathcal{O}(d))  &= \nm{ X - \Pi_{\mathcal{O}(d)}(X) }_F = \nm{U_X \Sigma_X V_X^\top  - U_X  V_X^\top }_F  = \nm{ I_d - \Sigma_X }_F \\
& \le \Tr\left( I_d - \Sigma_X \right)  = \Tr\left(  I_d - V_X \Sigma_X {V_X}^\top \right)  = \Tr\left( I_d - {X}^\top \Pi_{\mathcal{O}(d)}(X) \right).
\end{split}
\end{equation*}
Thus, Condition~\ref{cond:alpha} holds for $\mathcal{O}(d)$ with $\alpha = 1$. Furthermore, we have 
\begin{equation}
\begin{split}\label{ineq:26}
&\, \Tr\left( I_d - {X}^\top \Pi_{\mathcal{O}(d)}(X) \right)  = \Tr\left( I_d - \Sigma_X \right) = \sum_{j = 1}^d (1-\sigma_j(X)) \\
\le&\, \sum_{j = 1}^d (1-\sigma_j(X)) (1+\sigma_j(X)) = \sum_{j = 1}^d (1-\sigma_j^2 (X)) = d- \nm{X}_F^2,
\end{split}
\end{equation}
which shows that Condition~\ref{cond:beta} holds for $\mathcal{O}(d)$ with $\beta = 1$.

\subsubsection{Proof of Theorem~\ref{thm:group}: Special Orthogonal Group $\mathcal{SO}(d)$}\label{subsec:SO_d}
We adopt a similar strategy to the one used in Section~\ref{subsec:O_d} to establish Condition~\ref{cond:alpha} for $\mathcal{SO}(d)$. First, we have the following result, which gives an explicit formula for the projection $\Pi_{\mathcal{SO}(d)}$ and is known as the \emph{Kabsch algorithm}~\cite{kabsch1978discussion}.
\begin{lem}\label{lem:kabsch}
Let $X \in \RR^{d \times d}$ be a matrix with SVD $X = U_X \Sigma_X V_X^\top$. Let $I_X \in \RR^{d\times d}$ be the diagonal matrix defined by $I_X = {\rm Diag}(1,\ldots,1,\det(U_X V_X^\top))$. Then, we have
\begin{equation*}
\Pi_{\mathcal{SO}(d)} \left( X \right) = U_X I_X V_X^\top.
\end{equation*}
\end{lem}
\noindent Now, let $X \in {\rm conv}(\mathcal{SO}(d))$ be arbitrary. Using Lemma~\ref{lem:kabsch}, we have
\begin{equation*}
\begin{split}
\text{dist}(X, \mathcal{SO}(d))& = \nm{ X - \Pi_{\mathcal{SO}(d)}(X)}_F  = \nm{U_X \Sigma_X {V_X}^\top -  U_X I_X {V_X}^\top}_F = \nm{I_d -  \Sigma_X I_X }_F \\
& \le \Tr\left( I_d - \Sigma_X I_X \right)  = \Tr\left( I_d - V_X \Sigma_X I_X {V_X}^\top \right) = \Tr\left( I_d - X^\top \Pi_{\mathcal{SO}(d)}(X) \right),
\end{split}
\end{equation*}
where the inequality follows from \eqref{ineq:4}. This shows that Condition~\ref{cond:alpha} holds for $\mathcal{SO}(d)$ with $\alpha = 1$.

To establish Condition~\ref{cond:beta} for $\mathcal{SO}(d)$, let $X \in {\rm conv}(\mathcal{SO}(d))$ be arbitrary and consider first the case where $\det(U_X V_X) = 1$. Then, we have $I_X = I_d$ and $\Pi_{\mathcal{SO}(d)} (X) = \Pi_{\mathcal{O}(d)} (X)$. It follows from~\eqref{ineq:26} that
\[ 
\Tr\left( I_d - {X}^\top \Pi_{\mathcal{SO}(d)}(X) \right) = \Tr\left( I_d - {X}^\top \Pi_{\mathcal{O}(d)}(X) \right)  
\le d- \nm{X}_F^2.
\]
Next, consider the case where $\det(U_X V_X) = -1$, \ie, $\Pi_{\mathcal{O}(d)} (X) \in \mathcal{O}(d) \setminus \mathcal{SO}(d)$. Since $X\in \text{conv}(\mathcal{SO}(d))$, by Carath\'eodory's theorem, there exist $Q_1,\dots,Q_L\in \mathcal{SO}(d)$ and $\omega_1, \dots, \, \omega_L \ge 0 $ such that 
\begin{equation*}
\sum_{\ell = 1}^L \omega_\ell  = 1 \quad\text{and}\quad X = \sum_{\ell = 1}^L \omega_\ell Q_\ell.
\end{equation*}
Hence, we have
\begin{equation}\label{ineq:27}
\begin{split}
d - \nm{X}_F^2 & \ge \Tr\big( I_d - X^\top \Pi_{\mathcal{O}(d)} (X) \big) = \sum_{\ell = 1}^L \omega_\ell \sum_{j=1}^d \big(1 - \lambda_j( Q_\ell^\top \Pi_{\mathcal{O}(d)}(X) )\big) \ge \sum_{\ell = 1}^L  \omega_\ell\cdot 2 = 2,
\end{split}
\end{equation}
where the first inequality follows from~\eqref{ineq:26} and the second inequality follows from the fact that $Q_\ell^\top \Pi_{\mathcal{O}(d)}(X)\in \mathcal{O}(d)\setminus\mathcal{SO}(d)$ and hence the eigenvalues of $Q_\ell^\top \Pi_{\mathcal{O}(d)}(X)$ are either $+1$ or $-1$ with at least one being $-1$. This gives
\begin{align*}
&\, \Tr\big( I_d - X^\top \Pi_{\mathcal{SO}(d)} (X) \big) = \Tr\left( I_d - \Sigma_X I_X \right) = 1+ \sigma_d(X) + \sum_{j = 1}^{d-1} (1 - \sigma_j(X)) \\
\le &\, 2 + \sum_{j = 1}^{d-1} (1 - \sigma_j^2 (X)) \le 2 + \sum_{j = 1}^d (1 - \sigma_j^2 (X)) \le 2\left( d - \nm{X}_F^2 \right),
\end{align*}
where the first equality follows from Lemma~\ref{lem:kabsch}, the second equality follows from the assumption that $\det(U_X V_X) = -1$, the first inequality follows from the fact that $\sigma_j(X) \le 1$ for $j = 1,\dots, d$, and the last inequality follows from~\eqref{ineq:27}. This shows that Condition~\ref{cond:beta} holds for $\mathcal{SO}(d)$ with $\beta = \tfrac{1}{2}$.

\subsubsection{Proof of Theorem~\ref{thm:group}: Permutation Group $\mathcal{P}(d)$}\label{subsec:pi_d}
It is easy to verify that the projection $\Pi_{\mathcal{P}(d)}$ is given by an optimal solution to the \emph{assignment problem}~\cite{burkard2012assignment}, \ie, for any $X\in \mathbb{R}^{d\times d}$, we have
\begin{equation*}
\Pi_{\mathcal{P}(d)}(X) = \argmax_{Q\in \mathcal{P}(d) } \, \langle Q, X \rangle.
\end{equation*}
The above problem can be solved in polynomial time by linear programming or the Hungarian (also known as the Kuhn--Munkres) algorithm.

To establish Condition~\ref{cond:alpha} for $\mathcal{P}(d)$, we first note that for any $Q\in \mathcal{P}(d)$, we either have $d = \Tr(Q)$ or $d - \Tr(Q) \ge 2$. It follows that
\begin{equation}\label{ineq:5}
\begin{split}
& \, \nm{I_d - Q }_F^2 = \sum_{i =1}^d \sum_{j=1}^d (I_d - Q )_{ij}^2  =  \sum_{i =1}^d \sum_{j=1}^d | (I_d - Q )_{ij} |   \\
=& \, \sum_{i \neq j}   Q_{ij} + \sum_{i=1}^d (1- Q_{ii} ) =  2d - 2 \Tr(Q)  \le ( d - \Tr(Q))^2 .
\end{split}
\end{equation}
For any $Y \in \text{conv}(\mathcal{P}(d))$, by Carath{\'e}odory's theorem, there exist $Q_1,\dots,Q_L\in \mathcal{P}(d)$ and $\omega_1, \dots, \, \omega_L \ge 0 $ such that
\begin{equation*}
\sum_{\ell = 1}^L \omega_\ell = 1\quad \text{and}\quad Y = \sum_{\ell  = 1}^L \omega_\ell Q_\ell .
\end{equation*}
Using~\eqref{ineq:5}, we get
\begin{equation}\label{ineq:21}
\nm{ I_d - Y }_F \le \sum_{\ell = 1}^L \omega_\ell \nm{ I_d - Q_\ell }_F \le \sum_{\ell = 1}^L \omega_\ell \left( d - \Tr\left( Q_\ell \right) \right) = d - \Tr(Y).
\end{equation}
Now, let $X\in \text{conv}(\mathcal{P}(d))$ be arbitrary. Since ${\Pi_{\mathcal{P}(d)}(X)}^\top X \in \text{conv}(\mathcal{P}(d))$, by invoking~\eqref{ineq:21} with $Y = {\Pi_{\mathcal{P}(d)}(X)}^\top X$, we obtain
\begin{equation*}
\begin{split}
\text{dist}(X, \mathcal{P}(d)) & = \nm{X - \Pi_{\mathcal{P}(d)}(X)}_F = \nm{ {\Pi_{\mathcal{P}(d)}(X)}^\top X - I_d }_F\\
& \le d - \Tr\left( {\Pi_{\mathcal{P}(d)}(X)}^\top X \right) 
=  \Tr\left( I_d - X^\top {\Pi_{\mathcal{P}(d)}(X)} \right).
\end{split}
\end{equation*}
This shows that Condition~\ref{cond:alpha} holds for $\mathcal{P}(d)$ with $\alpha=1$.

We will need the following definition.
\begin{defi}[Minimum Separation]
\label{defi:min_sep}
The \emph{minimum separation} of a discrete set $\G$ is defined as
\[ \tau =  \min_{Q, Q' \in \G, \, Q \neq Q'} \| Q - Q' \|_F. \]
\end{defi}

To establish Condition~\ref{cond:beta} for $\mathcal{P}(d)$, we prove the following stronger result, which relates $\beta$ to the minimum separation and states that Condition~\ref{cond:beta} actually holds for any discrete subgroup of $\mathcal{O}(d)$.

\begin{prop}\label{prop:cond_2_discrete}
Let $\G$ be a discrete subgroup of $\mathcal{O}(d)$ with at least two elements and minimum separation $\tau$. Then, Condition~\ref{cond:beta} holds for $\G$ with
\[ \beta = \frac{\tau^2}{2d}. \]
\end{prop}

\begin{proof}
The compactness of $\mathcal{O}(d)$ implies that $\G$ must be finite. Hence, we may let $\G = \{Q_1,\ldots,Q_L\}$ with $Q_1,\ldots,Q_L \in \mathcal{O}(d)$. For any $X\in \text{conv}(\G)$, there exist $\omega_1, \dots, \, \omega_L \ge 0 $ such that
\begin{equation*}
\sum_{\ell = 1}^L \omega_\ell = 1\quad \text{and}\quad X = \sum_{\ell  = 1}^L \omega_\ell Q_\ell .
\end{equation*}
Without loss of generality, suppose that $\omega_1 = \max_{\ell \in [L]} \omega_\ell$. Then, we have
\begin{equation*}
\begin{split}
\max_{Q\in \G} \Tr(X^\top Q) \ge \Tr(X^\top Q_1) = \Tr\left( \sum_{\ell = 1}^L \omega_\ell Q_\ell^\top Q_1 \right) \ge \omega_1 d + (1-\omega_1) \min_{i \neq j} \Tr(Q_i^\top Q_j),
\end{split}
\end{equation*}
which implies that
\begin{equation}
\label{ineq:28}
\Tr \big( I_d - X^\top \Pi_{\G}(X) \big) = d - \max_{Q\in \G} \Tr(X^\top Q) \le (1-\omega_1) \left( d - \min_{i \neq j} \Tr(Q_i^\top Q_j) \right).
\end{equation}
On the other hand, we have
\begin{equation*}
\begin{split}
&\, \nm{X}_F^2  = \sum_{\ell = 1}^L \omega_\ell^2 \nm{Q_\ell}_F^2 + \sum_{\ell' \neq \ell}\omega_\ell \omega_{\ell'} \Tr(Q_{\ell'}^\top Q_\ell) \le \sum_{\ell = 1}^L \omega_\ell^2 d + \sum_{\ell' \neq \ell}\omega_\ell \omega_{\ell'} \max_{i\neq j}\Tr(Q_i^\top Q_j )\\
\le & \, \sum_{\ell = 1}^L \omega_\ell^2 d + \left(1 - \sum_{\ell = 1}^L \omega_\ell^2 \right) \max_{i\neq j}\Tr(Q_i^\top Q_j ) = d - \left( d - \max_{i\neq j}\Tr(Q_i^\top Q_j ) \right)\left(1 - \sum_{\ell = 1}^L \omega_\ell^2 \right) \\
\le &\, d - \left( d - \max_{i\neq j}\Tr(Q_i^\top Q_j ) \right)\left(1 - \omega_1 \right) ,
\end{split}
\end{equation*}
which, upon substitution into~\eqref{ineq:28}, shows that Condition~\ref{cond:beta} holds for $\G$ with 
\[ \beta = \frac{d - \max_{ i \neq j } \Tr(Q_i^\top Q_j)}{d - \min_{ i \neq j } \Tr(Q_i^\top Q_j)}. \]
Finally, note that 
\[ 2\left( d- \min_{ i \neq j } \Tr(Q_i^\top Q_j) \right) = \max_{ i \neq j } \nm{ Q_i - Q_j }_F^2 = 2d, \]
and that 
\[ 2\left( d - \max_{ i \neq j } \Tr(Q_i^\top Q_j) \right) = \min_{ i \neq j } \nm{ Q_i - Q_j }_F^2 = \tau^2 .\]
This completes the proof.
\end{proof}

When $\G = \mathcal{P}(d)$, a simple calculation shows that
$ \tau = 2$ .
Hence, by Proposition~\ref{prop:cond_2_discrete}, Condition~\ref{cond:beta} holds for $\mathcal{P}(d)$ with $\beta = \frac{2}{d}$.

\subsubsection{Proof of Theorem~\ref{thm:group}: The Cyclic Group $\mathcal{Z}_m$}\label{subsec:Z_m}
We first prove that Condition~\ref{cond:alpha} holds for $\mathcal{Z}_m$ for any integer $m\ge1$. In particular, we show that
\begin{equation*}
\text{\normalfont{dist}} (X, \mathcal{Z}_m) \le
\begin{cases}
      \Tr\left( I_d - X^\top \Pi_{\mathcal{Z}_m}(X) \right), \hfill & \text{ if } m = 1, 2 , \\
      \hfill \frac{1}{\sqrt{2}\sin\left( \frac{\pi}{m} \right)}\Tr\left( I_d - X^\top \Pi_{\mathcal{Z}_m}(X) \right), \hfill & \text{ if } m \ge 3.
  \end{cases}
\end{equation*}
To do so, we note that every $X\in \text{conv}(\mathcal{Z}_m)$ is of the form
\begin{equation*}
X = \begin{bmatrix}
x & -y \\
y & x
\end{bmatrix}
\end{equation*}
and gives rise to a complex number $z_X \in \mathbb{C}$ defined by $z_X = x + \mathbf{i} y$, where $\mathbf{i}$ is the imaginary unit.
Let $\Re(z)$ denote the real part of any complex number $z\in\mathbb{C}$,
\begin{equation}\label{eq:6}
Q_k = \begin{bmatrix}
\cos \frac{2k\pi}{m} & - \sin \frac{2k\pi}{m}\\[1mm]
\sin \frac{2k\pi}{m} & \cos \frac{2k\pi}{m}
\end{bmatrix},\quad k = 0,\dots,m-1,
\end{equation}
and 
\begin{equation*}
\hat{k}(X) \in \argmin_{ k \in \{0, \dots, m-1\}} \nm{ X - Q_k }_F^2.
\end{equation*}
For $m = 1$, we have $\hat{k}(X) = 0$ for any $X\in \text{conv}(\mathcal{Z}_1)$ and hence $\Pi_{\mathcal{Z}_1} (X) = Q_{0} = I_2$. It follows that
\begin{equation*}
\text{\normalfont{dist}} (X, \mathcal{Z}_1) = \nm{X - I_2}_F \le \Tr(I_2 - X) = \Tr(I_2 - X^\top \Pi_{\mathcal{Z}_1} (X)),
\end{equation*}
\ie, Condition~\ref{cond:alpha} holds for $\mathcal{Z}_1$ with $\alpha = 1$. For $m = 2$, if $\Re(z_X) = x \ge 0$, then $\hat{k}(X) = 0$ and $\Pi_{\mathcal{Z}_2} (X) = Q_{0} = I_2$; if $\Re(z_X) = x < 0$, then $\hat{k}(X) = 1$ and $\Pi_{\mathcal{Z}_2} (X) = Q_{1} = -I_2$. Therefore, we have
\begin{align*}
\text{\normalfont{dist}} (X, \mathcal{Z}_2) & =
				\left\{
                \begin{array}{l}
                  \nm{X - I_2}_F\\
                  \nm{X + I_2}_F
                \end{array} \le
              \right. 
              \left\{
                \begin{array}{l}
                   \Tr(I_2 - X)\\
                  \Tr(I_2 + X)
                \end{array}
                \right. 
                 \le 
                 \Tr(I_2 - X^\top \Pi_{\mathcal{Z}_2} (X) ),
                 \quad
                \begin{array}{l}
                   \text{ if } x \ge 0, \\
                   \text{ if } x < 0,
                \end{array}
\end{align*}
which shows that Condition~\ref{cond:alpha} holds for $\mathcal{Z}_2$ with $\alpha = 1$.

Now, consider the case where $m\ge 3$. On one hand, we have 
\begin{equation}\label{eq:2}
 \text{\normalfont{dist}} (X, \mathcal{Z}_m) = \nm{ X - Q_{\hat{k}(X)} }_F
 = \sqrt{2} \left| z_X - z_{Q_{\hat{k}(X)}} \right| = \sqrt{2}  \left| z_X - e^{\frac{2\hat{k}(X) \pi \mathbf{i}}{m}} \right| ,
\end{equation}
where $\left| \,\cdot\,\right|$ denotes the modulus of a complex number.
On the other hand,
\begin{equation}\label{eq:3}
\Tr\left( I_2 - X^\top \Pi_{\mathcal{Z}_m}(X) \right) = 2 - \Tr (X^\top Q_{\hat{k}(X)} ) = 2 \left(1-  \Re\left( z_X \cdot e^{-\frac{2\hat{k}(X) \pi \mathbf{i}}{m}} \right)\right).
\end{equation}
Therefore, in order to establish Condition~\ref{cond:alpha} for $\mathcal{Z}_m$, we need to bound the ratio
\begin{equation}\label{eq:ratio}
\frac{\left| z - e^{\frac{2 k \pi \mathbf{i}}{m}} \right|}{1-  \Re\left( z \cdot e^{-\frac{2 k \pi \mathbf{i}}{m}} \right)}
\end{equation}
subject to the constraint that $z = z_X$ for some $X\in \text{conv}(\G) $ with $\hat{k}(X) = k$. 
\begin{figure}[t]
\centering

\begin{subfigure}{.48\textwidth}
  \centering
  \includegraphics[width=0.75\linewidth]{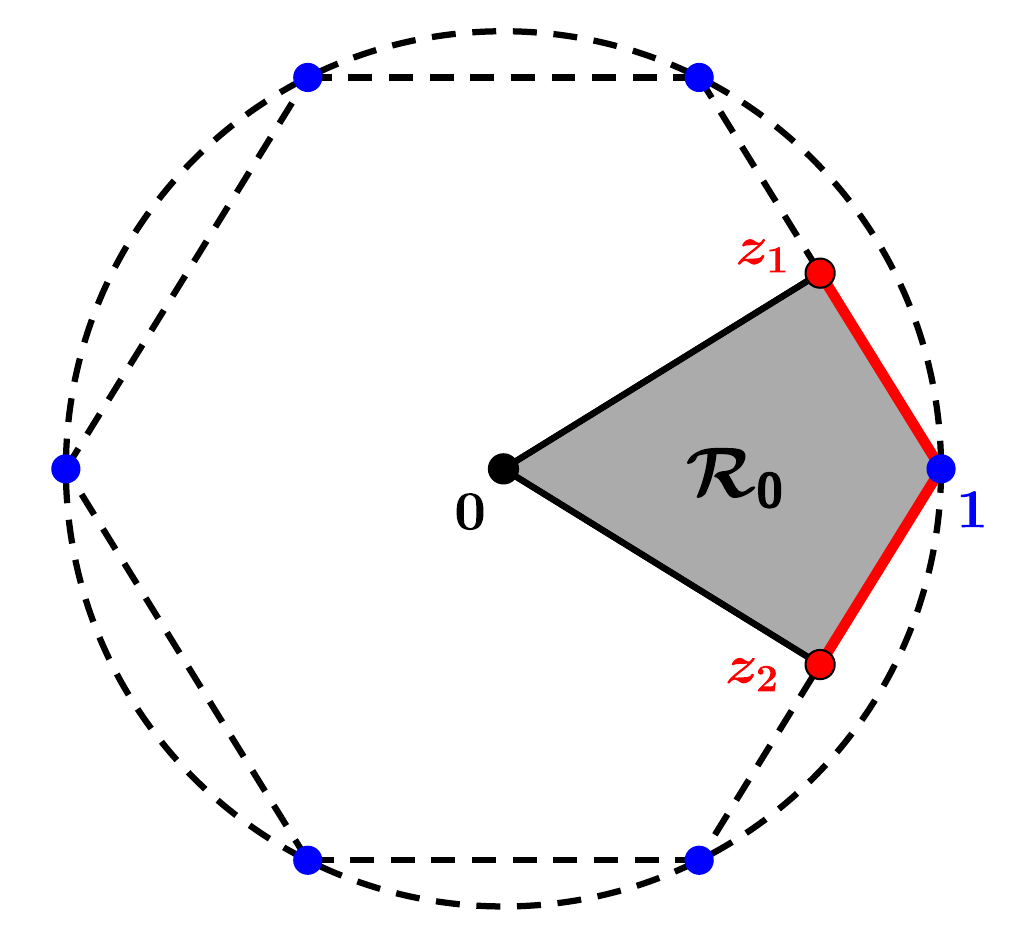}  
  \caption{$\mathcal{R}_0$}
  \label{fig:R0}
\end{subfigure}
\hfill
\begin{subfigure}{.48\textwidth}
  \centering
  \includegraphics[width=0.75\linewidth]{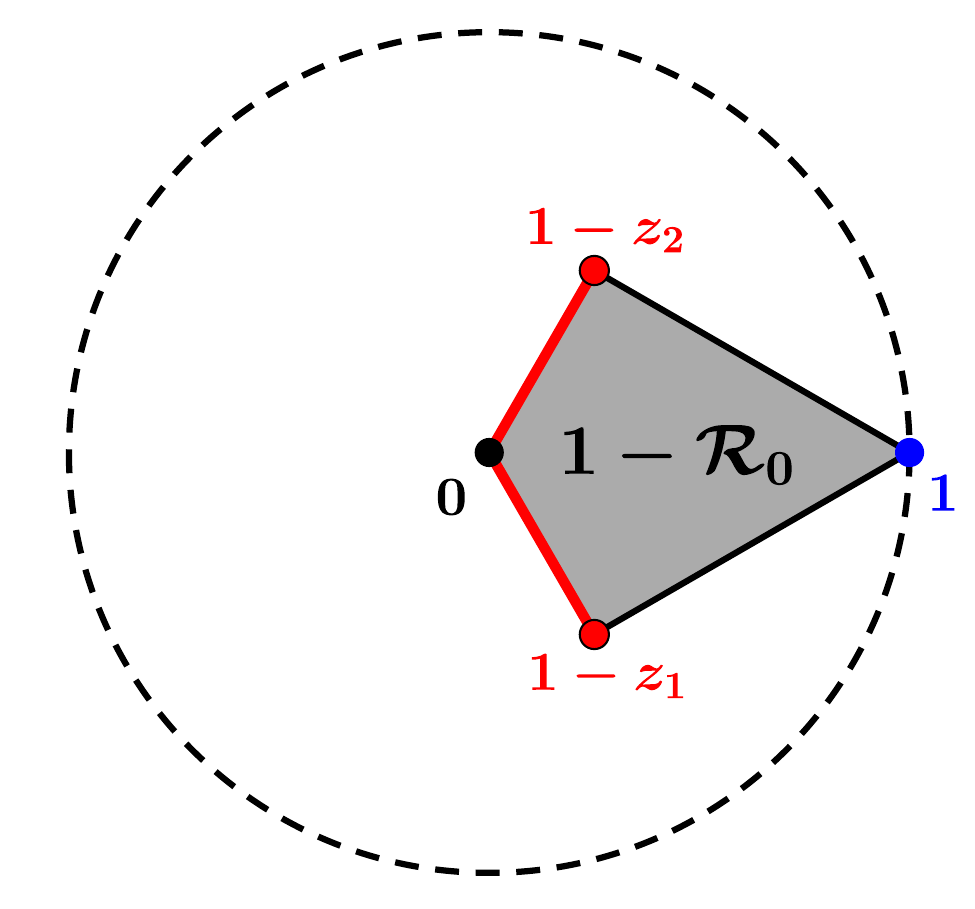}  
  \caption{$1 - \mathcal{R}_0$}
  \label{fig:1-R0}
\end{subfigure}
\caption{The regions $\mathcal{R}_0$ and $1 - \mathcal{R}_0$ associated with the cyclic group $\mathcal{Z}_6$. On panel~\subref{fig:R0}, the blue points are the group elements of $\mathcal{Z}_6$ and the points on the two red line segments are the maximizers of the ratio~\eqref{eq:ratio_R0}.} 
\label{fig:cyclic}
\end{figure}
By symmetry, we can assume without loss of generality that $\hat{k}(X) = 0$ and consider the shaded region $\mathcal{R}_0$ shown in Figure~\ref{fig:R0}, which is defined by 
\[ \mathcal{R}_0 := \left\lbrace z\in\mathbb{C}: z \in \text{conv}(\{z_{Q_0},\ldots,z_{Q_{m-1}}\}), \ 0 \in \argmin_{k \in \{0,\dots, m-1\}} \left| z - e^{\frac{2k \pi \mathbf{i}}{m}} \right| \right\rbrace. \]
Over the region $\mathcal{R}_0$, the ratio~\eqref{eq:ratio} reduces to
\begin{equation}\label{eq:ratio_R0}
\frac{\left| 1 - z \right|}{1-  \Re\left( z \right)}.
\end{equation}
Let $z_1$ and $z_2$ be the midpoints between 1 and its two neighboring group elements $e^{2\pi\mathbf{i}/m}$ and $e^{-2\pi \mathbf{i}/m}$, respectively.
We claim that any point on the line segments $[1 , z_1]\cup [1,z_2]$ (see Figure~\ref{fig:R0}) is a maximizer of the ratio~\eqref{eq:ratio_R0} over $\mathcal{R}_0$. To see this, we consider the polar representation $z' = |z'| \cdot e^{\mathbf{i}\phi}$ of the transformed variable $z' = 1- z$, which lies in the region $1 - \mathcal{R}_0$ (see Figure~\ref{fig:1-R0}). By the angle sum formula for the regular $m$-gon, we have $\phi \in \left[ -\pi \left( \frac{1}{2} - \frac{1}{m} \right), \pi \left( \frac{1}{2} - \frac{1}{m} \right) \right]$. It follows that
\begin{equation}\label{ineq:20}
\begin{split}
&\, \max_{z\in \mathcal{R}_0} \frac{\left| 1 - z \right|}{1-  \Re\left( z \right)} = \max_{z' \in 1-\mathcal{R}_0} \frac{|z'|}{\Re(z')}
 =  \max_{\phi \in \left[ -\pi \left( \frac{1}{2} - \frac{1}{m} \right), \pi \left( \frac{1}{2} - \frac{1}{m} \right) \right]} \frac{1}{\Re(e^{\mathbf{i} \phi})} \\
 = &\, \max_{\phi \in \left[ -\pi \left( \frac{1}{2} - \frac{1}{m} \right), \pi \left( \frac{1}{2} - \frac{1}{m} \right) \right]}\frac{1}{\cos\phi} \le \frac{1}{\cos \left( \frac{1}{2} - \frac{1}{m} \right)\pi} = \frac{1}{\sin \frac{\pi}{m}}.
 \end{split}
\end{equation}
This proves the claim since
\begin{equation*}
\frac{\left| 1 - z \right|}{1-  \Re\left( z \right)} = \frac{1}{\sin \frac{\pi}{m}}, \quad \forall z\in [1 , z_1]\cup [1,z_2].
\end{equation*}

Now, using \eqref{eq:2}, \eqref{eq:3}, and \eqref{ineq:20}, we get
\begin{equation*}
\text{\normalfont{dist}} (X, \mathcal{Z}_m) \le \frac{1}{\sqrt{2}\sin \frac{\pi}{m}}\Tr\left( I_d - X^\top \Pi_{\mathcal{Z}_m}(X) \right).
\end{equation*}
This shows that Condition~\ref{cond:alpha} holds for $\mathcal{Z}_m$ (where $m\ge3$) with $\alpha = \frac{1}{\sqrt{2}\sin \frac{\pi}{m}}$.

Next, we establish Condition~\ref{cond:beta} for $\mathcal{Z}_m$ with $m\ge1$. It is trivial to show that Condition~\ref{cond:beta} holds for $\mathcal{Z}_1$ with $\beta=1$. For $m\ge2$, since $\mathcal{Z}_m$ is a discrete subgroup of $\mathcal{O}(2)$, we compute
\[ \tau = \left| 1 - e^{\frac{2\pi \mathbf{i}}{m} }  \right| = \sqrt{ 2 - 2\cos \frac{2\pi}{m} } = 2 \sin \frac{2\pi}{m} \]
and apply Proposition~\ref{prop:cond_2_discrete} to conclude that Condition~\ref{cond:beta} holds for $\mathcal{Z}_m$ with 
\[ \beta = \frac{4 \sin^2 \frac{2\pi}{m}}{2(2)} = \sin^2 \frac{\pi}{m}.\]

Finally, let us derive an explicit formula for the projection $\Pi_{\mathcal{Z}_m}$, where $m\ge1$.
\begin{prop}\label{prop:cyclic_proj}
For any
\begin{equation*}
X = \begin{bmatrix}
x_{11} & x_{12} \\
x_{21} & x_{22}
\end{bmatrix} \in \mathbb{R}^{2\times 2},
\end{equation*}
define
\begin{equation*}
\theta = \begin{cases}
      \arccos \frac{x_{21} - x_{12}}{\sqrt{( x_{11} + x_{22})^2 + (x_{21} - x_{12})^2}},  \hfill & \text{ if } x_{11} + x_{22} \ge 0, \\
      \hfill 2\pi - \arccos \frac{x_{21} - x_{12}}{\sqrt{( x_{11} + x_{22})^2 + (x_{21} - x_{12})^2}},  \hfill & \text{ if } x_{11} + x_{22} < 0,
  \end{cases}
\end{equation*}
which always lies in $[0,2\pi)$.
Then, for any $m \ge 1$,
\begin{equation*}
\Pi_{\mathcal{Z}_m}(X) = 
\begin{cases}
      Q_{{\rm round}\left(\frac{m}{4} - \frac{m\theta}{2\pi} \right)}, \hfill & \text{ if } 0 \le \theta \le \frac{\pi}{2} + \frac{\pi}{m}, \\
      \hfill Q_{{\rm round}\left(  \frac{5m}{4} - \frac{m\theta}{2\pi} \right)}, \hfill & \text{ if } \frac{\pi}{2} + \frac{\pi}{m} < \theta < 2\pi, 
  \end{cases}
\end{equation*}
where ${\rm round}(\,\cdot\,)$ rounds a number to its closest integer and $Q_k$ is defined in \eqref{eq:6}.
\end{prop}
\begin{proof}
The result is trivial when $m = 1$. Thus, we assume that $m \ge 2$. 
Then, we have
\begin{equation*}
\begin{split}
\argmin_{k \in \{0,\dots, m-1\}} \nm{ X - Q_k }_F^2 & = \argmax_{k \in \{0,\dots, m-1\}} \left\langle X , Q_k \right\rangle \\
& = \argmax_{k \in \{0,\dots, m-1\}} \left\{ (x_{11} + x_{22}) \cos  \frac{2k\pi}{m} + (x_{21} - x_{12}) \sin  \frac{2k\pi}{m} \right\} \\
& =  \argmax_{k \in \{0,\dots, m-1\}} \sin\left( \theta + \frac{2k\pi}{m} \right).
\end{split}
\end{equation*}
Since $\theta$ can take any value in $[0, 2\pi)$ and $\frac{2k\pi}{m}$ lies in $[0, \frac{2(m-1)\pi}{m} ]$, we have 
\begin{equation*}
0 \le \theta + \frac{2k\pi}{m} \le 2\pi +\frac{2(m-1)\pi}{m} < 4\pi.
\end{equation*}
Moreover, the function $\varphi \mapsto \sin\varphi$ has two peaks in $\left[0, 2\pi +\frac{2(m-1)\pi}{m} \right)$, namely at $\varphi = \frac{\pi}{2}$ and $\varphi = \frac{5\pi}{2}$. It follows that
\begin{equation}\label{eq:5}
\argmin_{k \in \{0,\dots, m-1\}} \nm{ X - Q_k }_F^2  = \argmin_{k \in \{0,\dots, m-1\}} \min\left\lbrace \left| \theta + \frac{2k\pi}{m} - \frac{\pi}{2} \right|, \left| \theta + \frac{2k\pi}{m} - \frac{5\pi}{2} \right| \right\rbrace .
\end{equation}

We first consider the case where $0 \le \theta \le \frac{\pi}{2} + \frac{\pi}{m}$. Note that
\begin{equation*}
\theta + \frac{2k\pi}{m} \le \frac{\pi}{2} + \frac{\pi}{m} + \frac{2(m-1)\pi}{m} = \frac{5\pi}{2} - \frac{\pi}{m},
\end{equation*}
which implies that
\begin{equation*}
\left| \theta + \frac{2k\pi}{m} - \frac{5\pi}{2} \right| \ge \frac{\pi}{m}.
\end{equation*}
Also, since
\begin{equation*}
\frac{m}{4} - \frac{m\theta}{2\pi} \ge \frac{m}{4}- \frac{m}{2\pi}\left( \frac{\pi}{2} + \frac{\pi}{m} \right) = -\frac{1}{2},
\end{equation*}
we have $\text{round}\left( \frac{m}{4} - \frac{m\theta}{2\pi} \right) \in \{0,\dots,m-1\}$. Setting $k = \text{round}\left( \frac{m}{4} - \frac{m\theta}{2\pi} \right)$, we get
\begin{equation*}
\left| k - \left(\frac{m}{4} - \frac{m\theta}{2\pi} \right) \right| \le \frac{1}{2}
\end{equation*}
and hence
\begin{equation*}
 \left| \theta + \frac{2k\pi}{m} - \frac{\pi}{2} \right| \le \frac{\pi}{m}.
\end{equation*}
These, together with \eqref{eq:5}, yield
\begin{align*}
&\, \argmin_{k \in \{0,\dots, m-1\}} \nm{ X - Q_k }_F^2  = \argmin_{k \in \{0,\dots, m-1\}}  \left| \theta + \frac{2k\pi}{m} - \frac{\pi}{2} \right| \\
= &\, \argmin_{k \in \{0,\dots, m-1\}}  \left|k - \left(\frac{m}{4} - \frac{m\theta}{2\pi} \right) \right| = \text{round}\left(\frac{m}{4} - \frac{m\theta}{2\pi} \right).
\end{align*}

Next, we consider the case where $\frac{\pi}{2} + \frac{\pi}{m} < \theta < 2\pi$. Note that
\begin{equation*}
\theta + \frac{2k\pi}{m} - \frac{\pi}{2} > \frac{\pi}{2} + \frac{\pi}{m} - \frac{\pi}{2} = \frac{\pi}{m},
\end{equation*}
which implies that
\begin{equation*}
\left|  \theta + \frac{2k\pi}{m} - \frac{\pi}{2} \right| \ge \frac{\pi}{m}.
\end{equation*}
Also, since
\begin{equation*}
\frac{5m}{4} - \frac{m\theta}{2\pi} < \frac{5m}{4} -  \frac{m}{2\pi}\left( \frac{\pi}{2} + \frac{\pi}{m} \right) = m - \frac{1}{2},
\end{equation*}
we have $\text{round}\left( \frac{5m}{4} - \frac{m\theta}{2\pi} \right) \in \{0,\dots, m-1\}$. Setting $k = \text{round}\left( \frac{5m}{4} - \frac{m\theta}{2\pi} \right)$, we get
\begin{equation*}
\left| k - \left(  \frac{5m}{4} - \frac{m\theta}{2\pi} \right)  \right| \le \frac{1}{2}
\end{equation*}
and hence
\begin{equation*}
\left| \theta + \frac{2k\pi}{m} - \frac{5\pi}{2} \right| \le \frac{\pi}{m}.
\end{equation*}
These, together with \eqref{eq:5}, yield
\begin{align*}
&\, \argmin_{k \in \{0,\dots, m-1\}} \nm{ X - Q_k }_F^2  = \argmin_{k \in \{0,\dots, m-1\}} \left| \theta + \frac{2k\pi}{m} - \frac{5\pi}{2} \right| \\
= &\, \argmin_{k \in \{0,\dots, m-1\}}  \left| k - \left(  \frac{5m}{4} - \frac{m\theta}{2\pi} \right)  \right| = \text{round}\left(  \frac{5m}{4} - \frac{m\theta}{2\pi} \right).
\end{align*}
This completes the proof.
\end{proof}

\subsection{Erd{\H{o}}s-R{\'e}nyi Measurement Graphs}\label{sec:graph}
Recall from our discussion in Section~\ref{sec:master} that the parameter $\kappa$ can be viewed as a measure of the connectivity of the measurement graph $([n],E)$. In particular, when the measurement graph is complete, we have $\kappa = 0$, which shows that condition~\eqref{master_condi_ii} in the master theorem is satisfied. As it turns out, the condition can be satisfied by measurement graphs that are much sparser. In this subsection, we show that if the measurement graph is an \emph{Erd{\H{o}}s-R{\'e}nyi} random graph with \emph{observation rate} $p \ge \frac{c\log n}{n}$ for some constant $c>0$ --- \ie, the edge weights $\{w_{ij} : 1 \le i < j \le n\}$ are independent and identically distributed (i.i.d.) Bernoulli random variables with
\begin{equation*}
w_{ij} = \begin{cases}
      1, & \text{ with probability } p, \\
      0, & \text{ with probability } 1-p,
  \end{cases}\quad 1\le i < j \le n
\end{equation*}
--- then condition~\eqref{master_condi_ii} in the master theorem will be satisfied with high probability. More precisely, we have the following result:
\begin{thm}\label{thm:kappa}
Suppose that the measurement graph is an Erd{\H{o}}s-R{\'e}nyi random graph with observation rate $p \in (0,1]$. Then, there exist constants $c_1,c_2 > 0$ such that whenever $p \ge \frac{c_1 \log n}{n}$, we will have $\kappa \le \frac{1}{32}$ with probability at least $1-\frac{1}{n^{c_2}}$.
\end{thm}
\begin{proof}
Using the definitions of $D$, $W$, $F$, and $\kappa$ in Section~\ref{sec:master}, we compute
\begin{align*}
\kappa &= \nm{ D^{-1}W - \frac{1}{n}F } = \nm{ (\bar{D} \otimes I_d)^{-1} (\bar{W} \otimes I_d) - \frac{1}{n} (ee^\top \otimes I_d) } \\
&= \nm{ \left( \bar{D}^{-1} \otimes I_d \right)(\bar{W} \otimes I_d) - \frac{1}{n} (ee^\top \otimes I_d) } \\
&= \nm{ \left( \bar{D}^{-1}\bar{W} \right) \otimes I_d - \frac{1}{n}(ee^\top \otimes I_d) } \\
&= \nm{ \left( \bar{D}^{-1}\bar{W} - \frac{1}{n} ee^\top \right) \otimes I_d } \\
&= \nm{ \bar{D}^{-1}\bar{W} - \frac{1}{n} ee^\top },
\end{align*}
where the second line follows from the fact that $(A_1 \otimes A_2)^{-1} = A_1^{-1} \otimes A_2^{-1}$ for any invertible matrices $A_1$ and $A_2$, the third line follows from the fact that $(A_1 \otimes A_2)(A_3 \otimes A_4) = (A_1 A_3) \otimes (A_2 A_4)$ for any matrices $A_1,A_2,A_3,A_4$ with conformable dimensions, the fourth line follows from the bilinearity of the Kronecker product, and the last line follows from the fact that $\nm{ A_1 \otimes A_2 } = \nm{A_1} \cdot \nm{A_2}$ (these properties of the Kronecker product can be found in, \eg,~\cite[Chapter 4.2]{horn1991topics}). Now, we bound
\begin{align*}
\nm{ \bar{D}^{-1}\bar{W} - \frac{1}{n}ee^\top } &\le \nm{ I_n - \frac{1}{p(n-1)+1}\bar{D} } \cdot \nm{ \bar{D}^{-1}\bar{W} } + \frac{1}{p(n-1)+1} \nm{ \bar{W} - \mathbb{E}[\bar{W}] } \\
&\quad+ \nm{ \frac{1}{p(n-1)+1} \mathbb{E}[\bar{W}] - \frac{1}{n} ee^\top }.
\end{align*}
Since $\bar{D}^{-1}\bar{W}$ has non-negative entries and each of its rows sums to 1, we have $\nm{ \bar{D}^{-1} \bar{W} } \le 1$ by~\cite[Corollary 6.1.5]{horn1985matrix}. Moreover, observe that 
\[ \nm{ I_n - \frac{1}{p(n-1)+1}\bar{D} } = \max_{i \in [n]} \left| 1 - \frac{r_i}{p(n-1)+1} \right|, \]
where $r_i = \sum_{j=1}^n w_{ij} = 1 + \sum_{j \not= i} w_{ij}$ for $i=1,\ldots,n$. By Chernoff's inequality (cf.~\cite[Exercise 2.3.5]{vershynin2018high}), for any $t \in (0,1]$, we have 
\[ \Pr\left( \left| \sum_{j \not= i} w_{ij} - p(n-1) \right| \ge tp(n-1) \right) \le 2\exp\left( -\frac{p(n-1)t^2}{3} \right), \quad i = 1,\ldots,n. \]
It follows that for any $t \in (0,1]$,
\[ \Pr\left( \nm{ I_n - \frac{1}{p(n-1)+1} \bar{D} } \ge \frac{tp(n-1)}{p(n-1)+1} \right) \le 2n \cdot \exp\left( -\frac{p(n-1)t^2}{3} \right). \]
Next, by adapting the results in~\cite[Examples 3.14 and 6.8]{boucheron2013concentration} and~\cite[Corollary 3.6]{bandeira2016sharp}, we have, for any $p \ge \frac{\log n}{n}$, that
\begin{equation} \label{eq:ctr-adj-bd}
\Pr\left( \nm{ \bar{W} - \mathbb{E}[\bar{W}] } \ge 172\sqrt{pn} \right) \le \exp\left( -\frac{pn}{8} \right),
\end{equation}
see the proof of~\cite[Lemma 2]{lugosi2020concentration}. Lastly, since $\mathbb{E}[\bar{W}] = (1-p)I_n + p \cdot ee^\top$, we have
\begin{align*}
\nm{ \frac{1}{p(n-1)+1} \mathbb{E}[\bar{W}] - \frac{1}{n} ee^\top } &= \nm{ \frac{1-p}{p(n-1)+1} I_n + \left( \frac{p}{p(n-1)+1} - \frac{1}{n} \right) ee^\top } \\
&= \frac{1-p}{p(n-1)+1}.
\end{align*}
Upon setting $t=\frac{1}{66}$ and assuming that $p \ge \frac{(172 \times 66)^2 \log n}{n}$, we conclude that
\[ \kappa \le \frac{tp(n-1)}{p(n-1)+1} + \frac{172 \sqrt{pn}}{p(n-1)+1} + \frac{1-p}{p(n-1)+1} < \frac{1}{32} \]
with probability at least $1 - \frac{1}{n^{9000}}$.
\end{proof}


\subsection{Additive Sub-Gaussian Noise Model} \label{sec:noise}
Recall that under the additive noise model, the measurements are given by
\begin{equation} \label{eq:additive_noise}
C_{ij} = G_i^* {G_j^*}^\top + \Theta_{ij}, \quad (i,j)\in E. 
\end{equation}
The purpose of this subsection is to show that condition~\eqref{master_condi_iii} in the master theorem holds for a large class of noise matrices $\{\Theta_{ij} : (i,j) \in  E\}$.
Specifically, we focus on the case where $\{\Theta_{ij} : (i,j) \in E\}$ is a collection of independent random matrices whose entries are i.i.d. \emph{sub-Gaussian} random variables.
\begin{defi}[Sub-Gaussian Random Variable] \label{def:subG}
A random variable $\xi$ is said to be \emph{sub-Gaussian with parameter $\sigma > 0$} if
\begin{equation*}
\mathbb{E}[ \exp(\lambda\xi) ] \le \exp\left( \frac{\lambda^2\sigma^2}{2} \right), \quad \lambda \in \mathbb{R}.
\end{equation*}
\end{defi}
\noindent It can be shown that if $\xi$ is sub-Gaussian in the above sense, then it necessarily satisfies $\mathbb{E}[\xi]=0$. The class of sub-Gaussian random variables is rich. It contains, for example, the Gaussian, uniform, Bernoulli, and any bounded random variables, see~\cite[Example 2.5.8]{vershynin2018high}. In particular, a Gaussian random variable with mean zero and standard deviation $\sigma > 0$ is sub-Gaussian with parameter $\sigma$.

We first establish a tail inequality for the operator norm of the block matrix $\Delta$ (see~\eqref{eq:Delta-def} for the definition), which will be useful for verifying condition~\eqref{master_condi_iii} in the master theorem.


\begin{prop}\label{prop:Delta_norm}
Suppose that the measurement graph is an Erd{\H{o}}s-R{\'e}nyi random graph with observation rate $p \in (0,1]$. Let $\{\Theta_{ij} : 1 \le i < j \le n\}$ be independent noise matrices that are independent of the measurement graph and whose entries are i.i.d. sub-Gaussian random variables with parameter $\sigma>0$. Then, there exist constants $c_0, c_1 , c_2> 0$ such that whenever $p\ge \tfrac{c_0 (\log n)^2}{n}$, we have
\[ \Pr\left( \| \Delta \| \ge c_1\sigma\sqrt{pnd} \right) \le \frac{c_2}{n}. \]
If the sub-Gaussian entries $\{\Theta_{ij}: 1\le i < j\le n\}$ are zero-mean Gaussian with standard deviation $\sigma >0$, then the same inequality holds under the weaker requirement $p \ge \tfrac{c_0 \log n}{n}$.
\end{prop} 

\begin{proof}
Using~\eqref{eq:Delta-def} and~\eqref{eq:additive_noise}, we can write each block in $\Delta$ as
\begin{equation*}
[\Delta]_{ij} = w_{ij} \left( [C]_{ij} - \left[G^* {G^*}^\top  \right]_{ij} \right) = w_{ij} \left( C_{ij} - G_i^* {G_j^*}^\top \right) = w_{ij}  \Theta_{ij}, \quad 1 \le i < j \le n.
\end{equation*}
Therefore, conditioning on the measurement graph (\ie, on the values of the random variables $\{w_{ij} : 1 \le i < j \le n\}$), $\Delta$ is an $nd \times nd$ symmetric matrix whose upper triangular entries are independent random variables. 
Since the operator norm $\|\cdot\|$ is a convex, 1-Lipschitz function in the matrix entries, by Talagrand's inequality~\cite{ledoux2001concentration}, there exist constants $c_0, c_1 > 0$ such that
\begin{equation}\label{ineq:33}
\begin{split}
&\, \Pr\left( \| \Delta \| \ge \mathbb{E}[\|\Delta\|] + s \sigma \sqrt{\log nd} \,\Bigg|\, \|\Delta\|_{\infty} \le  3\sigma \sqrt{\log nd} \text{ and } \max_{i \in[n]} \sum_{j>i} w_{ij} \le epn \right)\\
 \le& \,  c_0 \exp\left( - c_1 s^2 \right).
\end{split}
\end{equation}
Using~\cite[Corollary~3.3]{bandeira2016sharp}, we obtain
\begin{equation}\label{ineq:34}
\begin{split}
&\, \mathbb{E}\left[\|\Delta\|\,\Bigg|\, \|\Delta\|_{\infty} \le  3\sigma \sqrt{\log nd} \text{ and } \max_{i \in[n]} \sum_{j>i} w_{ij} \le epn \right]  \\
\le& \, c_2 \sigma \left(  \max_{i\in [n]} \sqrt{d \sum_{ j >i} w_{ij} } + \sqrt{\log nd} \right) \le  c_2 \sigma \left(  \sqrt{e pnd } + \sqrt{\log nd} \right)
\end{split}
\end{equation}
for some constant $c_2 >0$.
Substituting~\eqref{ineq:34} into~\eqref{ineq:33} and taking $s = \sqrt{ \tfrac{\log nd }{c_1}}$,
we find that for some constants $c_3,c_4 > 0$, whenever $p \ge \frac{c_3 (\log n)^2}{n}$,
\begin{equation}
\label{ineq:35}
\begin{split}
\Pr\left( \|\Delta\| \ge c_4 \sigma \sqrt{pnd} \,\Bigg|\, \|\Delta\|_{\infty} \le 3 \sigma \sqrt{\log nd} \text{ and } \max_{i \in[n]} \sum_{j>i} w_{ij} \le epn  \right) \le \frac{c_0}{nd}.
\end{split}
\end{equation}

We now bound the probabilities of the two conditioned events.
Since $([\Theta]_{12})_{11}$ is a sub-Gaussian random variable with parameter $\sigma>0$, we can show by using the Markov inequality that 
\begin{equation*}
\Pr\left(  ([\Delta]_{12})_{11}  > 3 \sigma \sqrt{\log nd} \right) \le \Pr\left(  ([\Theta]_{12})_{11}  > 3 \sigma \sqrt{\log nd} \right)  \le \frac{1}{n^4 d^4},
\end{equation*}
which, by the union bound, implies that
\begin{equation}\label{eq:Delta_bound}
\Pr\left(  \|\Delta \|_\infty > 3 \sigma \sqrt{\log nd} \right) \le \frac{2}{n^2 d^2}.
\end{equation}

Next, since $\{w_{ij} : 1 \le i < j \le n\}$ are i.i.d. Bernoulli random variables with parameter~$p$, by the union bound and Chernoff's inequality~\cite[Theorem 2.3.1]{vershynin2018high}, there exist constants $c_5, c_6 > 0$ such that whenever $p \ge \frac{c_5 (\log n)^2}{n}$,
\begin{equation} \label{eq:edge-max}
\Pr\left( \max_{i \in [n]} \sum_{j>i} w_{ij} \ge epn \right) \le \sum_{i=1}^n \Pr\left( \sum_{j>i} w_{ij} \ge epn \right) \le n \cdot \exp(-p(n-1)) \le \frac{c_6}{n}.
\end{equation}
The desired inequality then follows by combining \eqref{ineq:35}, \eqref{eq:Delta_bound}, and \eqref{eq:edge-max}.

The last claim under the Gaussian assumption can be proved by similarly conditioning on the event
\[ \max_{i \in [n]} \sum_{j>i} w_{ij} \ge epn \]
and using \cite[Corollary~3.9]{bandeira2016sharp}. This completes the proof.
\end{proof}

The following theorem, which is a substantial generalization of \cite[Proposition 3.3]{bandeira2017tightness}, shows that under the setting of Proposition~\ref{prop:Delta_norm}, condition~\eqref{master_condi_iii} in Theorem~\ref{thm:master} will be satisfied with high probability.

\begin{thm}\label{thm:Delta_cond_random_graph}
Consider the setting of Proposition~\ref{prop:Delta_norm} and let $\alpha\ge1$ be arbitrary. Then, there exist constants $c_0, c_1, c_2 > 0$ such that whenever $p \ge \tfrac{c_0 (\log n)^2}{n}$ and $\sigma \le \frac{c_1\sqrt{pn}}{\alpha d}$, we have
\begin{equation*}
\Pr\left(\nm{D^{-1} \Delta} \le \frac{1}{32}\quad\text{and}\quad \nm{ \Pi^n \left( G^* + 2 D^{-1} \Delta G^* \right) - G^* }_F \le \frac{(\sqrt{2} -1)\sqrt{n}}{48\alpha}\right) \ge 1 - \frac{c_2}{n}.
\end{equation*}
If the sub-Gaussian entries $\{\Theta_{ij}: 1\le i < j\le n\}$ are zero-mean Gaussian with standard deviation $\sigma >0$, then the same inequality holds under the weaker requirement $p \ge \tfrac{c_0 \log n}{n}$.
\end{thm}

\begin{proof}
Upon applying Chernoff's inequality~\cite[Exercise 2.3.2]{vershynin2018high}, we have
\[ \Pr\left( r_i \le 1+\frac{p(n-1)}{2} \right) \le \left(\frac{2}{e}\right)^{p(n-1)/2}, \quad i=1,\ldots,n. \]
The above inequality and the union bound imply that for some constants $c_0, c_1 > 0$, whenever $p \ge \tfrac{c_0(\log n)^2}{n}$ ($p \ge \tfrac{c_0\log n}{n}$ in the Gaussian case), we have
\[
\Pr\left( \nm{ D^{-1} } \ge \left( 1+\frac{p(n-1)}{2} \right)^{-1} \right) \le \sum_{i=1}^n \Pr\left( r_i \le 1+\frac{p(n-1)}{2} \right) \le n \cdot \left(\frac{2}{e}\right)^{p(n-1)/2} \le \frac{c_1}{n}.
\]
This, together with Proposition~\ref{prop:Delta_norm}, implies the existence of constants $c_2, c_3, c_4 >0$ such that whenever $p \ge \frac{c_2 (\log n)^2}{n}$ ($p \ge \tfrac{c_2\log n}{n}$ in the Gaussian case) and $\sigma \le \frac{c_3\sqrt{pn}}{\alpha d}$, we have
\begin{equation}\label{ineq:10b}
\nm{D^{-1} \Delta} \le \nm{D^{-1}} \cdot \nm{\Delta} \le \left( 1 + \frac{p(n-1)}{2} \right)^{-1} \| \Delta \| \le \frac{\sqrt{2} - 1}{192\alpha\sqrt{d}} \le \frac{1}{32} 
\end{equation}
with probability at least $1-\frac{c_4}{n}$.

Next, we bound $\nm{ \Pi^n \left( G^* + 2 D^{-1} \Delta G^* \right) - G^* }_F$. By inequality~\eqref{ineq:contraction}, we have
\[ \nm{ \Pi^n \left( G^* + 2 D^{-1} \Delta G^* \right) - G^* }_F \le 4 \nm{ D^{-1} \Delta G^* }_F . \]
Since $G^*$ has $n$ blocks, each of which is a $d\times d$ orthogonal matrix, we have $\nm{G^*}_F = \sqrt{nd}$. Using the inequality $\nm{D^{-1} \Delta G^* }_F \le \nm{D^{-1} \Delta } \cdot \nm{ G^* }_F $, the desired bound then follows from~\eqref{ineq:10b}.
\end{proof}

\subsection{Entropic Spectral Initialization}\label{sec:spectral_ini} 
In order for our non-convex approach to enjoy the theoretical guarantee offered by the master theorem, we need to initialize GPM by a point that has a sufficiently small estimation error. As it turns out, the geometry of the closed subgroup $\G$ contains much information that can be used to guide our construction of such a point. Specifically, by considering the error-bound geometry of $\G$ as encapsulated in Condition~\ref{cond:beta} and the classic notion of \emph{metric entropy} of the quotient $\mathcal{O}(d)/\G$ of the orthogonal group $\mathcal{O}(d)$ by the closed subgroup $\G$,\footnote{Note that the quotient $\mathcal{O}(d)/\G$ may not be a group in general as we do not require the subgroup $\G$ to be normal.} we design a novel initialization procedure for GPM that produces a point satisfying condition~\eqref{master_condi_iv} in the master theorem. Before we present our proposed procedure, let us introduce some basic results concerning the metric entropy of the quotient $\mathcal{O}(d)/\G$.

\subsubsection{The Quotient $\mathcal{O}(d)/\G$ and Its Metric Entropy}\label{sec:entropy}
Given any closed subgroup $\G$ of $\mathcal{O}(d)$ and any orthogonal matrix $Q\in\mathcal{O}(d)$, the \emph{(left-) coset} $[Q]$ of $\G$ in $\mathcal{O}(d)$ is defined by
\begin{equation*}
[Q] := \left\{ O\in \mathcal{O}(d): O = QQ' \text{ for some } Q' \in \G \right\}.
\end{equation*}
The quotient $\mathcal{O}(d)/\G$ is then defined as the set of all cosets of $\G$ in $\mathcal{O}(d)$. We can define a natural distance on $\mathcal{O}(d)/\G$ by
\[ 
\text{dist}([Q_1], [Q_2]) = \min_{Q' \in \G} \nm{ Q_1 - Q_2 Q' }_F.
\] 
It can be easily seen that this distance is independent of the choice of the class representatives $Q_1$ and $Q_2$ of the cosets $[Q_1]$ and $[Q_2]$, respectively. Moreover, it turns $\mathcal{O}(d)/\G$ into a compact (under the quotient topology) metric space.

We next introduce the concepts of \emph{net} and \emph{covering number}, see, \eg,~\cite{kolmogorov1959varepsilon}.
\begin{defi}[Net and Covering Number]\label{def:net}
Let $(\mathcal{S},\nu)$ be a compact metric space and $\epsilon > 0$ be a parameter. A subset $\mathcal{N}\subseteq \mathcal{S}$ is said to be an \emph{$\epsilon$-net} of $\mathcal{S}$ if for any point $x \in \mathcal{S}$, there exists a point $y \in \mathcal{N}$ such that $\nu(x,y) \le \epsilon$. The cardinality of the smallest $\epsilon$-net is called the \emph{$\epsilon$-covering number}, denoted by $N(\mathcal{S},\epsilon)$. 
\end{defi}

The compactness of $\mathcal{S}$ implies that $N(\mathcal{S},\epsilon)$ is finite for any $\epsilon>0$. The following proposition offers an explicit and efficient construction of an $\epsilon$-net of the quotient $\mathcal{O}(d)/\G$.

\begin{prop}
\label{prop:entropy}
Let $\epsilon > 0$ and $Q_1,\dots, Q_K$ be random orthogonal matrices that are independently and uniformly distributed on $\mathcal{O}(d)$. Then, for any $O\in\mathcal{O}(d)$,
with probability at least $1 - \left( 1- N(\mathcal{O}(d)/\G,\tfrac{\epsilon}{2})^{-1} \right)^K$, we have
\[ \min_{k \in[K]} \text{\normalfont dist}([O], [Q_k]) \le \epsilon . \]
\end{prop}

\begin{proof}
Consider any $(\epsilon/2)$-net $\mathcal{N} = \left\lbrace[\bar{Q}_1], \dots, [\bar{Q}_{|\mathcal{N}|}]\right\rbrace$ of the quotient $\mathcal{O}(d)/\G$ and fix an arbitrary $O\in \mathcal{O}(d)$. Since the matrices $Q_1,\ldots,Q_K$ are independently and uniformly distributed on $\mathcal{O}(d)$, the cosets $[Q_1],\ldots,[Q_K]$ are independently and uniformly distributed on $\mathcal{O}(d)/\G$. Therefore, for any $k \in [K]$ and $\ell \in [|\mathcal{N}|]$,
\begin{equation*}
1 = \Pr\left( [Q_k] \in \bigcup_{\ell = 1}^{|\mathcal{N}|} \mathcal{B}([\bar{Q}_\ell], \tfrac{\epsilon}{2}) \right) \le \sum_{\ell = 1}^{|\mathcal{N}|} \Pr\left( [Q_k] \in \mathcal{B}([\bar{Q}_\ell], \tfrac{\epsilon}{2}) \right) = |\mathcal{N}| \cdot\Pr\left( [Q_k] \in \mathcal{B}([\bar{Q}_\ell], \tfrac{\epsilon}{2}) \right),
\end{equation*}
where $\mathcal{B}\left( [\bar{Q}_\ell],\tfrac{\epsilon}{2} \right) := \left\{ [Q]: {\rm dist}([\bar{Q}_\ell],[Q]) \le  \tfrac{\epsilon}{2} \right\}$ and the last equality follows from the fact that the distribution of $Q_k$ is invariant under multiplication by any fixed orthogonal matrix.
This gives
\begin{equation}\label{ineq:22}
\Pr\left( [Q_k] \in \mathcal{B}([\bar{Q}_\ell], \tfrac{\epsilon}{2}) \right) \ge |\mathcal{N}|^{-1}, \quad k = 1,\ldots,K; \, \ell = 1,\ldots,|\mathcal{N}|.
\end{equation}
Next, there exists some $\bar{\ell} \in [|\mathcal{N}|]$ such that $[O] \in \mathcal{B}([\bar{Q}_{\bar{\ell}}], \tfrac{\epsilon}{2})$. Hence, by~\eqref{ineq:22}, we have
\begin{align*}
&\, \Pr\left( [O] \in \mathcal{B}([Q_k], \epsilon)  \text{ for some } k \in [K] \right) = 1 - \prod_{k=1}^K\left( 1 - \Pr\left( [O] \in \mathcal{B}([Q_k], \epsilon) \right) \right)\\
\ge&\, 1 - \prod_{k=1}^K\left( 1 - \Pr\left(  [Q_k] \in \mathcal{B}([\bar{Q}_{\bar{\ell}}], \tfrac{\epsilon}{2})\right) \right)  \ge 1 -  (1 - |\mathcal{N}|^{-1})^K .
\end{align*}
Optimizing the lower bound over all possible $(\epsilon/2)$-nets $\mathcal{N}$ of $\mathcal{O}(d)/\G$ completes the proof.
\end{proof}

In view of Proposition~\ref{prop:entropy}, we are naturally interested in determining the covering number of the quotient $\mathcal{O}(d)/\G$, particularly when $\G$ is one of the four subgroups (\ie, $\mathcal{O}(d)$, $\mathcal{SO}(d)$, $\mathcal{P}(d)$, and $\mathcal{Z}_m$) we considered earlier. For $\G = \mathcal{O}(d)$, the quotient $\mathcal{O}(d)/\mathcal{O}(d)$ is the trivial group that contains only one element. Therefore, its $\epsilon$-covering number is $1$ for any $\epsilon >0$. For $\G = \mathcal{SO}(d)$, the quotient $\mathcal{O}(d)/\mathcal{SO}(d)$ is isomorphic to the Boolean group, which implies that its $\epsilon$-covering number is at most $ 2$ for any $\epsilon > 0$. In what follows, we provide an estimate of the covering number of the quotient $\mathcal{O}(d)/\G$ when $\G$ is a discrete subgroup of $\mathcal{O}(d)$. In particular, such an estimate applies to the cases of $\G=\mathcal{P}(d)$ and $\G=\mathcal{Z}_m$.

\subsubsection{Covering Number of the Quotient $\mathcal{O}(d)/\G$ for Discrete $\G$} 
To begin, let us introduce the concepts of \emph{packing} and \emph{packing number}, which are closely related to the concepts of net and covering number, respectively.
\begin{defi}[Packing and Packing Number]\label{def:packing}
Let $(\mathcal{S},\nu)$ be a compact metric space and $\epsilon > 0$ be a parameter. A subset $\mathcal{P}\subseteq \mathcal{S}$ is said to be an \emph{$\epsilon$-packing} of $\mathcal{S}$ if any two points $x,y \in \mathcal{P}$ satisfy $\nu(x,y) > \epsilon$. The cardinality of the largest $\epsilon$-packing is called the \emph{$\epsilon$-packing number}, denoted by $P(\mathcal{S},\epsilon)$.
\end{defi}
Given a compact metric space $(\mathcal{S},\nu)$ and a parameter $\epsilon>0$, we have the following relationship between the covering and packing numbers, see, \eg,~\cite[Inequality~(3)]{szarek1997metric}:
\begin{equation}\label{ineq:covering_packing}
N(\mathcal{S}, \epsilon) \le P(\mathcal{S}, \epsilon) \le N(\mathcal{S}, \tfrac{\epsilon}{2}).
\end{equation}

This inequality can then be used to establish the following result:
\begin{prop}\label{prop:discrete_entropy}
Let $\G$ be a discrete subgroup of $\mathcal{O}(d)$ and $\epsilon > 0$ be a given parameter.
Then, we have
\[ N( \mathcal{O}(d)/ \G ,\epsilon ) \le |\G|^{-1}\left(\frac{c}{\epsilon} \right)^{\frac{d(d-1)}{2}} \]
for some constant $c>0$.
\end{prop}
\begin{proof}
It suffices to consider the case where 
$\epsilon < \tau  $.
Let $\widetilde{\mathcal{P}}$ be a maximal $\epsilon$-packing of the quotient $\mathcal{O}(d)/ \G$, \ie, $ |\widetilde{\mathcal{P}}| = P(\mathcal{O}(d)/ \G , \epsilon) $. Then, for any distinct cosets $[Q_1], [Q_2] \in \widetilde{\mathcal{P}}$, we have
\begin{equation}\label{ineq:25}
\text{dist}([Q_1], [Q_2]) = \min_{Q'\in \G} \| Q_1 - Q_2 Q' \|_F > \epsilon.
\end{equation}
Consider the set $\mathcal{P} := \{ Q \in \mathcal{O}(d) : [ Q ] \in \widetilde{\mathcal{P}} \}$. We claim that $\mathcal{P}$ is an $\epsilon$-packing of $\mathcal{O}(d)$. To prove this, let $Q_1, Q_2\in \mathcal{P}$ be two distinct points. If $[Q_1] \neq [Q_2]$, then we have 
\[ \|Q_1 - Q_2\|_F \ge \min_{Q' \in \mathcal{G}} \| Q_1 - Q_2 Q' \|_F > \epsilon \]
by~\eqref{ineq:25}. If $[Q_1] = [Q_2]$, then $Q_1 = Q_2 Q'$ for some $Q' \in \mathcal{G} \setminus \{I_d\}$. It follows from Definition~\ref{defi:min_sep} that
\[ \| Q_1 - Q_2 \|_F = \| I_d - Q' \|_F > \epsilon. \]
Thus, the claim is established. Now, it is elementary to show that (i) for any $Q \in \mathcal{O}(d)$, we have $|[Q]| = |\G|$; (ii) for any two cosets $[Q_1],[Q_2] \in \mathcal{O}(d)/\G$, we either have $[Q_1]=[Q_2]$ or $[Q_1] \cap [Q_2] = \emptyset$. In particular, we see that $|\mathcal{P}| = |\widetilde{\mathcal{P}}| \cdot |\G| = P(\mathcal{O}(d)/\G, \epsilon) \cdot |\G|$. Consequently, there exists a constant $c>0$ such that
\begin{align*}
N( \mathcal{O}(d)/ \G ,\epsilon ) \cdot |\G| \le P( \mathcal{O}(d)/ \G ,\epsilon ) \cdot |\G| = | \mathcal{P} | \le P( \mathcal{O}(d), \epsilon ) \le N( \mathcal{O}(d), \tfrac{\epsilon}{2} ) \le \left(\frac{c}{\epsilon} \right)^{\frac{d(d-1)}{2}} ,
\end{align*}
where the first and third inequalities follow from \eqref{ineq:covering_packing}, the second inequality follows the fact that $\mathcal{P}$ is an $\epsilon$-packing, and the last inequality follows from~\cite[Theorem~7]{szarek1997metric}. This completes the proof.
\end{proof}

\subsubsection{Entropic Spectral Estimator}\label{subsec:ese}
We are now ready to develop our advertised initialization procedure for GPM. We will make use of the results in the previous subsection and the following variant of the Davis-Kahan Theorem~\cite{yu2015useful}.
\begin{thm}\label{thm:Davis-Kahan}
Let $Z, Z^* \in \RR^{M \times M}$ be symmetric matrices with eigenvalues $\lambda_1\ge \cdots \ge \lambda_M $ and $\lambda_1^* \ge \cdots \ge \lambda_M^* $, respectively.  For any integers $k,\ell$ such that $1\le k\le \ell\le M$, let $V, V^* \in \RR^{M\times (\ell - k + 1)}$ be the matrices defined by
\begin{equation*}
V = \begin{bmatrix}
| & & | \\
v_k & \cdots & v_{\ell}\\
| & & | 
\end{bmatrix} \quad\text{and}\quad
V^* = \begin{bmatrix}
| & & | \\
v_k^* & \cdots & v_{\ell}^*\\
| & & | 
\end{bmatrix},
\end{equation*}
where $v_i$ and $v_i^*$ are the eigenvectors of $Z$ and $Z^*$ corresponding to the eigenvalues $\lambda_i$ and $\lambda_i^*$, respectively, for $i \in [M]$.
Suppose that $\min\{\lambda_{k-1}^* - \lambda_k^*, \lambda_\ell^* - \lambda_{\ell+1}^* \}>0$, where $\lambda_0 = +\infty$ and $\lambda_{M+1} = -\infty$ by convention. Then, there exists an orthogonal matrix $Q^*\in \mathcal{O}( \ell- k + 1)$ such that
\begin{equation*}
\nm{ V Q^*  - V^* }_F \le \frac{ 2\sqrt{2} \min\{ \sqrt{\ell - k + 1} \nm{ Z - Z^* }, \nm{ Z - Z^* }_F \} }{\min\{\lambda_{k-1}^* - \lambda_k^*, \lambda_\ell^* - \lambda_{\ell+1}^* \}}.
\end{equation*}
\end{thm}

To begin, let $\G \subseteq \mathcal{O}(d)$ be the closed subgroup of interest and $\eta > 0$ be a given parameter. Consider taking $Z = C$, $Z^* = \eta \cdot G^* {G^*}^\top$, $k = 1$, $\ell = d$ in Theorem~\ref{thm:Davis-Kahan}, where, as before, $C \in \RR^{nd \times nd}$ is the matrix defined in~\eqref{eq:def-C} and $G^* \in \G^n$ is the ground truth. Note that each of the $d$ columns of $\frac{1}{\sqrt{n}}G^*$ is an eigenvector of $\eta \cdot G^* {G^*}^\top$ with eigenvalue $\eta n$. The remaining $(n-1)d$ eigenvalues of $\eta \cdot G^* {G^*}^\top$ are all equal to 0. Hence, by letting $V_C\in \RR^{nd\times d}$ to be the matrix whose $j$-th column is the eigenvector associated with the $j$-th largest eigenvalue of $C$ for $j = 1,\ldots,d$, we deduce from Theorem~\ref{thm:Davis-Kahan} the existence of a $Q^* \in \mathcal{O}(d)$ that satisfies
\begin{equation}\label{ineq:13}
\begin{split}
\nm{ V_C Q^* - \frac{1}{\sqrt{n}} G^*  }_F & \le \frac{ 2\sqrt{2} \min\left\lbrace \sqrt{d} \nm{ C - \eta \cdot G^* {G^*}^\top }, \nm{ C - \eta \cdot G^* {G^*}^\top }_F \right\rbrace }{\min\lbrace \lambda_{k-1}^* - \lambda_k^*, \lambda_\ell^* - \lambda_{\ell+1}^* \rbrace } \\ 
& \le \frac{ 2\sqrt{2d} \nm{ C - \eta \cdot G^* {G^*}^\top } }{\eta n } .
\end{split}
\end{equation}
This, together with the definition of $\varepsilon$ in~\eqref{eq:est_err}, Lemma~\ref{lem:1}, inequality~\eqref{ineq:contraction} and inequality~\eqref{ineq:13}, implies that $\Pi^n (V_C Q^*)$ enjoys the estimation error bound
\begin{equation}\label{ineq:29}
\begin{split}
\varepsilon( \Pi^n (V_C Q^*) ) &\le \nm{ \Pi^n ( V_C Q^*) - G^* }_F  = \nm{ \Pi^n ( \sqrt{n} \cdot V_C Q^*) - G^* }_F \\
& \le 2 \nm{  \sqrt{n} \cdot V_C Q^* - G^* }_F  \le  \frac{ 4\sqrt{2d} }{ \eta \sqrt{n} } \nm{ C - \eta \cdot G^* {G^*}^\top }.
\end{split}
\end{equation}

Unfortunately, the estimator $\Pi^n (V_C\, Q^*)$ is not implementable as we do not know $Q^*$ in general. To work around this, let us construct an approximation $\widetilde{Q}$ of the unknown $Q^*$ as follows. Suppose that we have a finite subset $\mathcal{Q}\subseteq \mathcal{O}(d)/\G$ satisfying
\begin{equation}\label{ineq:30}
\min_{[Q] \in \mathcal{Q}} \text{dist} ([Q], [Q^*]) = \min_{[Q] \in \mathcal{Q}} \min_{Q'\in \G} \| Q - Q^* Q' \|_F \le \epsilon.
\end{equation}
By definition, there must exist $[\widetilde{Q}] \in \mathcal{Q}$ and $Q'\in \G$ such that 
\begin{equation}\label{def:Q_tilde}
\nm{ \widetilde{Q} - Q^* Q' }_F \le \epsilon.
\end{equation}
We therefore consider the estimator
\begin{equation}\label{def:G_tilde}
\widetilde{G} = \Pi^n (V_C \widetilde{Q}).
\end{equation}
Let us now study the estimation performance of $\widetilde{G}$.
\begin{prop}\label{prop:G_tilde}
Let $ \eta > 0$, $\epsilon \ge 0$ be given parameters and $\mathcal{Q}\subseteq\mathcal{O}(d)/\G$ be a finite subset of equivalent classes satisfying~\eqref{ineq:30}. Consider the estimator $\widetilde{G}$ defined in~\eqref{def:G_tilde}. Then,
\begin{equation*}
\varepsilon (\widetilde{G}) \le 2 \sqrt{n} \epsilon + \frac{ 4\sqrt{2d} }{ \eta \sqrt{n} } \nm{ C - \eta \cdot G^* {G^*}^\top }.
\end{equation*}
\end{prop}
\begin{proof}[Proof]
We have
\begin{equation*}
\begin{split}
\varepsilon ( \widetilde{G} ) & = \min_{Q\in \G }\nm{ \widetilde{G} - G^* Q }_F \le \nm{ \Pi^n ( V_C \widetilde{Q} ) - G^*Q' }_F \le 2 \sqrt{n} \nm{ V_C \widetilde{Q} - \frac{1}{\sqrt{n}}G^* Q' }_F \\
& \le  2\sqrt{n} \left( \nm{ V_C \widetilde{Q} - V_C Q^* Q' }_F + \nm{ V_C Q^* Q' - \frac{1}{\sqrt{n}} G^* Q' }_F \right) \\
& \le 2\sqrt{n} \left( \nm{V_C} \cdot \nm{ \widetilde{Q} - Q^* Q' }_F + \nm{ V_C Q^* - \frac{1}{\sqrt{n}} G^* }_F \right) \\
& \le 2 \sqrt{n} \epsilon + \frac{ 4\sqrt{2d} }{ \eta \sqrt{n} } \nm{ C - \eta \cdot G^* {G^*}^\top },
\end{split}
\end{equation*}
where the first inequality follows from~\eqref{def:G_tilde}; the second inequality follows from Lemma~\ref{lem:1} and inequality~\eqref{ineq:contraction}; the last inequality follows from~\eqref{ineq:13},~\eqref{def:Q_tilde}, and the fact that $\| V_C \| \le 1$. This completes the proof.
\end{proof}
Compared with the bound~\eqref{ineq:29} on $\varepsilon( \Pi^n (V_C Q^*) )$, we see that the bound on $\varepsilon ( \widetilde{G} )$ has an extra term that is on the order of $\sqrt{n}\epsilon$. This can be attributed to the error incurred when using an element of an $\epsilon$-net of the quotient $\mathcal{O}(d)/\G$ to approximate the unknown element $[Q^*]$. 
However, since the estimator $\widetilde{G}$ relies on an element $\widetilde{Q}$ that validates inequality~\eqref{def:Q_tilde}, which involves the unknown orthogonal matrix $Q^*$, it is still not implementable. Nevertheless, the above idea can be further developed to construct another estimator that not only is implementable but also enjoys a good estimation error bound. Specifically, motivated by the objective function of the least squares formulation~\eqref{opt:LS}, let us consider the function $\psi: \mathcal{O}(d)/\G \to \RR$ defined by
\[ \psi ([Q]) := \langle C \Pi^n (V_C Q) , \Pi^n (V_C Q) \rangle .\]
Since $\Pi^n (V_C Q Q') = \Pi^n (V_C Q)Q'$ for any $Q'\in \mathcal{G}$, the function $\psi$ is independent of the choice of the representative $Q$ of the equivalent class $[Q]$. This shows that $\psi$ is a well-defined function on $\mathcal{O}(d)/\G$. Partly inspired by the work~\cite{wu2017sdr}, in which a randomized rounding scheme for approximating the optimal solution to certain robust non-convex quadratic optimization problem is analyzed using an $\epsilon$-net of the sphere, we propose the following estimator:
\begin{algorithm}[H]
\caption{Entropic Spectral Estimator} \label{alg:spec}
\begin{algorithmic}[1]
\State \textbf{Input:} the matrix $C$ and a finite subset $\mathcal{Q}\subseteq \mathcal{O}(d)/\G$ satisfying~\eqref{ineq:30} for some $\epsilon \ge 0$
\State find the maximizer $[\widehat{Q}]$ of the problem 
\begin{equation*}
\max_{[Q] \in \mathcal{Q}} \psi ([Q])
\end{equation*}
\State \textbf{Output:} the matrix $\widehat{G} = \Pi^n (V_C \widehat{Q})$
\end{algorithmic}
\end{algorithm}

A key difference between the estimators $\widetilde{G}$ and $\widehat{G}$ is that the former requires the knowledge of an element $\widetilde{Q}$ satisfying inequality~\eqref{def:Q_tilde}, whereas the latter works as long as there exists one and we do not need to know which element it is.

Let us discuss how to construct a subset $\mathcal{Q}\subseteq \mathcal{O}(d)/\G$ satisfying~\eqref{ineq:30}. 
When $\G = \mathcal{O}(d)$, there is only one equivalent class, \viz, the orthogonal group $\mathcal{O}(d)$ itself. In this case, the entropic spectral estimator $\widehat{G}$ reduces to the spectral estimator in the works \cite{boumal2016nonconvex} and \cite{pachauri2013solving}. Therefore, the entropic spectral estimator can be seen as a generalization of these spectral estimators. When $\G = \mathcal{SO}(d)$, there are two equivalent classes: One formed by the set of orthogonal matrices with determinant $+1$ and the other with determinant $-1$. Taking the representatives $I_d$ and $\Diag(-1, 1,\dots, 1)$ for these two equivalent classes, respectively, the entropic estimator is either $\Pi^n(V_C)$ or $\Pi^n (V_C')$ with $V_C' = V_C \cdot \Diag(-1,1,\ldots,1)$, depending on whether $\langle C \Pi^n (V_C ) , \Pi^n (V_C ) \rangle \ge \langle C \Pi^n (V_C' ) , \Pi^n (V_C' ) \rangle$.
If $\G$ is a discrete subgroup of $\mathcal{O}(d)$, then we can find a desired subset $\mathcal{Q}$ by invoking Propositions~\ref{prop:entropy} and \ref{prop:discrete_entropy}. Specifically, these two propositions imply that there exists a constant $c>0$ such that given any $\rho \in (0,1)$, if we generate 
\begin{equation}
\label{eq:K}
K \ge \frac{\log \rho}{\log\left( 1 - |\G| \left( c \epsilon \right)^{\frac{d(d-1)}{2}} \right)}
\end{equation}
random orthogonal matrices $Q_1,\ldots,Q_K$ that are independently and uniformly distributed on $\mathcal{O}(d)$, then with probability at least $1-\rho$, the set $\mathcal{Q} = \{[Q_1], \dots, [Q_K]\}$ satisfies inequality~\eqref{ineq:30}.
%
%

Interestingly, Proposition~\ref{prop:entropy} reveals an intimate relation between our proposed estimator and the notion of \emph{metric entropy} (the logarithm of covering number): The smaller the metric entropy of the quotient $\mathcal{O}(d)/\G$, the fewer independent copies of uniformly distributed random orthogonal matrices we need to construct the subset $\mathcal{Q}$. This explains why we name our estimator $\widehat{G}$ the \emph{entropic spectral estimator}.


The main result of this subsection is the following theorem, which concerns the estimation performance of the entropic spectral estimator.
\begin{thm}
\label{thm:spec_init_general}
Suppose that the group $\G$ satisfies Condition~\ref{cond:beta} with parameter $\beta \in (0,1]$. Let $\eta >0$ and $\epsilon \ge 0$ be given parameters. 
Then, the entropic spectral estimator $\widehat{G}$ returned by Algorithm~\ref{alg:spec} satisfies
\[ \varepsilon (\widehat{G}) \le 2 \sqrt{\frac{2n}{\beta}} \epsilon + \frac{12 \sqrt{d}}{\beta \eta \sqrt{n}} \nm{ C - \eta\cdot G^* {G^*}^\top }. \]
\end{thm}

\noindent Theorem~\ref{thm:spec_init_general} shows that even though we do not have access to the element $[\widetilde{Q}] \in \mathcal{Q}$ defined in~\eqref{def:Q_tilde} and hence cannot construct the estimator $\widetilde{G}$ in~\eqref{def:G_tilde}, we can get hold of another element $[\widehat{Q}] \in \mathcal{Q}$ by maximizing the least squares-based function $\psi$ over $\mathcal{Q}$ and use it to construct the estimator $\widehat{G}$, whose estimation error bound is worse than that of the estimator $\widetilde{G}$ by roughly a factor of $\frac{1}{\sqrt{\beta}}$ (recall that the parameter $\beta \in (0,1]$ is related to the geometry of the subgroup $\G$, see Condition~\ref{cond:beta}). As shown in Theorem~\ref{thm:group}, for many groups of interest (such as the orthogonal group $\mathcal{O}(d)$, the special orthogonal group $\mathcal{SO}(d)$, and the cyclic group $\mathcal{Z}_m$), the parameter $\beta$ is a constant, which implies that the bounds on $\varepsilon (\widetilde{G})$ and $\varepsilon (\widehat{G})$ differ by at most a constant factor.

\begin{proof}[Proof of Theorem~\ref{thm:spec_init_general}]
Let $\widetilde{G}$ be defined as in \eqref{def:G_tilde}. By definition of the entropic spectral estimator $\widehat{G}$, 
\begin{equation}
\label{ineq:trace-max}
\Tr\big( \widehat{G}^\top C \widehat{G} \big) \ge \Tr\big( \widetilde{G}^\top C\widetilde{G} \big).
\end{equation}
Upon letting $\Delta_\eta = C - \eta\cdot G^* {G^*}^\top$, for any $Q\in \G$, we have
\begin{equation*}
\begin{split}
			&\,\Tr \left({\widehat{G}}^\top \Delta_\eta \widehat{G} \right) - \Tr \left({G^*}^\top \Delta_\eta G^* \right)  =  \Tr \left((\widehat{G} Q - G^*)^\top \Delta_\eta (\widehat{G} Q + G^*) \right) \\
			\leq &\,  \nm{ \widehat{G} Q - G^* }_F \cdot \nm{ \Delta_\eta } \cdot \left( \nm{ \widehat{G} Q }_F + \|G^*\|_F \right) \leq 2\sqrt{nd} \cdot \nm{ \Delta_\eta } \cdot \nm{ \widehat{G} Q - G^* }_F.
\end{split}
\end{equation*}
This implies that
\begin{equation}\label{ineq:trace-G}
\begin{split}
\Tr \left({\widehat{G}}^\top C \widehat{G} \right) & = \eta \cdot \nm{ {G^*}^\top \widehat{G} }^2_F +  \Tr \left({\widehat{G}}^\top \Delta_\eta \widehat{G} \right) \\
& \leq \eta \cdot \nm{ {G^*}^\top \widehat{G} }^2_F + \Tr \left({G^*}^\top \Delta_\eta G^* \right) + 2\sqrt{nd} \cdot \|\Delta_\eta\| \cdot \varepsilon(\widehat{G}).
\end{split}		
\end{equation}
Similarly, we can get
\begin{equation}\label{ineq:new:1}
\begin{split}
\Tr \left({\widetilde{G}}^\top C \widetilde{G} \right) & = \eta \cdot \nm{ {G^*}^\top \widetilde{G} }^2_F +  \Tr \left({\widetilde{G}}^\top \Delta_\eta \widetilde{G} \right) \\
& \geq \eta \cdot \nm{ {G^*}^\top \widetilde{G} }^2_F + \Tr \left({G^*}^\top \Delta_\eta G^* \right) - 2 \sqrt{nd} \cdot \|\Delta_\eta\| \cdot \varepsilon (\widetilde{G}).
\end{split}
\end{equation}
Now, for any $G \in \G^n$, we have the identity
\begin{equation}\label{eq:identity}
\varepsilon(G)^2 = \nm{ G - G^* Q_G }_F^2 = 2\Tr\left(n\cdot I_d - Q_G^\top {G^*}^\top G \right),
\end{equation}
where $Q_G \in \G \subseteq \mathcal{O}(d)$ is defined in~\eqref{def:Qt}. This yields
\[ 
\nm{ {G^*}^\top \widetilde{G} }_F^2 = \nm{ Q_{\widetilde{G}}^\top {G^*}^\top \widetilde{G} }_F^2 \ge \frac{1}{d}\left( \Tr\left( Q_{\widetilde{G}}^\top {G^*}^\top \widetilde{G} \right)\right)^2 = \left( n\sqrt{d} - \frac{1}{2\sqrt{d}} \varepsilon (\widetilde{G})^2 \right)^2,
\]
where the inequality follows from the Cauchy-Schwarz inequality. Upon substituting the above into \eqref{ineq:new:1}, we obtain
\begin{equation}\label{ineq:trace-G0}
		\begin{split}
			\Tr \left(\widetilde{G}^\top C \widetilde{G} \right) 
			& \geq \eta \left(n\sqrt{d} - \frac{1}{2\sqrt{d}} \varepsilon (\widetilde{G})^2 \right)^2 + \Tr \left({G^*}^\top \Delta_\eta G^* \right) - 2 \sqrt{nd} \cdot \|\Delta_\eta\| \cdot \varepsilon (\widetilde{G}) \\
			& \geq \eta n^2 d - \eta n \cdot \varepsilon (\widetilde{G})^2 + \Tr \left({G^*}^\top \Delta_\eta G^* \right) - 2 \sqrt{nd} \cdot \|\Delta_\eta\| \cdot \varepsilon (\widetilde{G}).
		\end{split}
\end{equation}
It follows from~\eqref{ineq:trace-max},~\eqref{ineq:trace-G}, and~\eqref{ineq:trace-G0} that
\begin{equation}\label{ineq:23}
\begin{split}
\eta \nm{ {G^*}^\top \widehat{G} }^2_F & \ge \Tr \left({\widehat{G}}^\top C \widehat{G} \right) - \Tr \left({G^*}^\top \Delta_\eta G^* \right) - 2\sqrt{nd} \cdot \|\Delta_\eta\| \cdot \varepsilon(\widehat{G}) \\
& \ge \Tr \left({\widetilde{G}}^\top C \widetilde{G} \right) - \Tr \left({G^*}^\top \Delta_\eta G^* \right) - 2\sqrt{nd} \cdot \|\Delta_\eta\| \cdot \varepsilon(\widehat{G}) \\
&\geq \eta n^2d - \eta n \cdot \varepsilon (\widetilde{G})^2 - 2 \sqrt{nd} \cdot \|\Delta_\eta\| \cdot \left(\varepsilon (\widehat{G}) + \varepsilon (\widetilde{G})\right).
\end{split}
\end{equation}
Next, using the definition of $Q_{\widehat{G}}$ (see~\eqref{def:Qt}), identity~\eqref{eq:identity}, Condition~\ref{cond:beta} (with $X=\frac{1}{n}{G^*}^\top\widehat{G}$), and inequality~\eqref{ineq:23}, we have
\begin{equation*}
\begin{split}
\frac{\beta}{2} \varepsilon(\widehat{G})^2 & = \beta \Tr \left( n\cdot I_d - Q_{\widehat{G}}^\top {G^*}^\top \widehat{G} \right) \le nd - \frac{1}{n} \nm{ {G^*}^\top \widehat{G} }_F^2 \\
& \le \varepsilon (\widetilde{G})^2 +  \frac{2\sqrt{d} \cdot \|\Delta_\eta\|}{\eta \sqrt{n}} \left(\varepsilon (\widetilde{G}) + \varepsilon (\widehat{G}) \right).
\end{split}
\end{equation*}
This, together with the fact that $\beta \in (0,1]$, gives
\begin{align*}
\left( \frac{\beta}{2} \varepsilon(\widehat{G}) - \frac{\sqrt{d} \cdot \|\Delta_\eta\|}{\eta \sqrt{n}} \right)^2 & = \frac{\beta^2}{4} \varepsilon(\widehat{G})^2 - \frac{\beta \sqrt{d} \cdot \|\Delta_\eta\|}{\eta \sqrt{n}} \varepsilon(\widehat{G}) + \left( \frac{\sqrt{d} \cdot \|\Delta_\eta\|}{\eta \sqrt{n}} \right)^2 \\
& \le \frac{\beta}{2} \varepsilon (\widetilde{G})^2 + \frac{\beta\sqrt{d} \cdot \|\Delta_\eta\|}{\eta \sqrt{n}} \varepsilon(\widetilde{G}) + \left( \frac{\sqrt{d} \cdot \|\Delta_\eta\|}{\eta \sqrt{n}} \right)^2 \\
& \le \left( \sqrt{\frac{\beta}{2}} \varepsilon(\widetilde{G}) +  \frac{\sqrt{d} \cdot \|\Delta_\eta\|}{\eta \sqrt{n}} \right)^2.
\end{align*}
It follows that
\begin{align*}
\varepsilon (\widehat{G}) & \le \sqrt{\frac{2}{\beta}} \varepsilon(\widetilde{G}) + \frac{4\sqrt{d}}{\beta \eta \sqrt{n}}\|\Delta_\eta\| \le \sqrt{\frac{2}{\beta}}\left( 2 \epsilon \sqrt{n} + \frac{ 4\sqrt{2d} }{\eta \sqrt{n} } \nm{ \Delta_\eta } \right) + \frac{4\sqrt{d}}{\beta \eta \sqrt{n}}\|\Delta_\eta\| \\
& \le 2 \sqrt{\frac{2n}{\beta}} \epsilon + \frac{4\sqrt{d}}{\eta\sqrt{n}} \left( 2\sqrt{\frac{1}{\beta}} + \frac{1}{\beta} \right) \|\Delta_\eta\| \le 
2 \sqrt{\frac{2n}{\beta}} \epsilon + \frac{12 \sqrt{d}}{\beta \eta \sqrt{n}} \|\Delta_\eta\|,
\end{align*}
where the second inequality follows from Proposition~\ref{prop:G_tilde}. This completes the proof.
\end{proof}

Armed with Theorem~\ref{thm:spec_init_general}, we now show that when the measurement graph and additive noise follow the settings in Sections~\ref{sec:graph} and~\ref{sec:noise}, respectively, condition~\eqref{master_condi_iv} in the master theorem will be satisfied with high probability.
\begin{thm}\label{thm:init}
Suppose that (i) the group $\G$ satisfies Conditions~\ref{cond:alpha} and~\ref{cond:beta} with parameters $\alpha \ge 1$ and $\beta \in (0,1]$, respectively; (ii) the measurement graph is an Erd{\H{o}}s-R{\'e}nyi random graph with observation rate $p \in (0,1]$; (iii) the noise matrices $\{\Theta_{ij} : 1 \le i < j \le n\}$ are independent of each other and of the measurement graph, and whose entries are i.i.d. sub-Gaussian random variables with parameter $\sigma>0$. 
Then, there exist constants $c_0,c_1,c_2,c_3 >0$ such that when $p \ge   c_0 \cdot \max\left\{ \frac{\alpha^2d}{\beta^2n}, \frac{(\log n)^2}{n} \right\}$ and $\sigma \le \frac{c_1\beta\sqrt{pn}}{\alpha d }$, 
the entropic spectral estimator $\widehat{G}$ generated by Algorithm~\ref{alg:spec} with $\epsilon \le \frac{c_2\sqrt{\beta}}{\alpha}$ will satisfy
\begin{equation*}
\Pr\left( \varepsilon ( \widehat{G} ) \le \frac{\sqrt{n}}{2\alpha}  \right) \ge  1 - \frac{c_3}{n} .
\end{equation*}
If the sub-Gaussian entries $\{\Theta_{ij}: 1\le i < j\le n\}$ are zero-mean Gaussian with standard deviation $\sigma >0$, then the same inequality holds under the weaker requirement $p \ge   c_0 \cdot \max\left\{ \frac{\alpha^2d}{\beta^2n}, \frac{\log n}{n} \right\}$.
\end{thm}

\begin{proof}
Let $C^*\in \mathbb{R}^{nd\times nd}$ be the block matrix defined by $[C^*]_{ij} = w_{ij} G^*_i {G^*_j}^\top$, $i,j = 1,\dots, n$ and $\Delta^* := C^* - p \cdot G^* {G^*}^\top$. Using the definition of $\Delta$ in~\eqref{eq:Delta-def}, we can write 
\[ C - p \cdot G^* {G^*}^\top = \Delta^* + C - C^* = \Delta^* + \Delta. \]
Thus, by invoking Theorem~\ref{thm:spec_init_general} with $\eta = p$ and $\epsilon \le \frac{\sqrt{\beta}}{ 8\sqrt{2} \alpha }$, we have
\begin{equation} \label{eq:est-err-prebd}
\varepsilon (\widehat{G}) \le 2 \sqrt{\frac{2n}{\beta}} \epsilon + \frac{12 \sqrt{d}}{\beta p \sqrt{n}} \nm{ C - p \cdot G^* {G^*}^\top } \le \frac{\sqrt{n}}{4\alpha} + \frac{12 \sqrt{d}}{\beta p \sqrt{n}} ( \|\Delta^*\| + \|\Delta\| ).
\end{equation}
Let us now bound $\|\Delta^*\|$ and $\|\Delta\|$ separately.

By Proposition~\ref{prop:Delta_norm}, there exist constants $c_0, c_1, c_2>0$ such that for any $p \ge \frac{c_0 (\log n)^2}{n}$ ($p \ge \tfrac{c_0\log n}{n}$ in the Gaussian case), we have 
\begin{align*}
\Pr\left( \|\Delta\| \ge c_1 \sigma\sqrt{pnd} \right) \le \frac{c_2}{n}.
\end{align*}
On the other hand, observe that
\[ \Delta^* = \left( \bar{W} - p \cdot ee^\top \right) \boxtimes \left( G^*{G^*}^\top \right), \]
where $\bar{W} \in \RR^{n \times n}$ is the matrix defined in Section~\ref{sec:master} and $\boxtimes$ denotes the block Kronecker product introduced in~\cite[Definition 2.5]{horn1992block}. Since $G^*{G^*}^\top$ is positive semidefinite with $\left[ G^*{G^*}^\top \right]_{ii} = I_d$ for $i=1,\ldots,n$, by~\cite[Corollary 4.5(b)]{horn1992block} and the identity $\bar{W} - p \cdot ee^\top = \bar{W} - \mathbb{E}[\bar{W}] + (1-p) I_n$, we have
\[ \| \Delta^* \| \le \nm{ \bar{W} - p \cdot ee^\top } \le 1 + \nm{ \bar{W} - \mathbb{E}[\bar{W}] }. \]
This, together with~\eqref{eq:ctr-adj-bd}, implies that for any $p \ge \frac{\log n}{n}$,
\[ \Pr\left( \| \Delta^* \| \ge 173\sqrt{pn} \right) \le \exp\left( -\frac{pn}{8} \right). \]
The above calculations yield the existence of constants $c_3, c_4, c_5 >0$ such that whenever
\[ p \ge c_3 \max\left\{ \frac{  \alpha^2 d }{\beta^2 n}, \frac{ \log n}{n} \right\}, \quad \sigma \le \frac{ c_4 \beta\sqrt{pn} }{  \alpha d  }, \]
we will have
\[ \frac{12 \sqrt{d}}{\beta p \sqrt{n}} \| \Delta^* \| \le \frac{\sqrt{n}}{8\alpha} \quad\text{and}\quad \frac{12 \sqrt{d}}{\beta p \sqrt{n}} \| \Delta \| \le \frac{\sqrt{n}}{8\alpha} \]
with probability at least $1-\frac{c_5}{n}$. Upon substituting these bounds into~\eqref{eq:est-err-prebd}, we obtain $\varepsilon (\widehat{G}) \le \frac{\sqrt{n}}{2\alpha}$.
\end{proof}

We remark that there are works studying the estimation performance of various spectral estimators for $\mathcal{O}(d)$-\sync~\cite{ling2020near}, $\mathcal{SO}(2)$-\sync~\cite{boumal2016nonconvex}, and $\mathcal{P}(d)$-\sync~\cite{pachauri2013solving,shen2016normalized,bajaj2018smac,ling2020near}, and it is worth comparing the results for these estimators with that for our entropic spectral estimator. For $\mathcal{O}(d)$-\sync, 
the work~\cite{ling2020near} establishes a bound on the estimation error, measured in the \emph{operator norm}, of each block $[\widehat{G}]_1,\ldots,[\widehat{G}]_n$ of its proposed spectral estimator $\widehat{G}$. By contrast, our work establishes a bound on the estimation error, measured in the \emph{Frobenius norm}, of the proposed entropic spectral estimator in its entirety. Although for $\mathcal{O}(d)$-\sync the blockwise error bound in~\cite[Theorem 3.1]{ling2020near} is generally sharper than the error bound in Theorem~\ref{thm:init} of our work, the former applies only to the setting of \emph{complete} measurement graph and additive \emph{Gaussian} noise, while the latter applies to the more general setting of additive \emph{sub-Gaussian} noise and \emph{Erd{\H{o}}s-R{\'e}nyi} measurement graph with observation rate $p$ that can go down to the order of $\frac{(\log n)^2}{n} $ (if we fix the group $\G$ and hence the parameters $d$, $\alpha$, and $\beta$). For $\mathcal{SO}(2)$-\sync, the work~\cite{boumal2016nonconvex} studies the estimation error of a spectral estimator under the additive noise model with a \emph{complete} measurement graph. Since such a setting is covered by that of Theorem~\ref{thm:spec_init_general}, we can compare the corresponding estimation error bounds. Recall from our discussion immediately following Theorem~\ref{thm:init} that we may take $\epsilon = 0$. Moreover, we have $\beta = \frac{1}{2}$ by Theorem~\ref{thm:group}. Thus, Theorem~\ref{thm:spec_init_general} (with $\epsilon=0$, $d=2$, $\beta=\frac{1}{2}$, $\eta=1$) and the definition of $\Delta$ in~\eqref{eq:Delta-def} imply that the estimation error of the entropic spectral estimator is at most on the order of $\frac{\|\Delta\|}{\sqrt{n}}$, which is the same as that of the spectral estimator in~\cite{boumal2016nonconvex}; see~\cite[Lemma 6]{boumal2016nonconvex}. Lastly, for $\mathcal{P}(d)$-\sync, the works~\cite{pachauri2013solving,shen2016normalized,bajaj2018smac,ling2020near} consider an outlier noise model, which is different from the additive noise model considered in our work. As such, the corresponding estimation error bounds cannot be compared directly.

\section{Estimation Error Bound}\label{sec:est_err}
The purpose of this section is to show that our approach enjoys near-optimal estimation error.
We start by proving a bound on the normalized deviation $[ D^{-1} \Delta G^* ]_i$, where $i \in [n]$, for a general measurement graph (not necessarily the Erd{\H{o}}s-R{\'e}nyi random graph).
\begin{lem}
\label{lem:deviation_bound}
Let $\{\Theta_{ij} : 1 \le i < j \le n\}$ be independent noise matrices whose entries are i.i.d. sub-Gaussian random variables with parameter $\sigma>0$. Then, for any $s \ge d$ and $i \in [n]$,
\begin{equation*}
\Pr\left( \nm{[ D^{-1} \Delta G^* ]_i}_F > \frac{\sigma s}{\sqrt{r_i}}  \right) \le \exp\left(- \frac{s^2}{2} \right).
\end{equation*}
\end{lem}

\begin{proof}
Consider a fixed $i \in [n]$. Note that 
\begin{align*}
[ D^{-1} \Delta G^* ]_i = \frac{1}{r_i} \sum_{j = 1}^n w_{ij} \Theta_{ij} G^*_j .
\end{align*}
Therefore, we have
\begin{align*}
\nm{[ D^{-1} \Delta G^* ]_i}_F \le  \frac{1}{r_i} \sum_{j=1}^n w_{ij} \nm{ \Theta_{ij} G^*_j }_F = \frac{1}{r_i} \sum_{j=1}^n w_{ij} \nm{ ( I_d \otimes G^*_j) \vect(\Theta_{ij}) }_F ,
\end{align*}
where $\vect(\Theta_{ij}) \in\mathbb{R}^{d^2}$ denotes the $d^2$-dimensional vector obtained by stacking the columns of $\Theta_{ij}$.
Since the entries of $\vect(\Theta_{ij})$ are zero-mean i.i.d. sub-Gaussian random variables with parameter $\sigma>0$ and
\[( I_d \otimes G^*_j)^\top ( I_d \otimes G^*_j) = ( I_d \otimes {G^*_j}^\top) ( I_d \otimes G^*_j) = I_d^2 \otimes ({G^*_j}^\top G^*_j) = I_{d^2},\]
it follows from \cite[Theorem~2.1]{hsu2012tail} that
\begin{align*}
\Pr\left( \nm{ ( I_d \otimes G^*_j) \vect(\Theta_{ij}) }_F^2 > \sigma^2 (d^2 + 2d\sqrt{t}  + 2t)  \right) \le \exp(-t).
\end{align*}
In particular, the above inequality implies that if $s \ge d$, then 
\begin{equation*}
\Pr\left(  \nm{ ( I_d \otimes G^*_j) \vect(\Theta_{ij}) }_F^2 > \sigma^2 s^2  \right) \le \exp\left(- \frac{s^2}{2} \right),
\end{equation*}
which, upon using $r_i = \sum_{j=1}^n w_{ij}$, yields
\begin{equation*}
\Pr\left( \nm{[ D^{-1} \Delta G^* ]_i}_F > \frac{\sigma s}{\sqrt{r_i}}  \right) \le \exp\left( -\frac{s^2}{2} \right).
\end{equation*}
This completes the proof.
\end{proof}

The following proposition provides an upper bound on the estimation error of our approach for discrete subgroups of the orthogonal group.
\begin{prop}
\label{prop:est_err_discrete}
Suppose that $\G$ is a discrete subgroup of $\mathcal{O}(d)$ with minimum separation $\tau$ and that the measurement graph is an Erd{\H{o}}s-R{\'e}nyi random graph with observation rate $p \in (0,1]$. Let $\{\Theta_{ij} : 1 \le i < j \le n\}$ be independent noise matrices that are independent of the measurement graph and whose entries are i.i.d. sub-Gaussian random variables with parameter $\sigma>0$.
Then, there exist constants $c_0,c_1,c_2,c_3,c_4>0$ such that whenever $\sigma \le \frac{c_0 \tau \sqrt{pn}}{d}$ and $p \le \frac{c_1\log n}{n}$, we have
\begin{align*}
&\, \nm{ \Pi^n \left( G^* + 2 D^{-1} \Delta G^* \right) - G^* }_F^2 \\
\le& \, c_3 \min\left\{  dn\exp\left(- \frac{c_2 \tau^2 pn}{\sigma^2}\right) + d\log n ,\, d n^2 \log n \exp\left(- \frac{c_2 \tau^2 pn}{\sigma^2}\right) \right\} ,
\end{align*}
with probability at least $1 - \frac{c_4}{n}$.
\end{prop}

Very recently, minimax rates for synchronization problems over discrete groups have been obtained in \cite{fei2020achieving} and \cite[Section~8]{gao2022iterative}. Their results imply that for the Boolean or permutation group with $p=1$, if the sub-Gaussian parameter $\sigma = \sigma_n$ in the noise satisfies $ \frac{n}{\sigma_n^2} \to \infty $ as $n\to \infty$, then the estimation error of any estimator $G$ is lower bounded by
\begin{align*}
\varepsilon^2 (G) \ge  n \exp\left( -\frac{ n (1 + o(1))}{2 \sigma^2}   \right).
\end{align*}
Moreover, in~\cite[Section~8]{gao2022iterative}, an iterative algorithm achieving a matching upper bound on the estimation error has been developed. Proposition~\ref{prop:est_err_discrete} implies that the estimator $G^\infty$ output by Algorithm~\ref{alg:GPM} achieves near-optimal estimation error. Indeed, for any discrete subgroup of $\mathcal{O}(d)$, using Theorems~\ref{thm:master}, \ref{thm:group}, \ref{thm:kappa}, \ref{thm:Delta_cond_random_graph}, \ref{thm:init} and the second bound in Proposition~\ref{prop:est_err_discrete}, we see that our estimator $G^\infty$ will satisfy the estimation error bound
\[ \varepsilon^2 (G^\infty) \le c_0 n \exp\left( -\frac{c_1 \tau^2 pn (1 - o(\log (nd)))}{\sigma^2}  \right)  \]
with probability converging to 1, where $c_0,c_1>0$ are some constants. This shows that our approach enjoys not only great flexibility but also near-optimal estimation error.

\begin{proof}[Proof of Proposition~\ref{prop:est_err_discrete}]
By the union bound and Chernoff's inequality for the lower tail~\cite[Exercise~2.3.2]{vershynin2018high}, there exist some constants $c_0, c_1, c_2 > 0$ such that whenever $p \ge \frac{c_0 \log n}{n}$, we have
\begin{equation}
\label{ineq:36}
\Pr\left( \min_{i \in [n] } r_i \ge c_1 pn \right) \ge 1- \frac{c_2}{n}.
\end{equation}
Also, if $\sigma^2 \le \frac{c_1 \tau^2 pn}{16 d^2}$, then by Lemma~\ref{lem:deviation_bound}, we have
\begin{align*}
\Pr\left( \nm{[ D^{-1} \Delta G^* ]_i}_F > \frac{\tau}{4} \,\Bigg|\,  \min_{i \in [n] } r_i \ge c_1 pn \right) \le \exp\left(- \frac{c_1 \tau^2 pn}{32 \sigma^2}\right).
\end{align*}
Since $\tau$ is the minimum separation, by the definition of the projection map $\Pi$ and using the above inequality, we have
\begin{equation}\label{ineq:37}
\Pr\left(\nm{ \Pi \left( G^*_i + 2 [D^{-1} \Delta G^*]_i \right) - G^*_i }_F > 0\,\Bigg|\,  \min_{i \in [n] } r_i \ge c_1 pn  \right) \le \exp\left(- \frac{c_1 \tau^2 pn}{32\sigma^2}\right).
\end{equation}
Moreover, we have $\nm{ \Pi \left( G^*_i + 2 [D^{-1} \Delta G^*]_i \right) - G^*_i }_F \le 2\sqrt{d} $.
Hence, by conditioning on the event $\min_{i \in [n]} r_i \ge c_1 pn$, the random variable 
\[ \frac{ 1 }{4d} \nm{ \Pi^n \left( G^* + 2 D^{-1} \Delta G^* \right) - G^* }_F^2 \]
is upper bounded\footnote{One can construct a binomial random variable $B$ on the same probability space as the random variable $\frac{ 1 }{4d} \nm{ \Pi^n \left( G^* + 2 D^{-1} \Delta G^* \right) - G^* }_F^2$ so that $\left(\frac{ 1 }{4d} \nm{ \Pi^n \left( G^* + 2 D^{-1} \Delta G^* \right) - G^* }_F^2\right) (\omega) \le B(\omega)$ for every outcome $\omega$ in the sample space.} by a binomial random variable with $n$ trials and success probability 
\[ p' = \exp\left(- \frac{c_1 \tau^2 pn}{32\sigma^2}\right) .\]
By Chernoff's inequality for the upper tail~\cite[Theorem 2.3.1]{vershynin2018high}, for any $t>0$,
\begin{align*}
\Pr\left( \frac{ 1 }{4d} \nm{ \Pi^n \left( G^* + 2 D^{-1} \Delta G^* \right) - G^* }_F^2 \ge p'n + t \,\Bigg|\,  \min_{i \in [n] } r_i \ge c_1 pn \right) \le \exp\left( - \frac{t^2}{2 (p'n+ \tfrac{t}{3})} \right).
\end{align*}
Taking $t = 2p'n + \log n $ and using~\eqref{ineq:36}, we get
\begin{align*}
\Pr\left( \frac{ 1 }{4d} \nm{ \Pi^n \left( G^* + 2 D^{-1} \Delta G^* \right) - G^* }_F^2 \ge 3 n\exp\left(- \frac{c_1 \tau^2 pn}{32\sigma^2}\right) + \log n \right) \le \frac{c_3}{n}
\end{align*}
for some constant $c_3 >0$. This proves the first bound.

For the second bound, we note that if $p' \le \tfrac{1}{n^2}$, then by using~\eqref{ineq:37} and the union bound, we have
\begin{align*}
\Pr\left( \frac{ 1 }{4d} \nm{ \Pi^n \left( G^* + 2 D^{-1} \Delta G^* \right) - G^* }_F^2 = 0  \,\Bigg|\,  \min_{i \in [n]} r_i \ge c_1 pn \right) \ge 1 - \frac{1}{n}.
\end{align*}
This, together with~\eqref{ineq:36}, implies that
\[ \Pr\left( \frac{ 1 }{4d} \nm{ \Pi^n \left( G^* + 2 D^{-1} \Delta G^* \right) - G^* }_F^2 = 0  \right) \ge 1 - \frac{c_2 + 1}{n}. \]
If $p' \ge \frac{1}{n^2}$, then $3 p' n^2 \log n \ge 2 p'n + \log n$. In this case, similar to the first bound, we get
\begin{align*}
\Pr\left( \frac{ 1 }{4d} \nm{ \Pi^n \left( G^* + 2 D^{-1} \Delta G^* \right) - G^* }_F^2 \ge 3 n^2 \log n \exp\left(- \frac{c_1 \tau^2 pn}{32\sigma^2}\right)  \right) \le \frac{c_3}{n} .
\end{align*}
Combining the last two inequalities yields the second bound.
\end{proof}

Under a slightly stronger assumption on the noise, we can even show that Algorithm~\ref{alg:GPM} will converge to the ground truth with high probability.
\begin{prop}
Consider the setting of Proposition~\ref{prop:est_err_discrete}. Then, there exist constants $c_0 , c_1, c_2>0$ such that whenever $\sigma \le \frac{c_0 \tau \sqrt{pn}}{d \sqrt{\log n}}$ and $p\ge \frac{c_1\log n}{n}$, Algorithm~\ref{alg:GPM} will converge to the ground truth (\ie, $G^\infty = G^* Q$ for some $Q\in \G$) with probability at least $1-\frac{c_2}{n}$.
\end{prop}

\begin{proof}
It suffices to prove that $\nm{ \Pi^n \left( G^* + 2 D^{-1} \Delta G^* \right) - G^* }_F^2 = 0$ with high probability. Using~\eqref{ineq:37} and the union bound, there exist some constants $c_0 > 0$ such that if $\sigma^2 \le \frac{c_0 \tau^2 pn}{16 d^2}$, then
\begin{equation*}
\Pr\left(\nm{ \Pi \left( G^* + 2 D^{-1} \Delta G^* \right) - G^* }_F^2 > 0\,\Bigg|\,  \min_{i \in [n]} r_i \ge c_0 pn \right) \le n \exp\left(- \frac{c_0 \tau^2 pn}{32\sigma^2}\right).
\end{equation*}
Combining this inequality with \eqref{ineq:36}, there exists a constant $c_1 > 0$ such that whenever $\sigma^2 \le \frac{c_1 \tau^2 pn}{ d^2 \log n}$, we have
\begin{equation*}
\Pr\left(\nm{ \Pi \left( G^* + 2 D^{-1} \Delta G^* \right) - G^* }_F^2 > 0 \right) \le \frac{c_2}{n}
\end{equation*}
for some constant $c_2 >0$. This completes the proof.
\end{proof}

We then present a bound that is weaker than the one in Proposition~\ref{prop:est_err_discrete} but applies to both continuous and discrete subgroups of the orthogonal group.
\begin{prop}\label{prop:est_err_continuous}
Suppose that the measurement graph is an Erd{\H{o}}s-R{\'e}nyi random graph with observation rate $p \in (0,1]$. Let $\{\Theta_{ij} : 1 \le i < j \le n\}$ be independent noise matrices that are independent of the measurement graph and whose entries are i.i.d. sub-Gaussian random variables with parameter $\sigma>0$.
Then, there exist constants $c_0,c_1, c_2>0$ such that whenever $p \ge \frac{c_0(\log n)^2}{n}$, we will have
\begin{align*}
& \nm{ \Pi^n \left( G^* + 2 D^{-1} \Delta G^* \right) - G^* }_F^2 
\le \frac{c_1 d^2 \sigma^2 }{p},
\end{align*}
with probability at least $1 - \frac{c_2}{n}$. If the sub-Gaussian entries $\{\Theta_{ij}: 1\le i < j\le n\}$ are zero-mean Gaussian with standard deviation $\sigma >0$, then the same inequality holds under the weaker requirement $p \ge \tfrac{c_0 \log n}{n}$.
\end{prop}

In the recent paper~\cite{gao2021optimal}, it has been shown that for $\mathcal{O}(d)$-\sync and $\mathcal{SO}(d)$-\sync under the setting of Erd{\H{o}}s-R{\'e}nyi measurement graph and additive Gaussian noise, if the standard deviation $\sigma = \sigma_n$ and the observation rate $p = p_n$ satisfy $\frac{pn}{\sigma^2} \to \infty$ and $\frac{pn}{\log n} \to \infty$ as $n\to\infty$, then the estimation error of any estimator $G$ is lower bounded by
\[ \varepsilon^2 (G) \ge  (1 - o(1)) \frac{d(d-1) \sigma^2}{2p}. \]
In addition, it has been shown that an iterative polar decomposition algorithm achieves the estimation error
\[ \varepsilon^2 (G) \le (1 + o(1)) \frac{d(d-1) \sigma^2}{2p}, \]
which matches with the lower bound asymptotically. Under the same measurement graph and noise setting, Theorems~\ref{thm:master}, \ref{thm:group}, \ref{thm:kappa}, \ref{thm:Delta_cond_random_graph}, \ref{thm:init} and Proposition~\ref{prop:est_err_continuous} imply that for any subgroup of $\mathcal{O}(d)$, the estimator $G^\infty$ output by Algorithm~\ref{alg:GPM} will satisfy the estimation error bound
\[ \varepsilon^2 (G^\infty) \le \frac{c_1 d^2 \sigma^2}{p} \]
with probability converging to 1, where $c_1 > 0$ is some constant. This once again shows the near-optimality of our approach.

\begin{proof}[Proof of Proposition~\ref{prop:est_err_continuous}]
Using inequality~\eqref{ineq:contraction}, it suffices to bound $\nm{ D^{-1}\Delta G^* }_F^2$. Since each of the $d$ columns of $G^*$ has length $\sqrt{n}$, we have 
\begin{align*}
& \nm{ D^{-1}\Delta G^* }_F^2 = \sum_{i = 1}^d \nm{ (D^{-1}\Delta G^*)_{:,i} }_F^2 \le \frac{\nm{\Delta}^2 n}{\min_{i\in[n]} r_i^2} .
\end{align*}
The desired inequality then follows from inequality~\eqref{ineq:36} and Proposition~\ref{prop:Delta_norm}.
\end{proof}

\subsection{Summary of Results}
Now, let us summarize our technical developments so far. By combining the results in Theorems~\ref{thm:group},~\ref{thm:kappa},~\ref{thm:Delta_cond_random_graph}, and~\ref{thm:init}, we see that under the setting of Theorem~\ref{thm:init} and with $\G$ being either the orthogonal group $\mathcal{O}(d)$, the special orthogonal group $\mathcal{SO}(d)$, the permutation group $\mathcal{P}(d)$, or the cyclic group $\mathcal{Z}_m$, all four conditions~\eqref{master_condi_i}--\eqref{master_condi_iv} in the master theorem (Theorem~\ref{thm:master}) will be satisfied with high probability. Consequently, the estimator output by GPM will satisfy the estimation error bound in Proposition~\ref{prop:est_err_discrete} for discrete subgroups and that in Proposition~\ref{prop:est_err_continuous} for continuous subgroups of the orthogonal group. To the best of our knowledge, this is the first time an estimation performance guarantee of such generality is obtained for GPM. 

As is evident from Theorems~\ref{thm:group},~\ref{thm:kappa},~\ref{thm:Delta_cond_random_graph}, and~\ref{thm:init}, the aforementioned guarantee requires the observation rate $p$ to be at least on the order of $\max\left\{ \frac{\alpha^2d}{\beta^2n}, \frac{\log n}{n} \right\}$ and the noise level $\sigma$ to be at most on the order of $\frac{\beta\sqrt{pn}}{\alpha d}$, where $\alpha \ge 1$ and $\beta \in (0,1]$ are related to the error-bound geometry of the subgroup $\G$ (see Conditions~\ref{cond:alpha} and~\ref{cond:beta}). 
When $\G = \mathcal{O}(d)$ and $p=1$, our bound on $\sigma$ has a more favorable order than that in~\cite[Theorem 3.2]{ling2020improved}. Nevertheless, we should point out that our result pertains to the \emph{estimation} performance of GPM, while that in~\cite{ling2020improved} pertains to the \emph{optimization} performance (\ie, convergence behavior) of GPM. In particular, as detailed in the discussion following Theorem~\ref{thm:master}, our master theorem guarantees that the estimation error of the iterates generated by the suitably initialized GPM decreases at least geometrically to some threshold, while~\cite[Theorem 3.2]{ling2020improved} establishes the linear convergence of the iterates, albeit only for the setting of complete measurement graph and additive Gaussian noise.


To put our results in context, in Table~\ref{table:sync_results}, we compare the estimation error bounds achieved by our approach with the best estimation error bounds achieved by non-convex approaches in the literature, under the setting of Erd{\H{o}}s-R{\'e}nyi measurement graph and additive Gaussian noise. We should point out that our results are more general than the other ones listed in the table, in the sense that the former also apply to general subgroups satisfying Conditions~\ref{cond:alpha} and~\ref{cond:beta} and to the setting of additive \emph{sub-Gaussian} noise.
We should point out that the estimation error bounds from \cite{gao2022iterative} and \cite{gao2021optimal} listed in Table~\ref{table:sync_results} are under the asymptotics $n\to \infty$, where as ours are non-asymptotic that hold for finite $n$. Nevertheless, in~\cite{gao2021optimal}, non-asymptotic estimation error bounds with a more explicit expression for the term $o(1)$ are also derived for $\mathcal{O}(d)$ and $\mathcal{SO}(d)$, which are omitted here.

\renewcommand{\arraystretch}{1.5}
\begin{table}[h]
\centering
\makebox[\textwidth][c]{
\begin{tabular}{|c|c|c|c|c|}
\hline
       $\G$       & Reference & $p  $ & $\sigma^2 $ & $\varepsilon^2(G)$ \\ \hline\hline
\multirow{2}{*}{$\mathcal{O}(1)$} & \cite{gao2022iterative}  & $1$ & $o(n)$  & 
$n\exp\left( -\tfrac{n(1+o(1))}{2\sigma^2} \right)$  \\ \cline{2-5}
& \textbf{Ours}  & $\Omega(\frac{\log n}{n})$ & $ O(p n)$  & 
$n \exp\left( -\frac{c  pn (1 - o(\log n))}{\sigma^2}  \right)$ \\ \hline
\multirow{2}{*}{$\mathcal{Z}_m$} & \cite{gao2022iterative}  & $1$ & $ o(\frac{n}{m^7})$  & 
$n\exp\left( -\tfrac{n(1+o(1))}{8\sigma^2} \right)$  \\ \cline{2-5}
& \textbf{Ours}  & $\Omega(\max\{\frac{d m^6}{n} , \frac{\log n}{n}\})$ & $O(\frac{ p n}{m^6} )$  & 
$n \exp\left( -\frac{c \tau^2 pn (1 - o(\log n))}{\sigma^2}  \right)$  \\ \hline
\multirow{2}{*}{$\mathcal{P}(d)$} & \cite{gao2022iterative}  & $1$ & $ o(\frac{n}{d^2})$  & 
$n\exp\left( -\tfrac{n(1+o(1))}{2\sigma^2} \right)$  \\ \cline{2-5}
& \textbf{Ours}  & $\Omega(\max\{\frac{d^3}{n}, \frac{\log n}{n}\})$ & $O(\frac{pn}{d^4})$  & 
$n \exp\left( -\frac{c  pn (1 - o(\log (nd)))}{\sigma^2}  \right)$ \\ \hline
\multirow{2}{*}{$\mathcal{O}(d)$, $\mathcal{SO}(d)$} & \cite{gao2021optimal} ($d = \Theta(1)$) & $\omega (\frac{\log n}{n }) $ & $ o( pn ) $  & 
$(1+o(1))\frac{d(d-1)\sigma^2}{2p}$  \\ \cline{2-5}
& \textbf{Ours}  & $\Omega(\max\{\frac{d}{n}, \frac{\log n}{n} \})$ & $O(\frac{pn}{d^2})$  & 
$O(\frac{d^2\sigma^2}{p})$  \\ \hline
\end{tabular}}
\caption{Comparison of the estimation error bounds $\varepsilon^2(G)$ achieved by various non-convex approaches under the setting of Erd{\H{o}}s-R{\'e}nyi measurement graph with observation rate at least $p$ and additive Gaussian noise with standard deviation at most $\sigma$. Bounds in \cite{gao2022iterative} and \cite{gao2021optimal} are asymptotic (as $n\to\infty$), whereas ours are non-asymptotic.}
\label{table:sync_results}
\end{table}

Let us briefly explain how we obtain the bounds in the $p$- and $\sigma^2$-columns of Table~\ref{table:sync_results}. For the continuous subgroups $\mathcal{O}(d)$ and $\mathcal{SO}(d)$, by Theorem~\ref{thm:group}, the parameters $\alpha $ and $\beta $ are both constants. From Theorem~\ref{thm:master}, Section~\ref{sec:verify}, and Proposition~\ref{prop:est_err_continuous}, our non-convex approach requires that 
\[p = \Omega \left( \max \left\{ \frac{\alpha^2 d}{\beta^2 n}, \frac{\log n}{n} \right\} \right) = \Omega \left( \max \left\{ \frac{d}{n}, \frac{\log n}{n} \right\} \right) \quad \text{and}\quad \sigma^2 = O\left( \frac{\beta^2 pn}{\alpha^2 d^2} \right) = O\left( \frac{pn}{d^2} \right).\]
For discrete subgroups, by Proposition~\ref{prop:cond_2_discrete}, we can take $\beta = \frac{\tau^2}{2d}$. 
Therefore, from Theorem~\ref{thm:master}, Section~\ref{sec:verify}, and Proposition~\ref{prop:est_err_discrete}, our non-convex approach requires that 
\[p = \Omega \left( \max \left\{ \frac{\alpha^2 d}{\beta^2 n}, \frac{\log n}{n} \right\}\right) \]
and
\[ \sigma^2 = O\left( \min \left\{ \frac{\beta^2 pn}{\alpha^2 d^2}, \frac{\tau^2 pn}{d^2} \right\}\right) = O\left( \min \left\{ \frac{\beta^2 pn}{\alpha^2 d^2}, \frac{\beta pn}{d} \right\}\right) = O\left( \frac{\beta^2 pn}{\alpha^2 d^2}\right). \] 
For the Boolean group $\mathcal{O}(1)$, the parameters $\alpha$ and $\beta$ are both constants. For the cyclic group $\mathcal{Z}_m$, we have
\[ \frac{\alpha^2}{\beta^2} = O\left( \tfrac{1}{\sin^6 \tfrac{\pi}{m}} \right) = O\left( \frac{1}{m^6} \right) \]
for large $m$, while for the permutation group, we have
\[ \frac{\alpha^2}{\beta^2} = O(d^2) . \]

To compare the bounds, note that our estimation error bounds for the subgroups $\mathcal{O}(1)$, $\mathcal{Z}_m$ and $\mathcal{Z}(d)$ are slightly worse than those in~\cite{gao2022iterative}. However, we have an advantage in terms of the assumptions. Indeed, our estimation error bounds apply to synchronization problems with incomplete observations, \ie, $p<1$, but those in~\cite{gao2022iterative} do not. Moreover, when $p = 1$, $m = \Theta(1)$ and $d = \Theta (1)$, our requirements on the noise variance $\sigma^2$ for the subgroups $\mathcal{O}(1)$, $\mathcal{Z}_m$ and $\mathcal{P}(d)$ are all $O(n)$, whereas those in~\cite{gao2022iterative} are $o(n)$. For the subgroups $\mathcal{O}(d)$ and $\mathcal{SO}(d)$, the estimation error bound obtained in~\cite{gao2021optimal} is slightly worse than ours. In terms of the assumptions, we manage to explicitly quantify the dependence on the group dimension $d$, whereas \cite{gao2021optimal} focuses only on the case $d = \Theta(1)$.

\section{Numerical Results} \label{sec:exp}
We have conducted numerical experiments to compare the computational speed, scalability, and estimation performance of our proposed entropic spectral estimator and the GPM-based non-convex approach with those of existing methods. In our experiments, we have considered noise models that go beyond the additive one covered by our theoretical development so as to test the viability of our proposed approach.
All our codes are implemented using MATLAB and tested on a desktop with Intel Core i7-10700 CUP (2.90GHz$\times $8). As will be seen from the results, our approach demonstrates superior performance in many different experiment settings.

\subsection{Special Orthogonal Synchronization}
We first present numerical results on $\mathcal{SO}(d)$-\sync. 

\subsubsection{Setting}\label{sec:SO(d)-setting}
We focus on the case where $d = 3$, which is most relevant to real-world applications. The experiment setting, which is the same as that in~\cite{wang2013exact}, is as follows.
We take an Erd{\H{o}}s-R{\'e}nyi random graph with observation rate $p\in (0,1]$ as the measurement graph $([n],  E )$. We consider a multiplicative noise model with two layers of multiplicative noise. More precisely, the observations are given by
\begin{equation*}\label{eq:SO(3)_noise}
	C_{ij} = G_i^* {G_j^*}^\top \Theta_{ij}^{\text{out}} \Theta_{ij}^{\text{Lan}},\quad (i,j)\in E,
\end{equation*}
where $\Theta_{ij}^{\text{out}}$ is the so-called outlier noise defined by
\begin{equation*}
	\Theta_{ij}^{\text{out}} =
	\begin{cases}
		\hfill I_3, \hfill & \text{ with probability } q, \\
		\hfill Q_{ij}\sim \text{Uniform}(\mathcal{SO}(3)), \hfill & \text{ with probability } 1-q
	\end{cases}
\end{equation*}
with $q\in (0,1]$ being the non-corruption rate, $\text{Uniform}(\mathcal{SO}(3))$ being the uniform distribution on $\mathcal{SO}(3)$,
and $\Theta_{ij}^{\text{Lan}} \in \mathcal{SO}(3)$ being generated according to the Langevin distribution (also called the von Mises-Fisher distribution) \cite{chiuso2008wide,wang2013exact} on $\mathcal{SO}(3)$ with mean $I_3$ and concentration parameter $\gamma \ge 0$, \ie, the density function of each $\Theta_{ij}^{\text{Lan}}$ is given by 
\begin{equation*}
	c(\gamma) \exp\left( \gamma \Tr ( \Theta_{ij}^{\text{Lan}}) \right)
\end{equation*}
for some normalization constant $c(\gamma) > 0$. The parameter $\gamma \ge 0$ controls the concentration of the random matrix $\Theta_{ij}^{\text{Lan}}$ around the mean $I_3$ --- the larger the parameter $\gamma$, the more concentrated around the mean $I_3$ the random matrix $\Theta_{ij}^{\text{Lan}}$ is. In particular, it is the uniform distribution $\text{Uniform}(\mathcal{SO}(3))$ when $\gamma = 0$. As $\gamma \ra +\infty$, the distribution behaves like a Gaussian distribution with mean $I_3$ and variance $\tfrac{1}{\gamma}$.

\begin{figure}[t!]
	\begin{subfigure}{.49\textwidth}
		\centering
		\includegraphics[width=\linewidth]{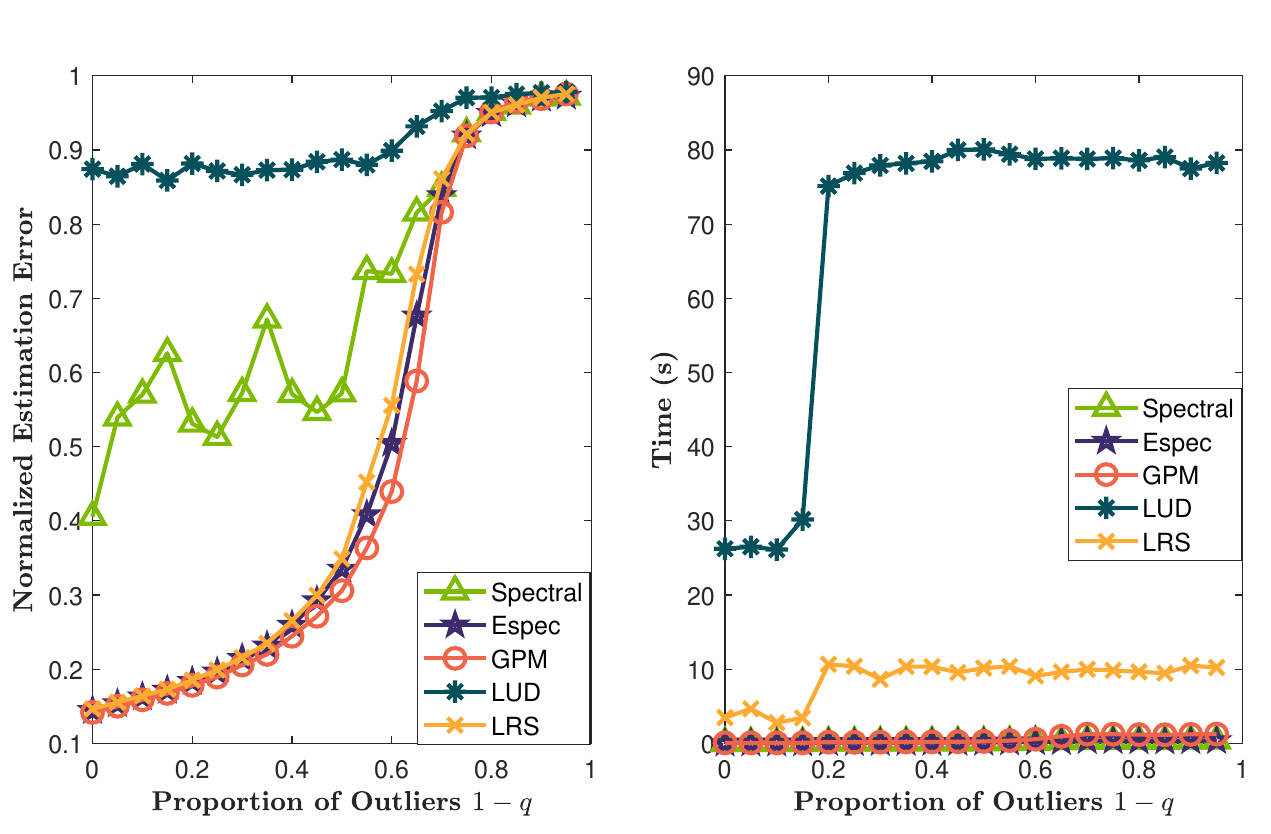}  
		\caption{$d = 3, n = 300, \gamma = 1, p = 0.2$}
	\end{subfigure}
	\hfill
	\begin{subfigure}{.49\textwidth}
		\centering
		\includegraphics[width=\linewidth]{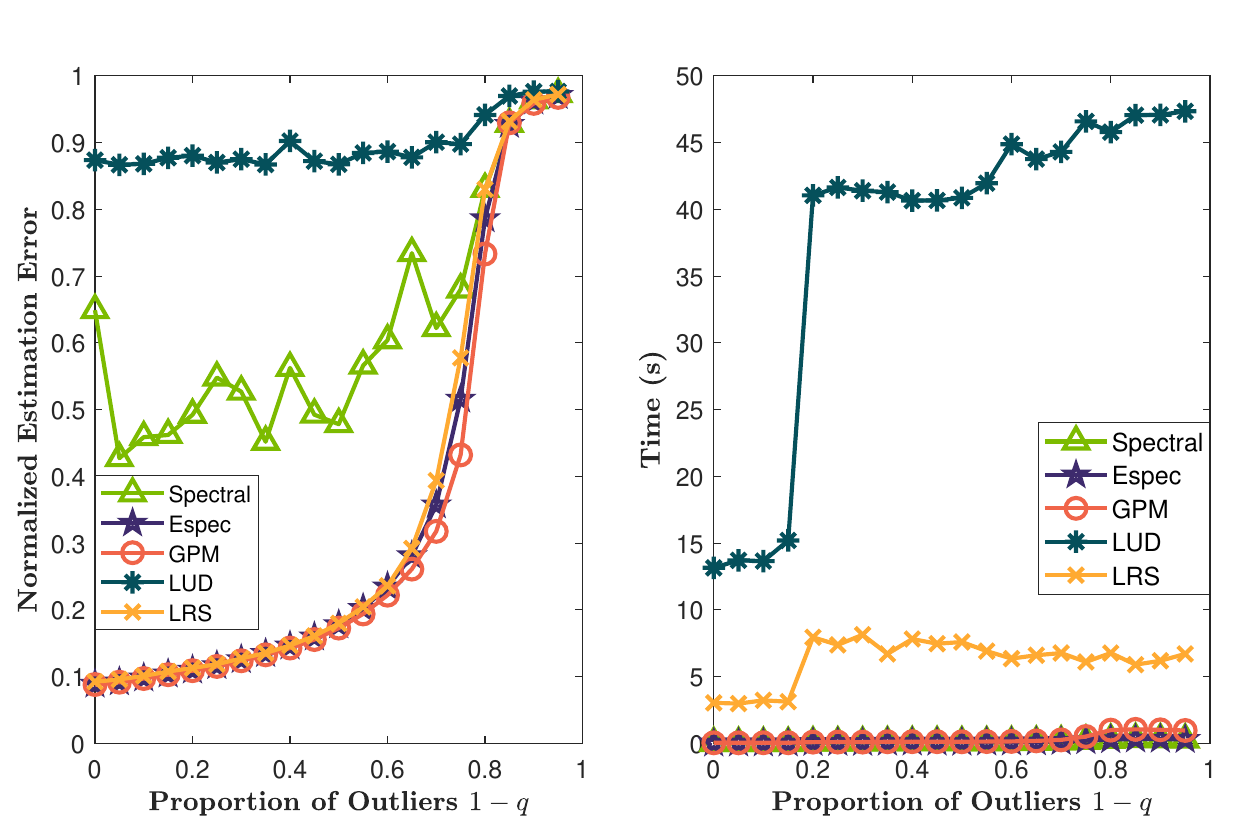}  
		\caption{$d = 3, n = 300, \gamma = 1, p = 0.5$}
	\end{subfigure}
	
	\begin{subfigure}{.49\textwidth}
		\centering
		\includegraphics[width=\linewidth]{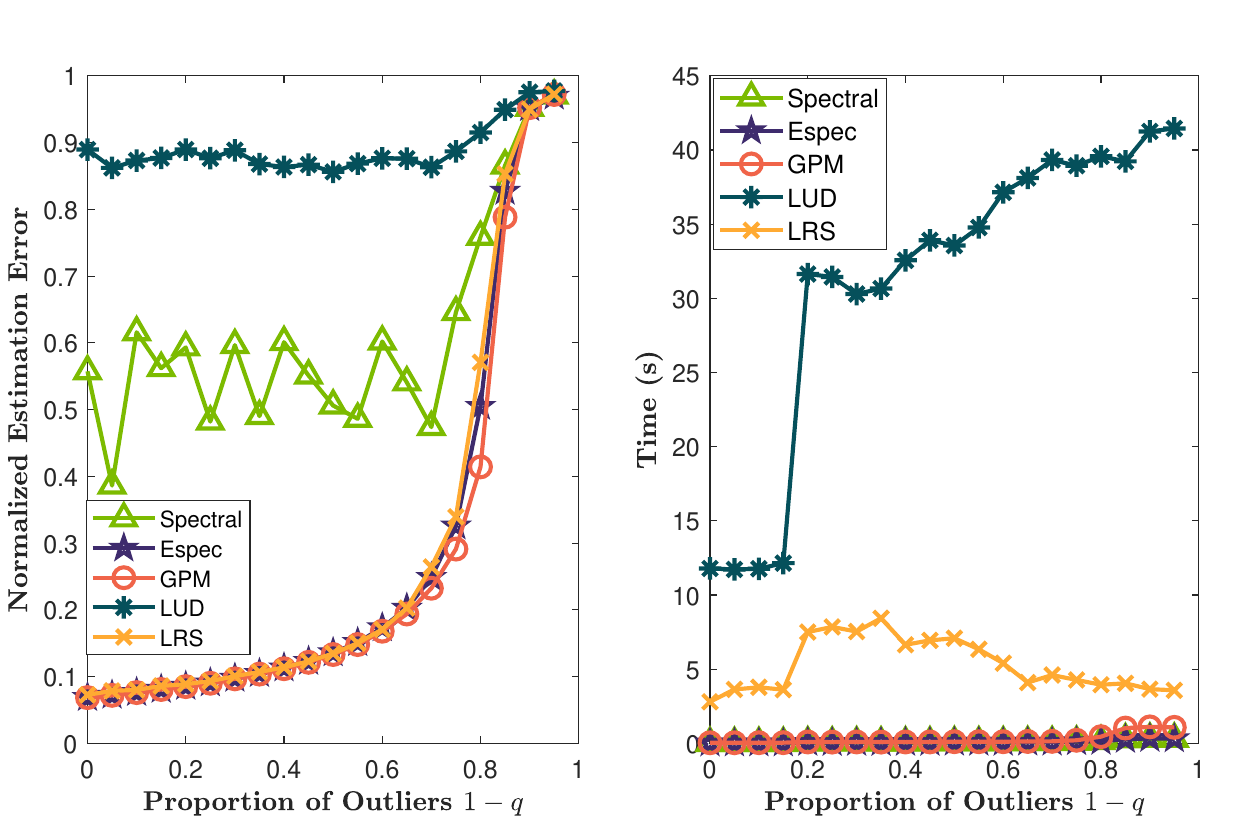}  
		\caption{$d = 3, n = 300, \gamma = 1, p = 0.8$}
	\end{subfigure}
	\hfill
	\begin{subfigure}{.49\textwidth}
		\centering
		\includegraphics[width=\linewidth]{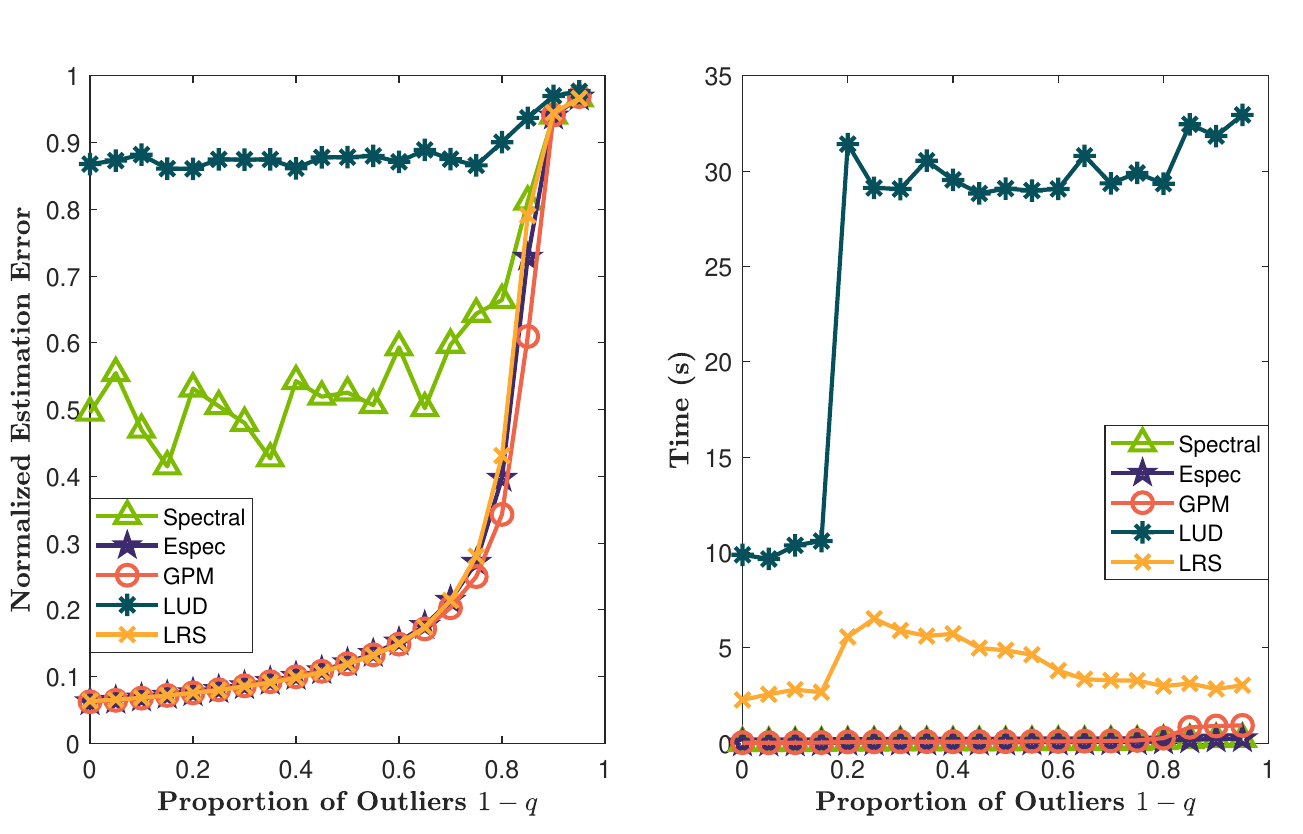}  
		\caption{$d = 3, n = 300, \gamma = 1, p = 1$}
	\end{subfigure}
	
	\caption{Estimation error and computational time of the standard spectral estimator (labeled as Spectral), the entropic spectral estimator (labeled as Espec), the non-convex approach initialized by the entropic spectral estimator (labeled as GPM), the least unsquared deviation approach (labeled as LUD), and the low-rank-sparse decomposition approach (labeled as LRS) under the setting $d=3$, $n=300$, $\gamma = 1$, and $p \in \{0.2, 0.5, 0.8, 1\}$.}
	\label{fig:SO(3)-low_noise}
\end{figure}

\begin{figure}[t!]
	\begin{subfigure}{.49\textwidth}
		\centering
		\includegraphics[width=\linewidth]{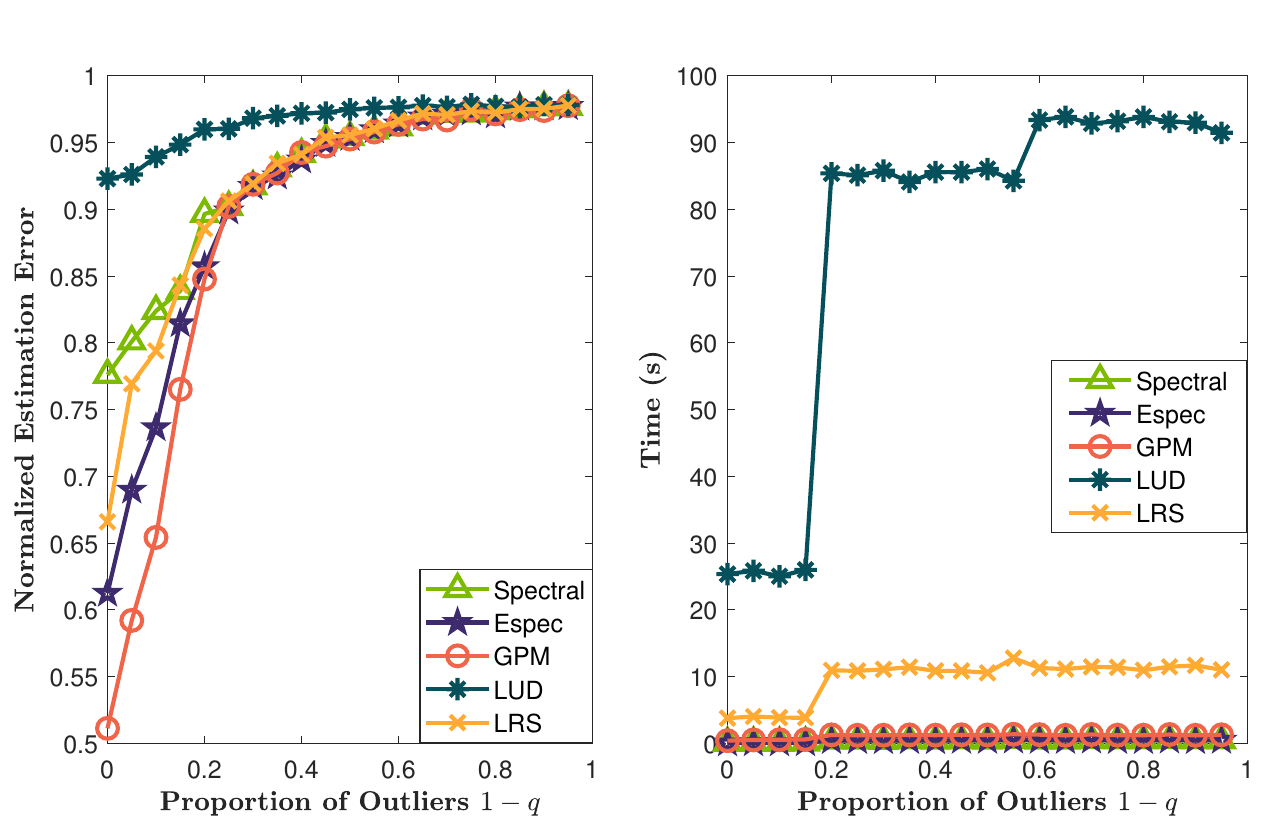}  
		\caption{$d = 3, n = 300, \gamma = 0.4, p = 0.2$}
	\end{subfigure}
	\hfill
	\begin{subfigure}{.49\textwidth}
		\centering
		\includegraphics[width=\linewidth]{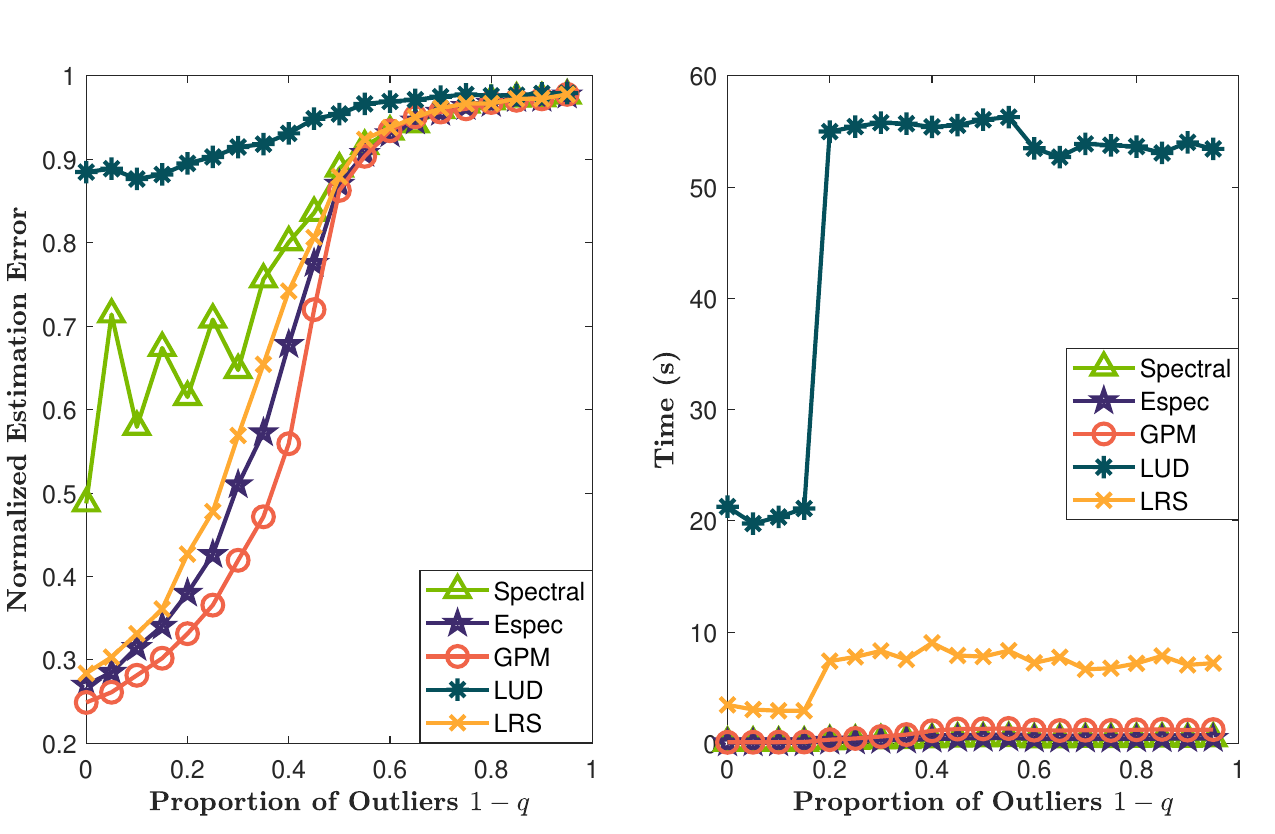}  
		\caption{$d = 3, n = 300, \gamma = 0.4, p = 0.5$}
	\end{subfigure}
	
	\begin{subfigure}{.49\textwidth}
		\centering
		\includegraphics[width=\linewidth]{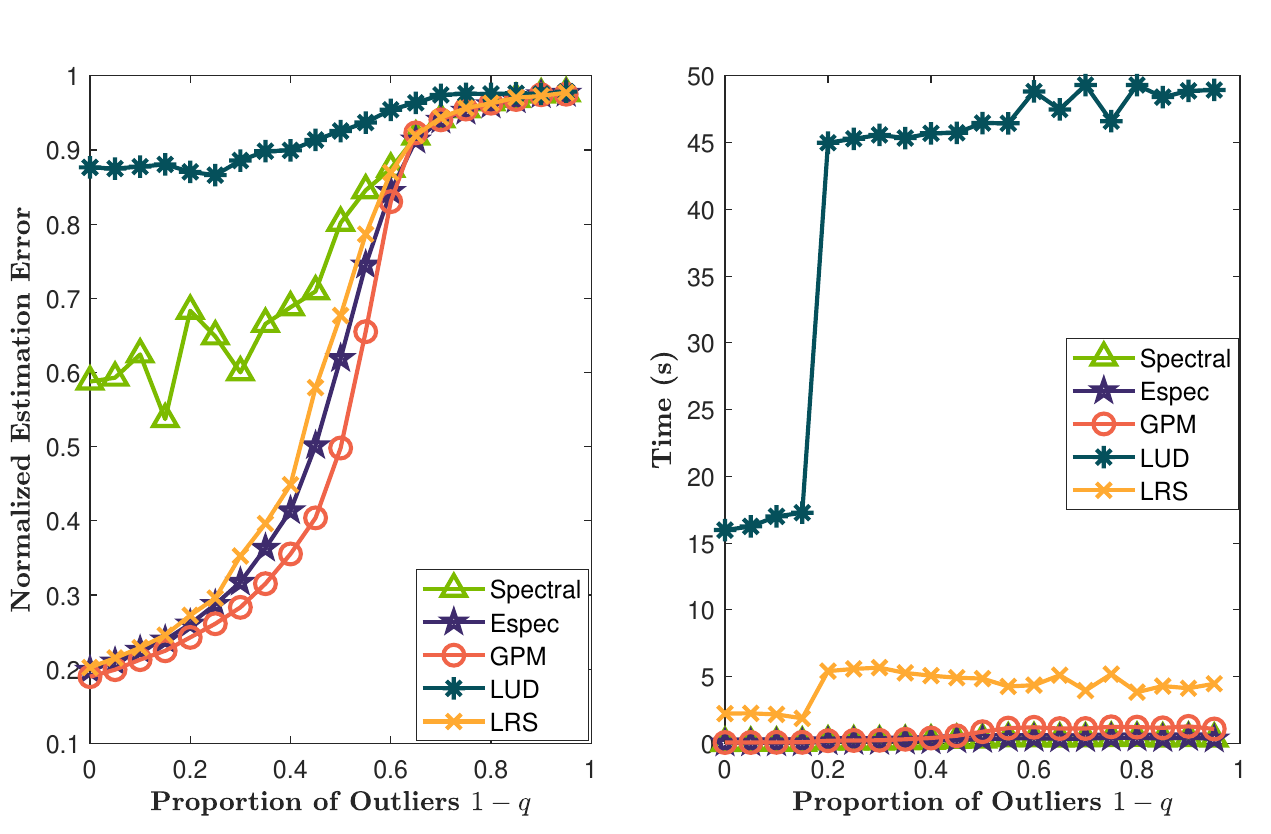}  
		\caption{$d = 3, n = 300, \gamma = 0.4, p = 0.8$}
	\end{subfigure}
	\hfill
	\begin{subfigure}{.49\textwidth}
		\centering
		\includegraphics[width=\linewidth]{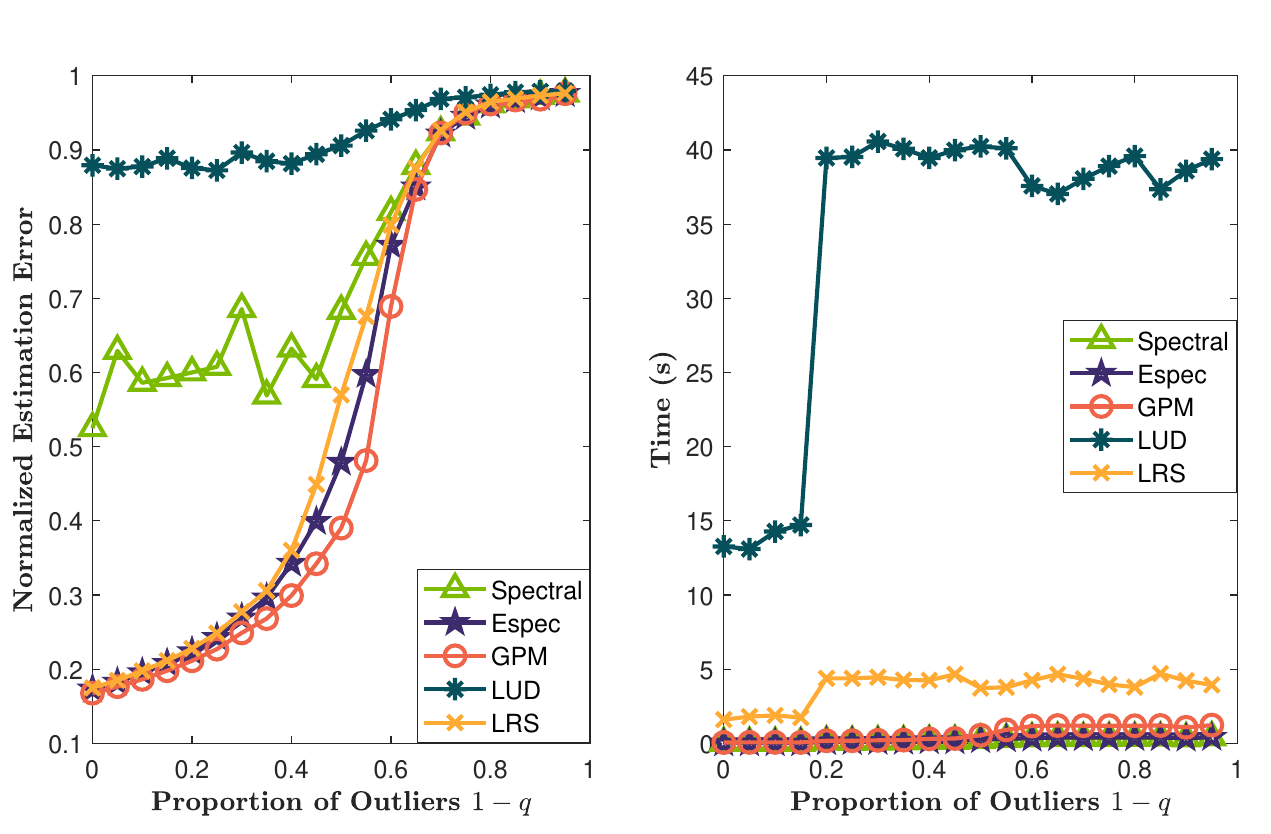}  
		\caption{$d = 3, n = 300, \gamma = 0.4, p = 1$}
	\end{subfigure}
	
	\caption{Estimation error and computational time of the standard spectral estimator (labeled as Spectral), the entropic spectral estimator (labeled as Espec), the non-convex approach initialized by the entropic spectral estimator (labeled as GPM), the least unsquared deviation approach (labeled as LUD), and the low-rank-sparse decomposition approach (labeled as LRS) under the setting $d=3$, $n=300$, $\gamma = 0.4$, and $p \in \{0.2, 0.5, 0.8, 1\}$.}
	\label{fig:SO(3)-high_noise}
\end{figure}

\subsubsection{Results}
We compare our proposed entropic spectral estimator for $\mathcal{SO}(d)$-\sync and the GPM-based non-convex approach with the least unsquared deviation approach in \cite{wang2013exact} and the low-rank-sparse decomposition approach in \cite{arrigoni2018robust}. We also include the standard spectral estimator in the comparison as baseline. 
The codes for the least unsquared deviation approach are provided by the authors of~\cite{wang2013exact}, while those for the low-rank-sparse decomposition approach are available online.\footnote{\url{https://fusiello.github.io/demo/gmf/index.html}} In our experiments, we use the default choice for all the parameters in their codes. Since the two competing methods are known to perform well against outlier noise, we mainly study the recovery performance and computational time by varying the proportion of outliers $1 - q$. For ease of comparison, we measure the recovery performance using the \emph{normalized estimation error} $\frac{\epsilon (G)}{\sqrt{2nd}}$, whose value always lies between 0 and 1. 
Figures~\ref{fig:SO(3)-low_noise} and~\ref{fig:SO(3)-high_noise} show the experiment results in the low noise ($\gamma = 1$) and high noise ($\gamma = 0.4$) regimes, respectively.
The reported time for our non-convex approach includes the time for computing the entropic spectral estimator, as the former is initialized by the latter.
All the points in the figures are obtained by averaging over 30 independent random instances.

From Figure~\ref{fig:SO(3)-low_noise}, we see that the estimation performance of the proposed entropic spectral estimator is significantly better than that of the standard spectral estimator. Moreover, when GPM is initialized by the entropic spectral estimator, it yields an estimator whose performance further improves upon that of the entropic spectral estimator.
We also note that in the low noise regime, the non-convex approach performs generally on par with the low-rank-sparse decomposition approach in terms of estimation error. This suggests that our approach is fairly robust to multiplicative and outlier noise, given that the low-rank-sparse decomposition approach was shown empirically to possess such a desirable property~\cite{arrigoni2018robust}. As for the computation time, both our entropic spectral estimator and the GPM-based non-convex approach are considerably faster than the low-rank-sparse decomposition approach. 
We also notice that the curves associated with the standard spectral estimator are less smooth and have large variance.
This could be explained as follows. Recall that the block matrix $V_C$ is formed by the eigenvectors of $C$ (see Section~\ref{subsec:ese}) and contains much useful information for our estimation problem. If a certain block of $V_C$ lies close to $\mathcal{SO}(d)$, then the standard spectral estimator directly projects it onto $\mathcal{SO}(d)$. In this case, the projection retains the useful information and hence serves as a good estimation of the corresponding block of the ground truth. But if this is not the case (\ie, the block is close to the opposite disconnected component associated with $-1$ determinant), then the direct projection onto $\mathcal{SO}(d)$ would be a bad estimation. By the symmetry in our random instances, there is a half chance that the block would lie close to $\mathcal{SO}(d)$. This introduces extra variance to the standard spectral estimator. We should also point out that for $\mathcal{SO}(d)$-\sync, the proposed entropic spectral estimator has a different projection mechanism for the blocks. Roughly speaking, it first projects the block to the closest disconnected component, irrespective of the determinant. Then, if the projection has determinant $-1$, it further ``flips'' it to the counterpart element on $\mathcal{SO}(d)$.
As another observation, quite surprisingly, the least unsquared deviation approach performs worse than the spectral estimator in terms of estimation error. The optimization problem associated with the least unsquared deviation approach is a nonlinear convex semidefinite program obtained via the relaxation technique and solved by the alternating direction augmented Lagrangian method~\cite{wen2010alternating}. We suspect that the unsatisfactory performance of the least unsquared deviation approach might be due to the fact that the alternating direction augmented Lagrangian method can only achieve low to medium accuracy and/or is sensitive to the choice of the penalty parameter in the augmented Lagrangian term.

The experiment results in the high noise regime are shown in Figure~\ref{fig:SO(3)-high_noise}. The behavior of the algorithms is mostly similar to that in the low noise case, except that the estimation performance of our GPM-based non-convex approach is fairly better than that of the low-rank-sparse decomposition approach.

\subsection{Permutation Synchronization}
Next, we present numerical results on $\mathcal{P}(d)$-\sync. 

\subsubsection{Setting}\label{sec:P(d)-setting}
Again, we take an Erd{\H{o}}s-R{\'e}nyi random graph with observation rate $p\in (0,1)$ as the measurement graph $([n], E)$. We consider an adversarial measurement model that contains both additive and multiplicative noise. Specifically, the observations are given by 
	\begin{equation*}
		C_{ij} = \Pi_{\mathcal{P}(d)}(G_i^* {G_j^*}^\top\Xi_{ij}^{\text{out}} + \sigma W_{ij}), \quad (i,j)\in E ,
	\end{equation*}
	where $\Xi_{ij}^{\text{out}}$ is the outlier noise defined by
	\begin{equation*}
		\Xi_{ij}^{\text{out}} =
		\begin{cases}
			\hfill I_d, \hfill & \text{ with probability } q, \\
			\hfill P_{ij}\sim \text{Uniform}(\mathcal{P}(d)), \hfill & \text{ with probability } 1-q
		\end{cases}
	\end{equation*}
	with $q\in (0,1]$ being the non-corruption rate, $\text{Uniform}(\mathcal{P}(d))$ being the uniform distribution on $\mathcal{P}(d)$,
	$\{ W_{ij} : (i,j) \in E \}$ being independent random matrices with i.i.d. standard Gaussian entries, and $\sigma \ge 0$ being a parameter controlling the magnitude of the additive noise. Due to the discrete nature of the signals (permutation matrices), instead of the estimation error $\varepsilon(\,\cdot\,)$, we quantify the estimation performance by the \emph{recovery rate}, which is defined as 
\begin{equation}\label{eq:recovery_rate}
	\text{Recovery Rate} = \frac{\text{number of correctly recovered group elements}}{n}.
\end{equation}

\subsubsection{Results}

\begin{figure}[t!]
	\begin{subfigure}{.49\textwidth}
		\centering
		\includegraphics[width=1\linewidth]{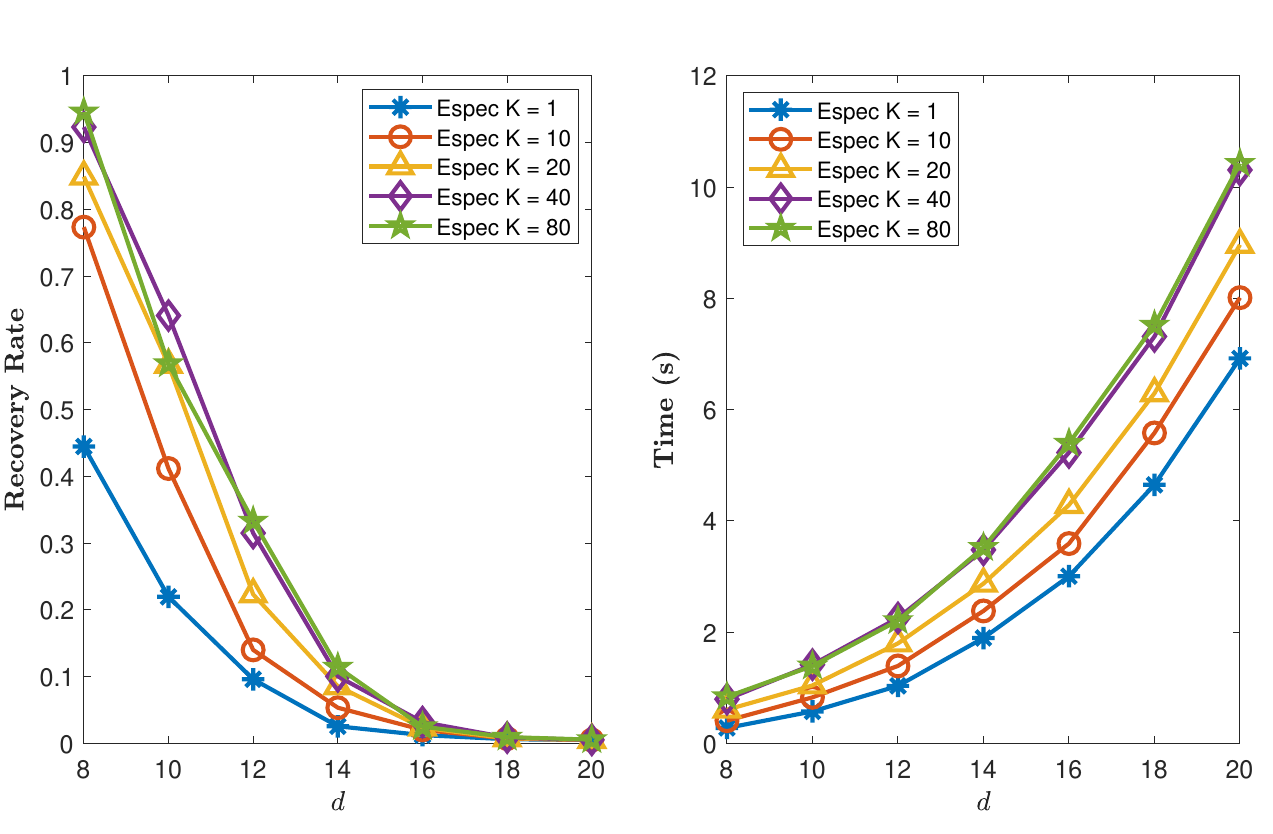}  
		\caption{Entropic Spectral Estimator}
	\end{subfigure}
	\hfill
	\begin{subfigure}{.49\textwidth}
		\centering
		\includegraphics[width=1.02\linewidth]{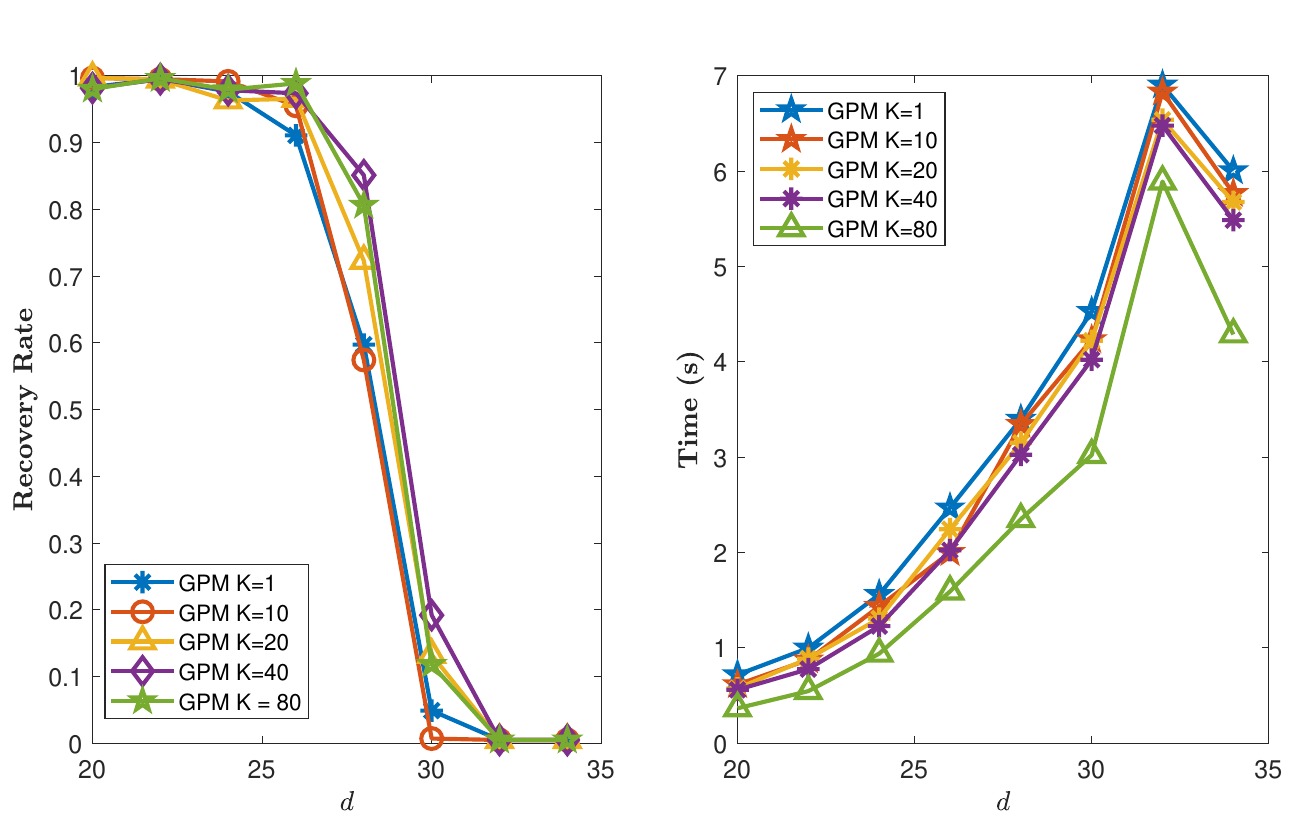}  
		\caption{GPM}
	\end{subfigure}
	\caption{Recovery rate and computational time of the proposed entropic spectral estimator (labeled as Espec) and the non-convex approach initialized by the entropic spectral estimator (labeled as GPM) under the setting  $n = 200$, $p = 0.5$, $q = 0.8$, $\sigma = 1$, and $K \in \{1, 10, 20, 40, 80\}$.}
	\label{fig:perm_exp_number}
\end{figure}

Recall that the entropic spectral estimator for $\mathcal{P}(d)$-\sync requires as input a finite subset $\mathcal{Q}\subseteq \mathcal{O}(d)/\mathcal{P}(d)$ (see Algorithm~\ref{alg:spec}). Following the approach discussed after Algorithm~\ref{alg:spec}, such a subset can be found with high probability by generating $K$ independent, uniformly distributed random orthogonal matrices, where $K$ is sufficiently large (see~\eqref{eq:K}). In our first experiment, we investigate the performance of the entropic spectral estimator and the GPM-based non-convex approach as $K$ varies.
%
Figure~\ref{fig:perm_exp_number} shows the results for $K \in \{1, 10, 20, 40, 80\}$. The algorithm we used to generate uniformly distributed random orthogonal matrices is based on the paper~\cite{diaconis1987subgroup}, see also~\cite{mezzadri2006generate} for more details. Note that when $K=1$, the entropic spectral estimator reduces to the standard spectral estimator. From the figure, we see that the recovery performance of both the entropic spectral estimator and the non-convex approach improves as $K$ increases, but the marginal benefit becomes smaller and smaller. Furthermore, we see that the computational time does not increase significantly as $K$ increases. The abrupt drop in the computational time of GPM after $d\approx 32$ is due the fact that GPM stops making progress and activates one of the stopping conditions in our implementation quite early. Therefore, for those values of $d$, the curves are not informative.

\begin{figure}[t!]
	\centering
	\begin{subfigure}{.49\textwidth}
		\centering
		\includegraphics[width=\linewidth]{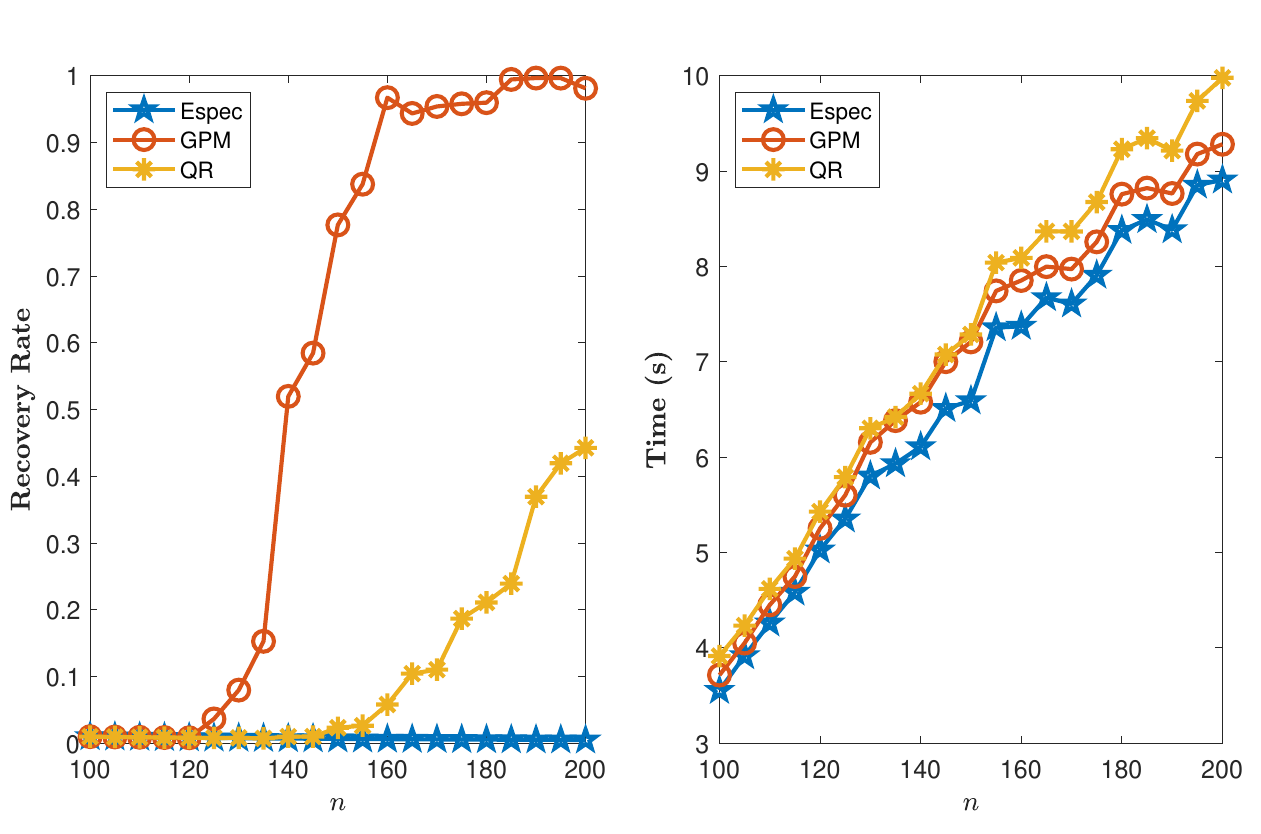}
	\caption{$d = 20, p = 0.5, q = 0.8, \sigma = 1$}
\label{fig:perm_exp_1_1} 
\end{subfigure}
	\hfill
	\begin{subfigure}{.49\textwidth}
		\centering
		\includegraphics[width=\linewidth]{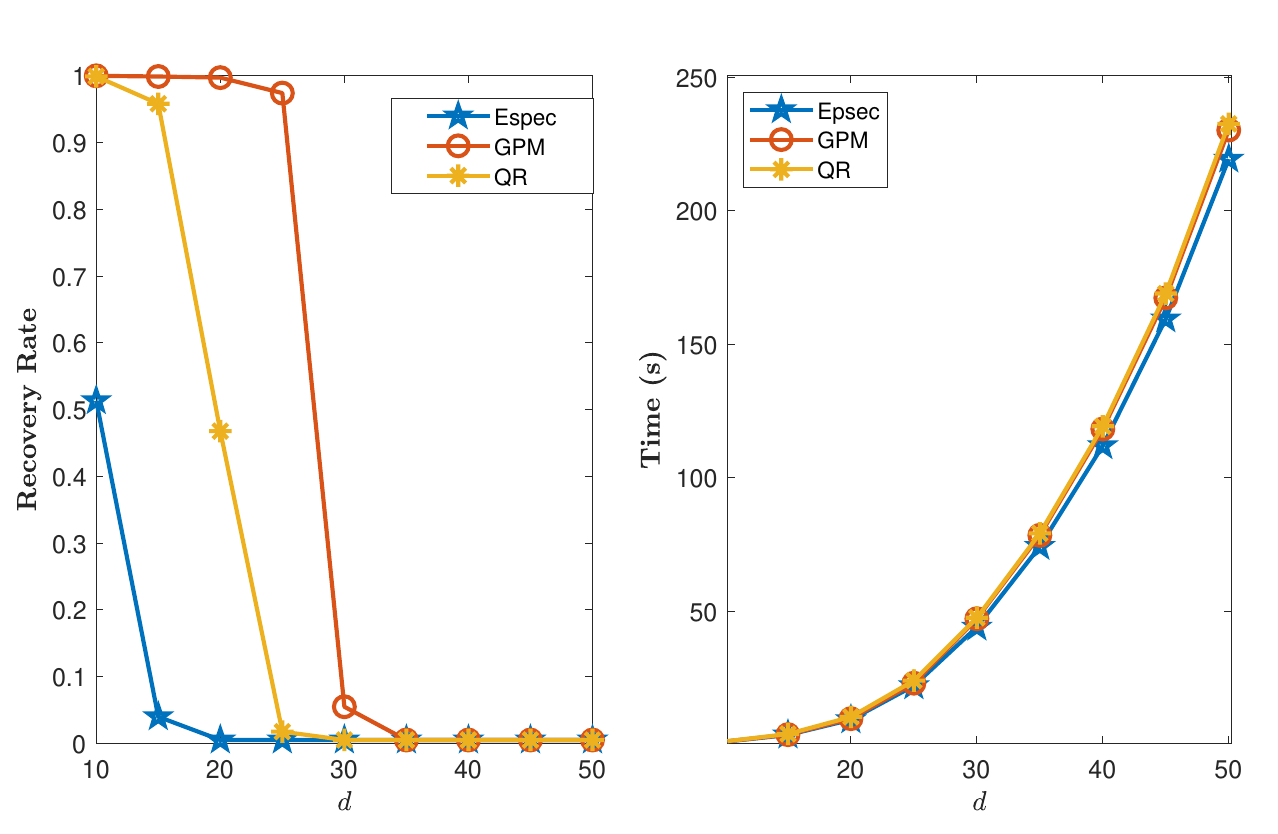} 
		\caption{$n = 200, p = 0.5, q = 0.8, \sigma = 1$} \label{fig:perm_exp_1_2}
	\end{subfigure}	
	\caption{Recovery rate and computational time of the entropic spectral estimator with $K=40$ (labeled as Espec), the non-convex approach initialized by the entropic spectral estimator (labeled as GPM), and the QR factorization-based approach initialized by the entropic spectral estimator (labeled as QR) as the parameters $n$ and $d$ vary.}
	\label{fig:perm_exp_1}
\end{figure}

In our second experiment, we compare the performance of our entropic spectral estimator (with $K = 40$) and GPM-based non-convex approach 
with that of the QR factorization-based iterative approach developed in \cite{shen2016normalized} as the parameters $n$ and $d$ vary. The results are summarized in Figure~\ref{fig:perm_exp_1}. All the points in the figure are obtained by averaging over 30 independent random instances. 
As we can see from Figure~\ref{fig:perm_exp_1}, the recovery performance of our GPM-based non-convex approach is significantly better than that of the entropic spectral estimator and the QR factorization-based approach.  
Furthermore, we see that our GPM-based non-convex approach is slightly faster than the QR factorization-based approach.
Of course, the computational time of the entropic spectral estimator is the shortest among the three tested methods, since both the GPM-based non-convex approach and the QR factorization-based approach are initialized by the entropic spectral estimator.

\subsection{Cyclic Synchronization (Joint Alignment Problem)}
Finally, we present our numerical results on $\mathcal{Z}_m$-\sync. We remark that $\mathcal{Z}_m$-\sync is equivalent to the joint alignment problem considered in~\cite{chen2018projected}.

\subsubsection{Setting}
We adopt the same experiment setting as in \cite{chen2018projected}. Specifically, we take an Erd{\H{o}}s-R{\'e}nyi random graph with observation rate $p\in (0,1)$ as the measurement graph $([n], E)$. We consider a multiplicative noise model, in which the observations are given by
\begin{equation*}
	C_{ij} = G_i^* {G_j^*}^\top \Theta_{ij}^{\text{out}}, \quad (i,j)\in E.
\end{equation*}
Here, 
\begin{equation*}
	\Theta_{ij}^{\text{out}} =
	\begin{cases}
		\hfill I_2,  \hfill & \text{ with probability } q, \\
		\hfill Q_k, \hfill & \text{ with probability } 1-q
	\end{cases}
\end{equation*}
with $q \in (0,1]$ being the non-corruption rate, $Q_k$ being defined in \eqref{eq:6}, and $k \sim \text{Uniform}([m])$ being uniformly distributed on $[m]$. To quantify the estimation performance, we again use the recovery rate\footnote{In~\cite{chen2018projected}, the \emph{misclassification rate} is used to quantify the estimation performance, which is equal to $1 - \mbox{Recovery Rate}$.} defined in~\eqref{eq:recovery_rate}.

\subsubsection{Results}
\begin{figure}[t!]
	\begin{subfigure}{.49\textwidth}
		\centering
		\includegraphics[width=1.02\linewidth]{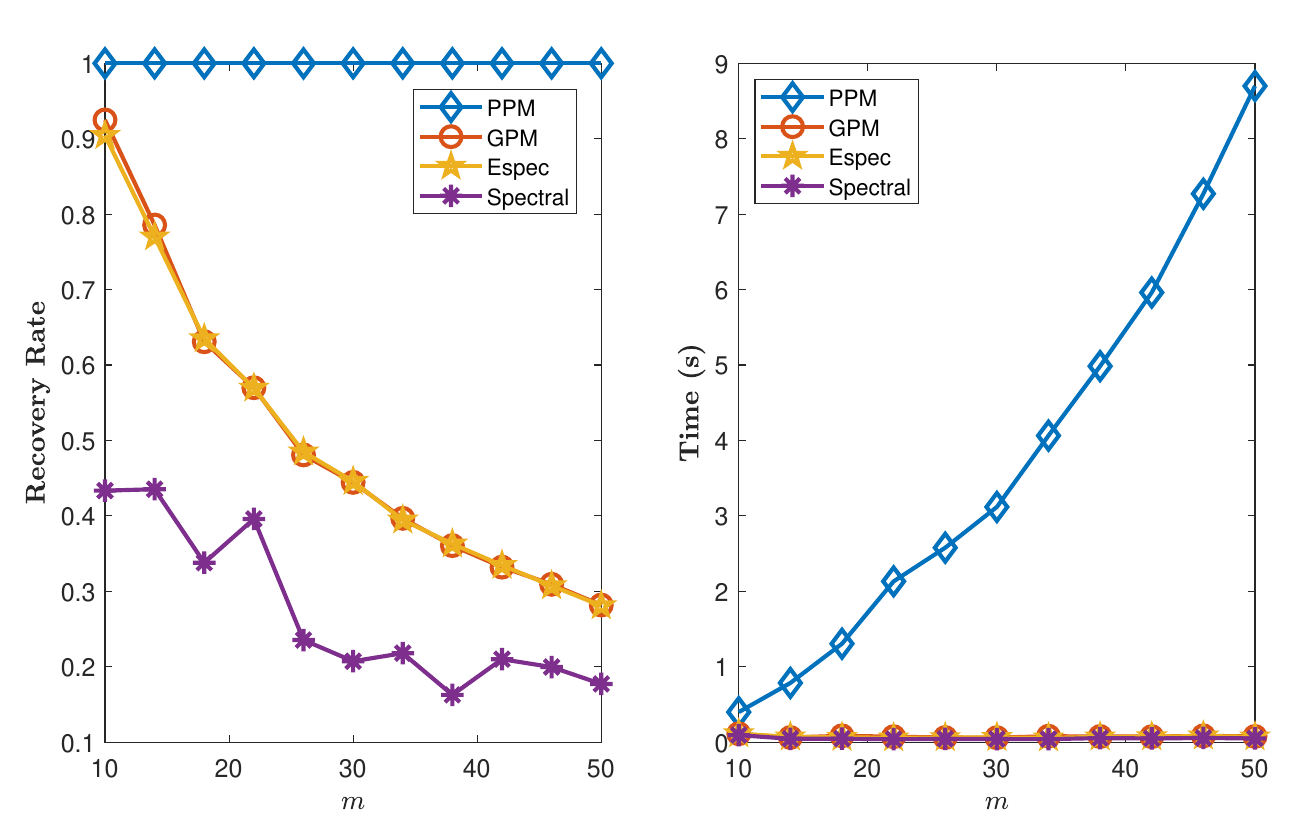}  
		\caption{$n = 500, p = 0.3, q = 0.3$}
		\label{fig:cyclic_sync_55}
	\end{subfigure}
	\hfill
	\begin{subfigure}{.49\textwidth}
		\centering
		\includegraphics[width=1.0\linewidth]{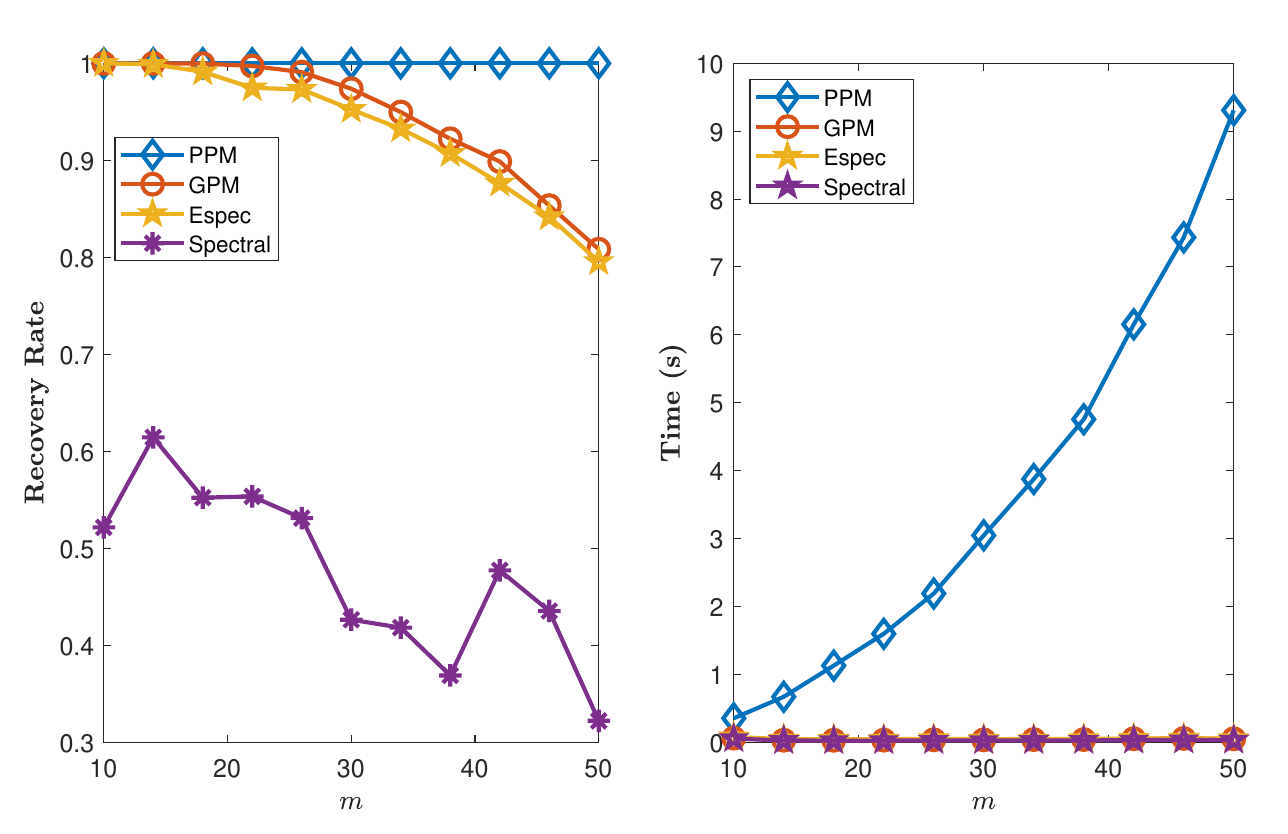}  
		\caption{$n = 500, p = 0.3, q = 0.7$}
		\label{fig:cyclic_sync_58}
	\end{subfigure}
	
	\begin{subfigure}{.49\textwidth}
		\centering
		\includegraphics[width=1.0\linewidth]{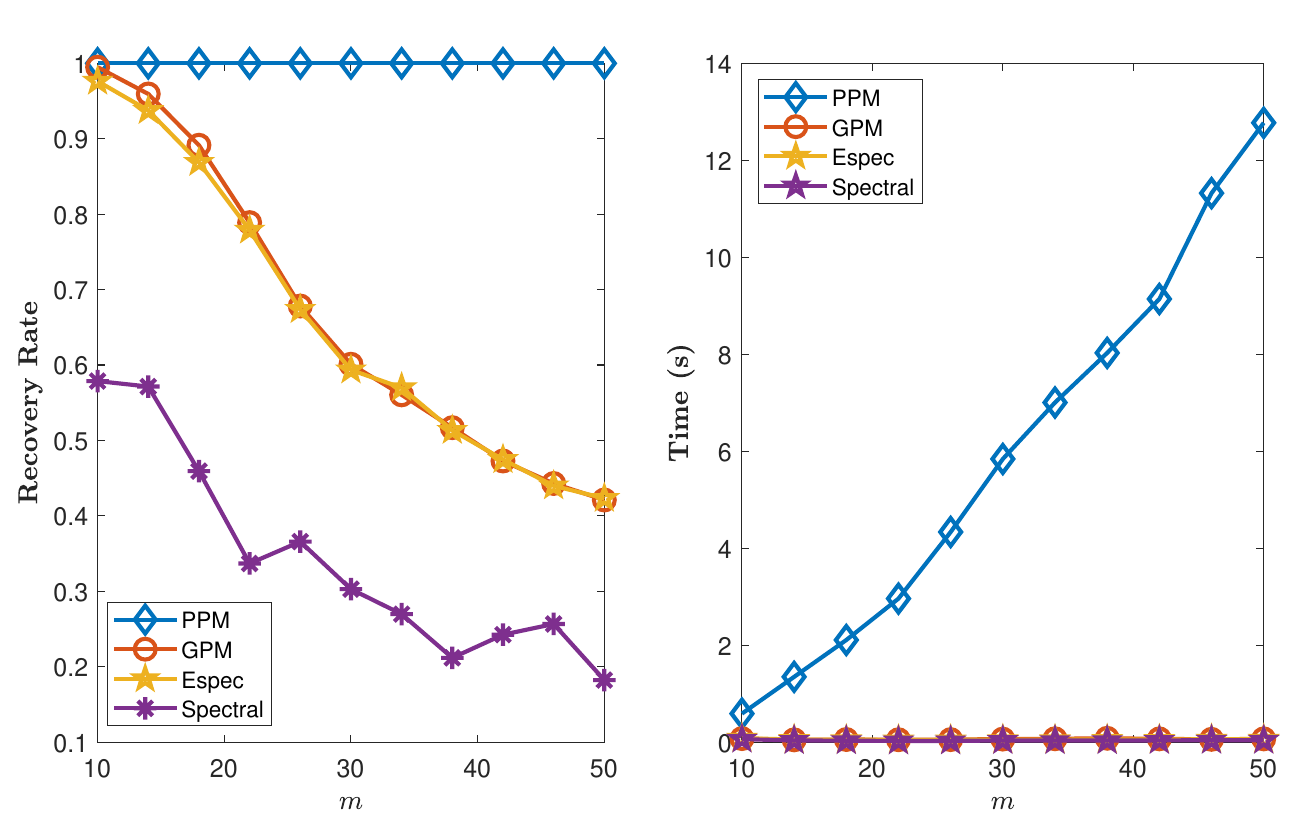}  
		\caption{$n = 500, p = 0.7, q = 0.3$}
		\label{fig:cyclic_sync_85}
	\end{subfigure}
	\hfill
	\begin{subfigure}{.49\textwidth}
		\centering
		\includegraphics[width=1.0\linewidth]{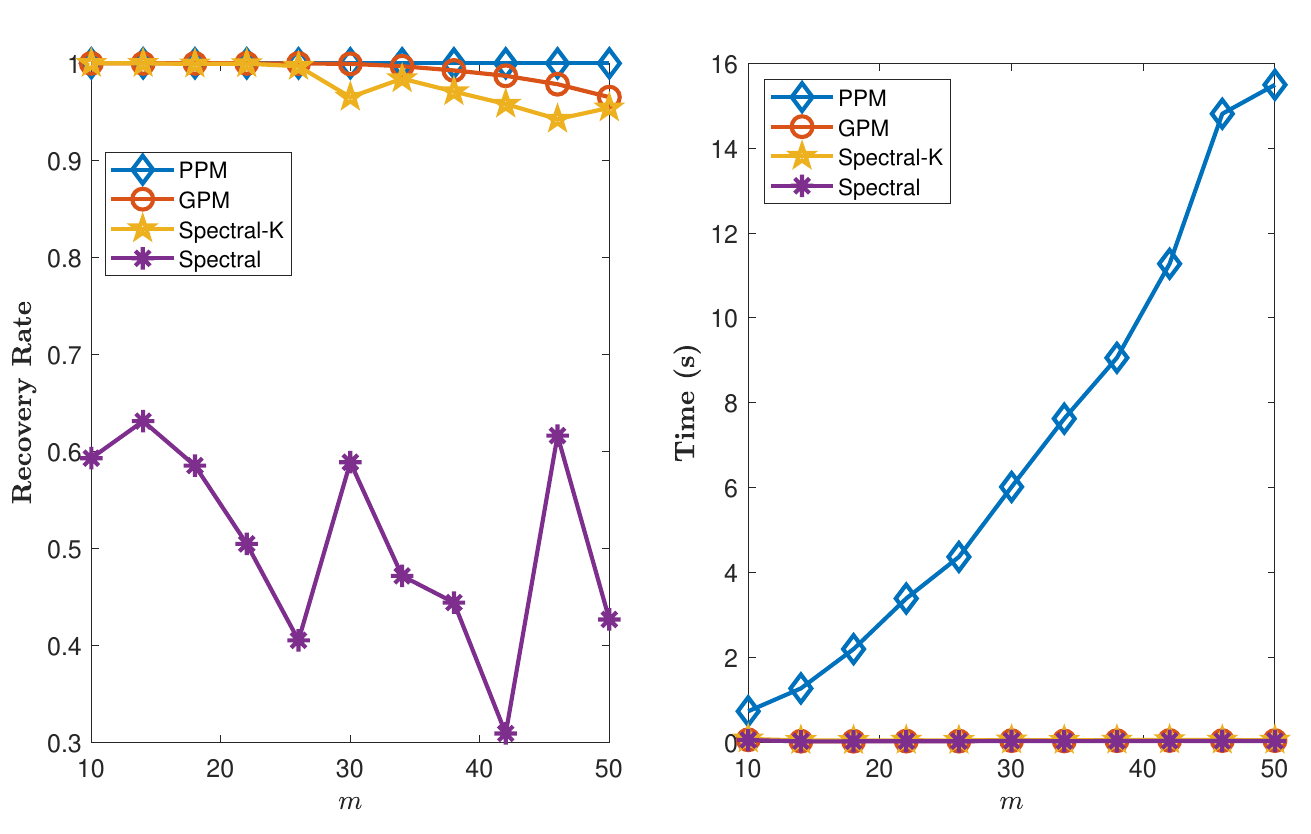}  
		\caption{$n = 500, p = 0.7, q = 0.7$}
		\label{fig:cyclic_sync_88}
	\end{subfigure}
	
	\caption{Recovery rate and computational time of the standard spectral estimator (labeled as Spectral), the proposed entropic spectral estimator with $K=10$ (labeled as Espec), the non-convex approach initialized by the entropic spectral estimator (labeled as GPM),  and the projected power method (labeled as PPM) under the setting $p \in \{0.3, 0.7\}$ and $q \in \{0.3, 0.7\}$.}
	\label{fig:cyclic_sync}
\end{figure} 

We compare the performance of the proposed entropic spectral estimator and GPM-based non-convex approach with that of another non-convex approach named \emph{projected power method}, which is developed in~\cite{chen2018projected}. The standard spectral estimator for $\mathcal{Z}_m$-\sync is once again included in the experiments as a baseline. We focus on how the recovery rate and computational time of these approaches depend on the order $m$ of the cyclic group. The results are plotted in Figure~\ref{fig:cyclic_sync}. All the points in the figure are obtained by averaging over 30 independent random instances.

As Figure \ref{fig:cyclic_sync} shows, our proposed entropic spectral estimator with $K = 10$ significantly outperforms the standard spectral estimator, and the non-convex approach can further improve the recovery rate, albeit by a small margin. The projected power method has the best recovery rate, especially when the group order $m$ is large.
Nevertheless, in terms of computational time, the proposed GPM-based non-convex approach is substantially faster than the projected power method. This is because each step of the projected power method relies on the projection onto an $m$-dimensional simplex, which is essentially an $m$-dimensional linear programming problem, but our method relies on the closed-form projection formula in Proposition~\ref{prop:cyclic_proj}, whose computational cost is independent of $m$.
Therefore, for cyclic synchronization problems, the projected power method is the way to go if recovery performance is the main concern. However, if speed and scalability are of great concern, then our results suggest that the proposed GPM-based non-convex approach would be a better alternative.

\subsection{Necessity and Tightness of Initial Estimation Error Bound}
From Theorem~\ref{thm:master}, a requirement for GPM to enjoy the theoretical guarantee on the estimation error is that the initial estimation error has to be bounded by $\tfrac{\sqrt{n}}{8\alpha}$. We now study the necessity of such a requirement through numerical experiments.

We consider $\mathcal{SO}(d)$-\sync under the same setting as that in Section~\ref{sec:SO(d)-setting} and use the family of initial points $\{G(r)\}_{r\in [0,1]}$ defined by 
\begin{equation*}
[G(r)]_i = \begin{cases}
G^*_i, \hfill & \text{with probability } 1 - r,\\
Q \sim \text{Uniform}(\mathcal{SO}(3)), \hfill & \text{with probability } r
\end{cases}
\end{equation*}
for $i\in[n]$. For simplicity, we introduce the following notation. We denote by $\varepsilon^{r,\infty}$ the normalized estimation error $\frac{\varepsilon (G^\infty) }{\sqrt{nd}}$ of the GPM initialized by $G(r)$. 
In particular, $\varepsilon^{0,\infty}$ is the normalized estimation error $\frac{\varepsilon (G^\infty) }{\sqrt{nd}}$ of the GPM initialized by the ground truth $G^*$. We also denote by $\varepsilon^{r,0}$ the normalized estimation error of the random initialization $G(r)$, \ie, $\varepsilon^{r,0} = \frac{\varepsilon (G(r))}{\sqrt{nd}}$.
Note that the initial error $\varepsilon^{r,0}$ increases as $r$ increases. Therefore, we are interested in how $\varepsilon^{r,\infty}$ scales with $r$ (green line). The results for the cases $(n, d, p, q, \gamma) = (400, 3, 0.4, 0.4, 0.4)$ and $(n, d, p, q, \gamma) = (1000, 3, 0.16, 0.4, 0.4) $ are plotted in Figure~\ref{fig:est_err_1}. To aid intuition, we include two other quantities: The initial error $\varepsilon^{r,0}$ (yellow line) and the error $\varepsilon^{0,\infty}$ of GPM initialized by the ground truth (red line). From Figure~\ref{fig:est_err_1}, we can see that $\varepsilon^{r,\infty}$ coincides with $\varepsilon^{0,\infty}$ for small $r$ and starts to deviate at about $r = 0.5$, which corresponds to an initial error $\varepsilon^{r,0}$ of roughly 0.7. The GPM error $\varepsilon^{r,\infty}$ grows sharply for $r>0.5$. The results suggest that a bound on the initial estimation error is necessary for GPM to enjoy a theoretical guarantee on the final estimation error.

\begin{figure}[t!]
	\centering
	\begin{subfigure}{.49\textwidth}
		\centering
		\includegraphics[width=\linewidth]{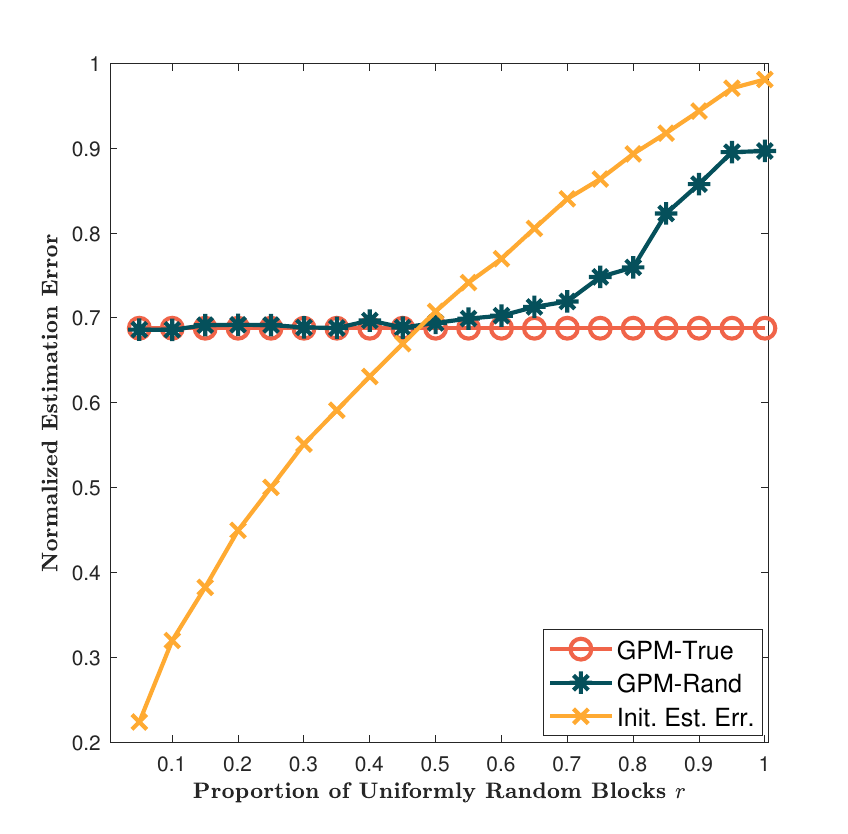}
	\caption{$n= 400, d = 3, p = 0.4, q = 0.4, \gamma = 0.4$}
\label{fig:est_err_1_1} 
\end{subfigure}
	\begin{subfigure}{.49\textwidth}
		\centering		\includegraphics[width=\linewidth]{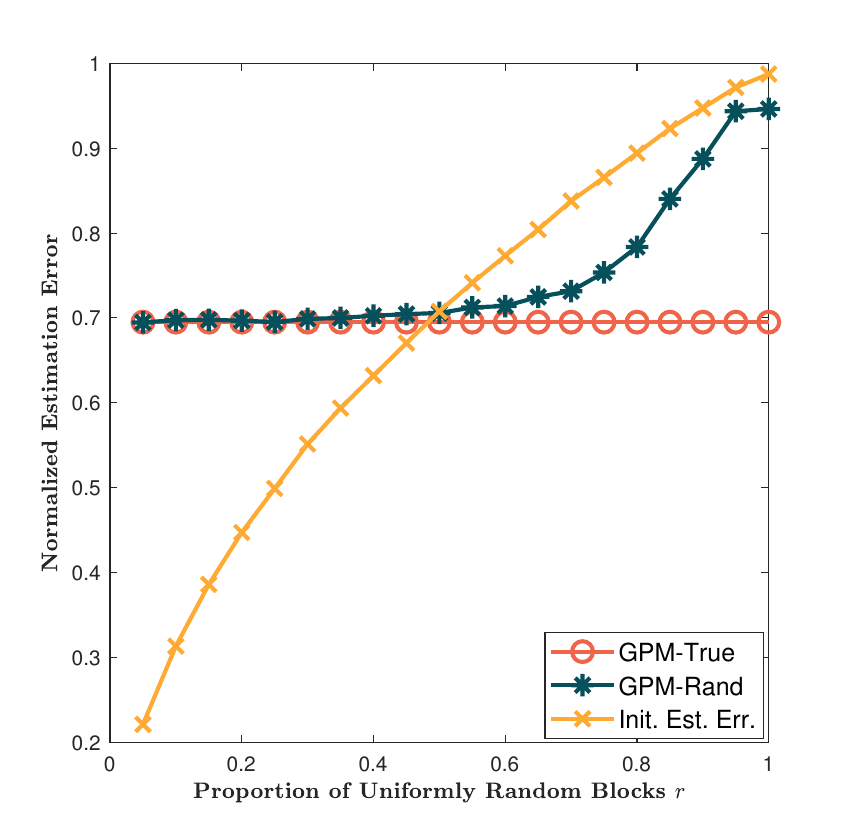} 
		\caption{$n= 1000, d = 3, p = 0.16, q = 0.4, \gamma = 0.4$} \label{fig:est_err_1_2}
	\end{subfigure}	
	\caption{Normalized estimation error $\varepsilon^{r,\infty}$ of GPM initialized by $G(r)$ (labeled as GPM-Rand), normalized estimation error $\varepsilon^{0,\infty}$ of GPM initialized by the ground truth (labeled as GPM-True), and normalized estimation error $\varepsilon^{r,0}$ of random initialization $G(r)$ (labeled as Init. Est. Err.).}
	\label{fig:est_err_1}
\end{figure}

To investigate how tight our requirement $\varepsilon (G^0) \le \frac{\sqrt{n}}{8\alpha}$ is, we repeat the above experiment for different $n$ (but keep the value $pn$ constant) and record the initial error $\varepsilon^{r,0}$ when $\varepsilon^{r,\infty}$ starts to deviate from $\varepsilon^{0,\infty}$. More precisely, we choose $n = 300, 350, \dots, 1000$, $p = \frac{160}{n}$ and record the smallest initial error $\varepsilon^{r,0}$ such that $ \varepsilon^{r, \infty} \ge 1.02\, \varepsilon^{0,\infty} $. The results are plotted in Figure~\ref{fig:init_err_1}, which show that $\varepsilon^{r,0} \approx 0.75$ stays roughly constant across different values of $n$. This empirically confirms the optimality of the requirement $\varepsilon (G^0) \le \frac{\sqrt{n}}{8\alpha}$.

\begin{figure}[t!]
		\centering		\includegraphics[width=0.95\linewidth]{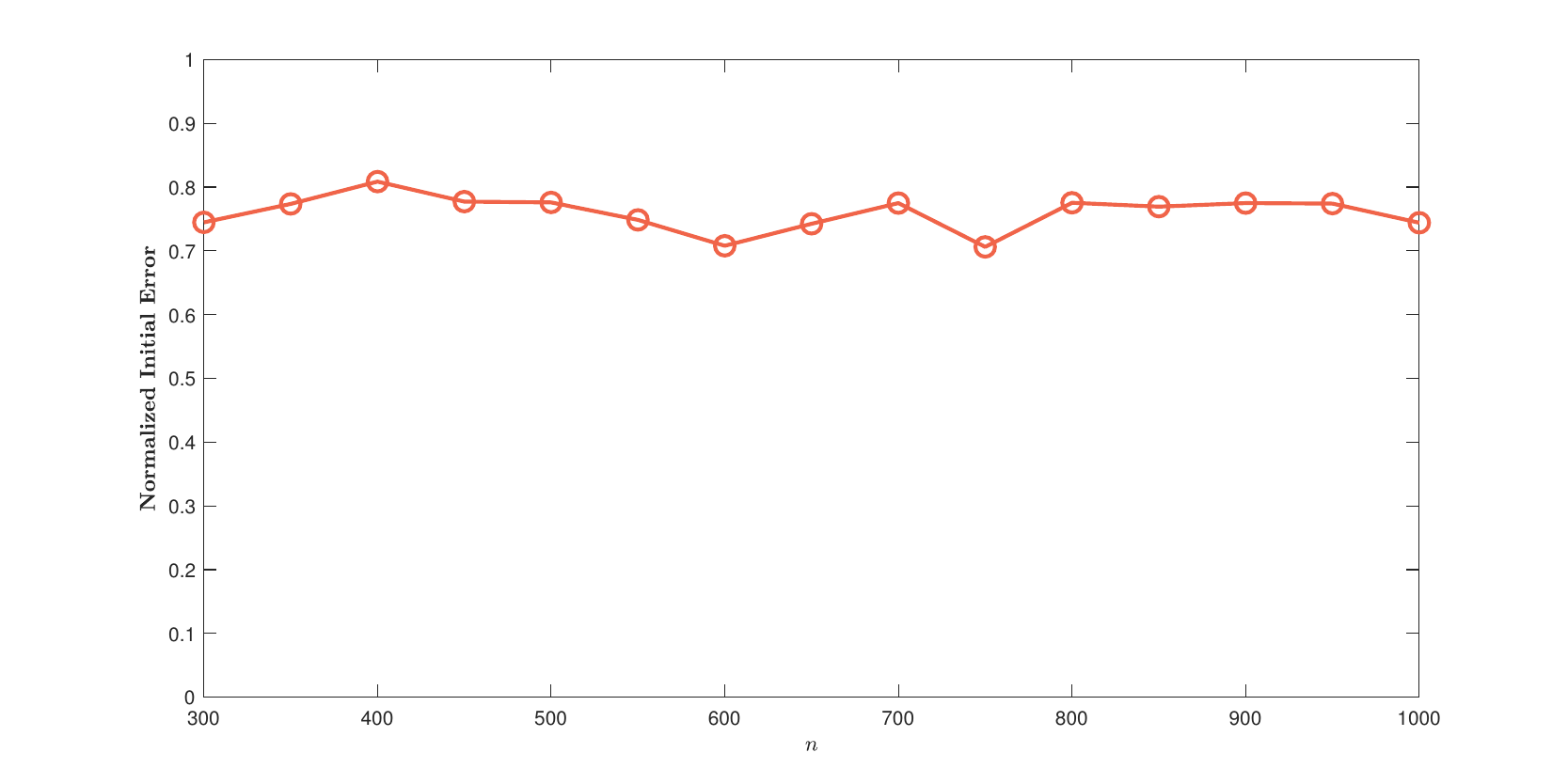} 
	\caption{The normalized initial error $\varepsilon^{r, 0}$ when $ \varepsilon^{r, \infty} \ge 1.02 \varepsilon^{0,\infty} $ happens.}
	\label{fig:init_err_1}
\end{figure}

\subsection{Estimation Error of GPM Initialized by Entropic Spectral Estimator}
Lastly, we empirically study how tightly the noise term $\nm{ \Pi^n \left( G^* + 2 D^{-1} \Delta G^* \right) - G^* }_F$ in Theorem~\ref{thm:master} characterizes the estimation error of GPM. 

Consider $\mathcal{SO}(3)$ under the same setting as that in Section~\ref{sec:SO(d)-setting}. We investigate how the estimation error of GPM (initialized by the entropic spectral estimator) relates to the term $\nm{ \Pi^n \left( G^* + 2 D^{-1} \Delta G^* \right) - G^* }_F$ by varying the noise parameter $\gamma$. Here, we recall that Section~\ref{sec:SO(d)-setting} makes use of the Langevin noise, which is parametrized by $\gamma$. 
The results are plotted in Figure~\ref{fig:est_err_tight_1}. From the figure, we see that the normalized estimation error of GPM coincides with $\nm{ \Pi^n \left( G^* + 2 D^{-1} \Delta G^* \right) - G^* }_F$ when the noise level is low, which corroborates Theorem~\ref{thm:master}. The former starts to deviate from the latter when $\gamma^{-1}$ becomes larger, which corresponds to a higher noise level. A natural guess for the cause of the deviation is that the requirement on $\nm{D^{-1} \Delta}$ in Theorem~\ref{thm:master} is more likely to be violated if $\gamma^{-1}$ becomes larger. Therefore, we also plotted this quantity in Figure~\ref{fig:est_err_tight_1}. If the guess is correct, according to Figure~\ref{fig:est_err_tight_1}, the bound on the term $\nm{D^{-1} \Delta}$ should be approximately $ 0.8$, instead of $\tfrac{1}{32}$ as in Theorem~\ref{thm:master}.
\begin{figure}[t!]
	\centering
	\begin{subfigure}{.48\textwidth}
		\centering
		\includegraphics[width=\linewidth]{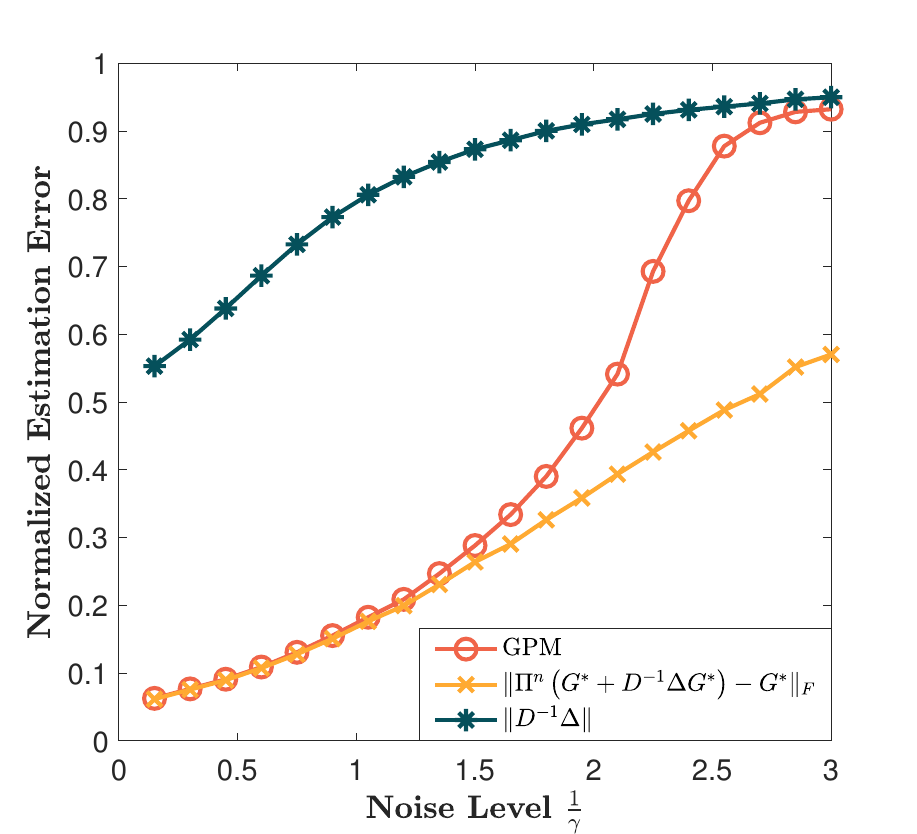}
	\caption{$n= 300, d = 3, p = 0.4, q = 0.5$}
\label{fig:est_err_tight_1_1} 
\end{subfigure}
	\begin{subfigure}{.48\textwidth}
		\centering		\includegraphics[width=\linewidth]{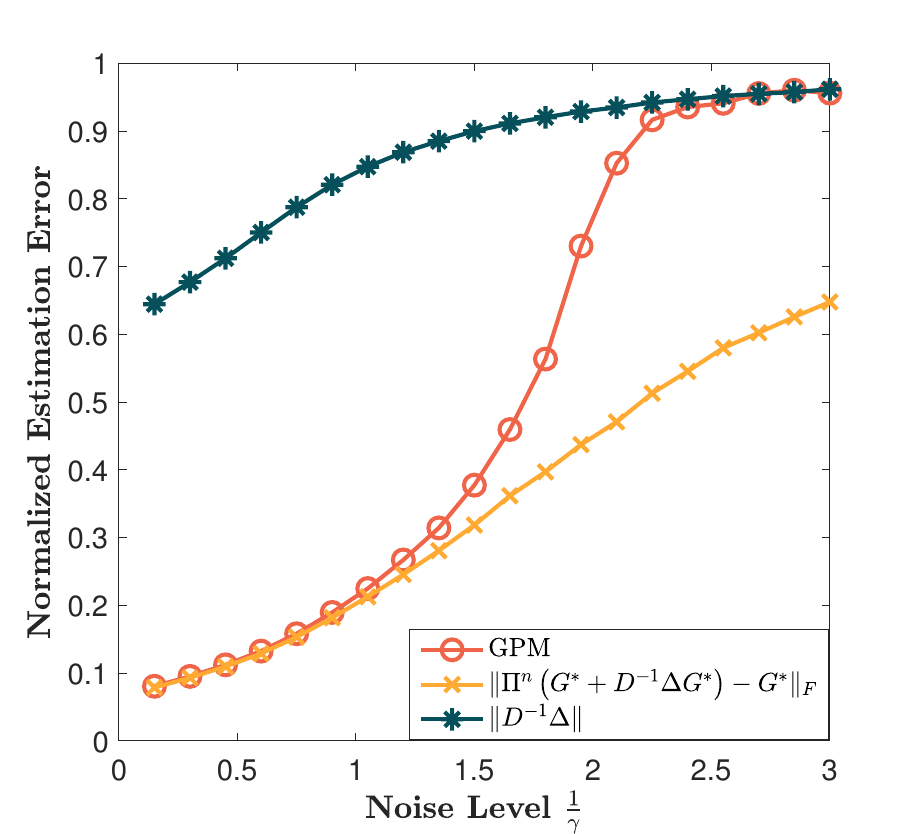} 
		\caption{$n= 400, d = 3, p = 0.4, q = 0.4$} \label{fig:est_err_tight_1_2}
	\end{subfigure}
	\caption{Normalized estimation error of GPM initialized by the entropic spectral estimator (labeled as GPM) and the normalized noise term $\frac{\nm{ \Pi^n \left( G^* + D^{-1} \Delta G^* \right) - G^* }_F }{ \sqrt{2nd}}$.}
	\label{fig:est_err_tight_1}
\end{figure}

Next, we conduct a similar experiment for $\mathcal{P}(d)$-\sync under the same setting as that in Section~\ref{sec:P(d)-setting}. We vary the standard deviation $\sigma$ of the additive Gaussian noise. Also, for the $y$-axis, we use the recovery rate as defined in~\eqref{eq:recovery_rate}. The results are plotted in Figure~\ref{fig:noise}. We observe a similar phenomenon as that in the experiment for $\mathcal{SO}(3)$. 

\begin{figure}[t!]
	\centering
	\begin{subfigure}{.49\textwidth}
		\centering
		\includegraphics[width=\linewidth]{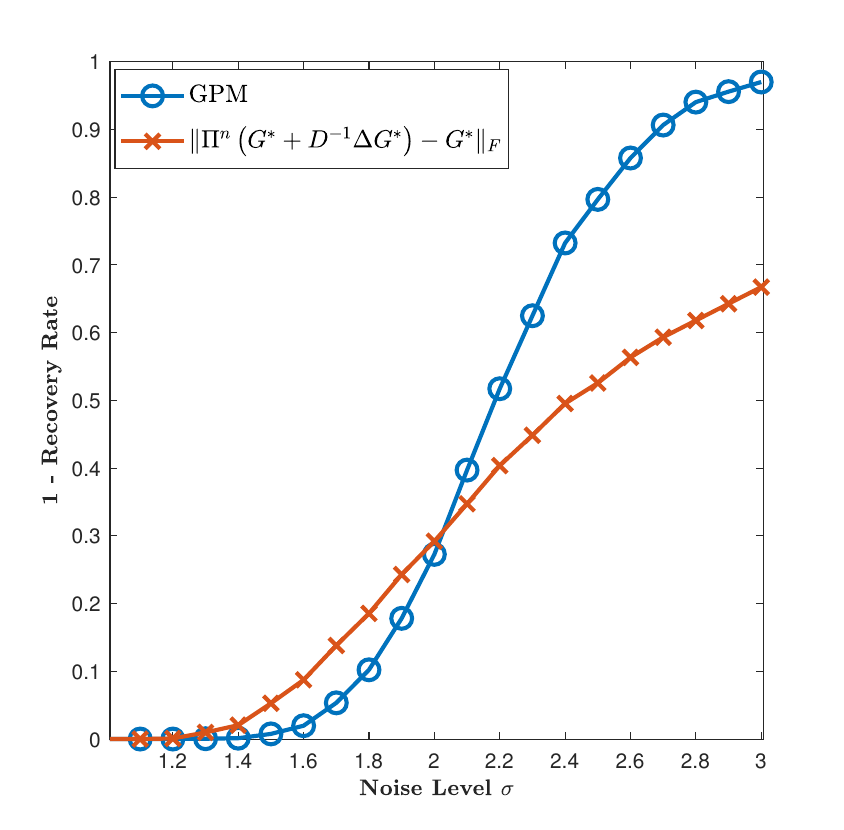}
	\caption{$n= 300, d = 12, p = 0.5, q = 0.7$}
\label{fig:noise_1} 
\end{subfigure}
	\hfill
	\begin{subfigure}{.48\textwidth}
		\centering		\includegraphics[width=\linewidth]{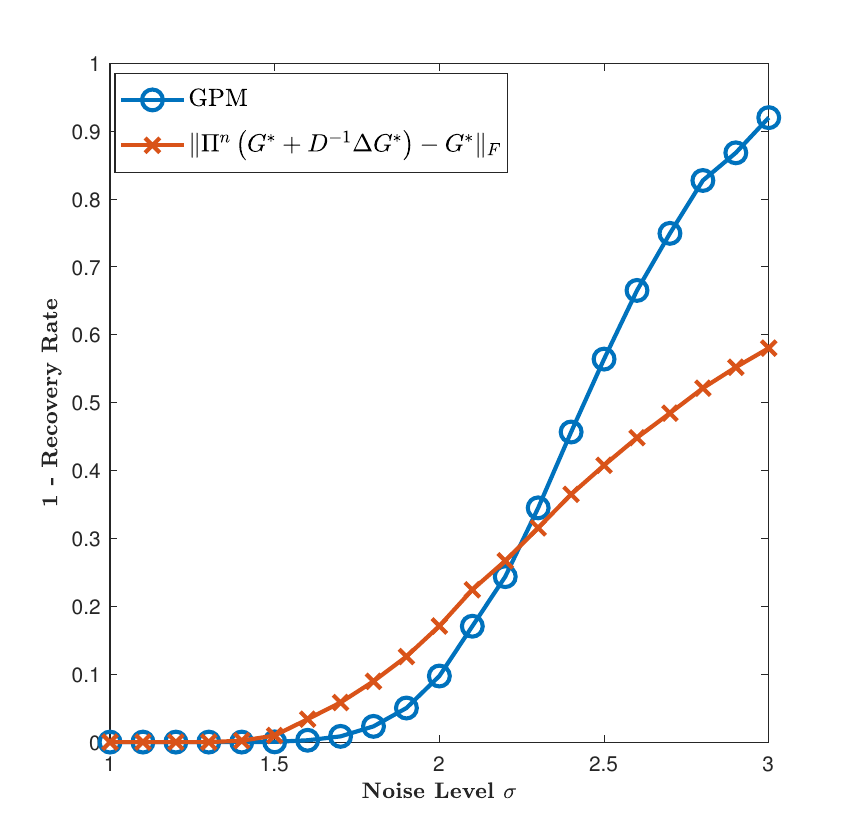} 
		\caption{$n= 400, d = 12, p = 0.5, q = 0.7$} \label{fig:noise_2}
	\end{subfigure}	
	\caption{Recovery error rate of GPM initialized by the entropic spectral estimator (labeled as GPM) and the normalized noise term $\frac{\nm{ \Pi^n \left( G^* + D^{-1} \Delta G^* \right) - G^* }_F }{ \sqrt{2nd}}$.}
	\label{fig:noise}
\end{figure}

The findings in the above two experiments suggest that the estimation error of GPM is fairly accurately predicted by Theorem~\ref{thm:master}, at least before the noise level reaches a certain threshold. 

\section{Conclusion}\label{sec:conclusion}
In this paper, we proposed a unified approach for tackling a class of synchronization problems over closed subgroups of the orthogonal group. The approach consists of a suitable initialization step and an iterative refinement step based on GPM. 
We then proved a master theorem, which shows that the estimation error of the iterates produced by GPM decreases geometrically under certain assumptions on the subgroup, measurement graph, noise, and initialization. We verified these assumptions for various practically relevant subgroups under standard random measurement graph and noise models. In the process, we formulated two conditions concerning the geometry of subgroups of the orthogonal group and developed a novel spectral-type estimator called the entropic spectral estimator based on the notion of metric entropy. These can be of independent interest. Our experiment results showed that the proposed approach outperforms existing approaches in terms of computational speed, scalability, and/or estimation error.

Besides Conjecture~\ref{conj:group}, there are two other research questions concerning non-convex approaches for solving group synchronization problems that are worth investigating. First, although the theory in Section~\ref{sec:noise} covers only additive noise models, our experiments showed that the proposed approach is also effective under various multiplicative noise models. Thus, it would be interesting to see whether our analysis can be extended to cover more general noise models. Second, an important example of group synchronization problems is $SE(d)$-\sync~\cite{rosen2016se}, where $SE(d)$ is the group of $d$-dimensional Euclidean motions. Such a problem arises in areas such as robotics and computer vision. The group $SE(d)$ is not a subgroup of the orthogonal group. It would be interesting to develop a non-convex approach similar to ours for solving $SE(d)$-\sync. 

\section*{Acknowledgment}
We thank Mihai Cucuringu, Yin-Tat Lee and Michael Kwok-Po Ng for helpful discussions and Amit Singer for kindly sharing with us their codes for the experiments. Man-Chung Yue is supported by the Hong Kong Research Grants Council (RGC) under the Early Career Scheme (ECS) project~25302420. Anthony Man-Cho So is supported in part by the Hong Kong Research Grants Council (RGC) General Research Fund (GRF) Project CUHK 14205421.

\bibliographystyle{abbrv}
\bibliography{refs}
\end{document}